%% file: main.tex
\documentclass[11pt]{article}

\usepackage{import}
\import{Packages/}{packages.tex}

\import{Packages/}{macros.tex}

\title{A graph-theoretical characterisation \\ of subgroups of Thompson's group $V$}

\author{Corentin Bodart\footnote{
University of Oxford, United Kingdom, \texttt{corentin.bodart@maths.ox.ac.uk}}\;,
Daniele D'Angeli\footnote{
Università Niccolò Cusano, Rome, Italy, \texttt{daniele.dangeli@unicusano.it}}\;, \\
Davide Perego\footnote{
Université de Genève, Switzerland, \texttt{davide.perego@unige.ch}}\;
and Emanuele Rodaro\footnote{
Politecnico di Milano, Italy, \texttt{emanuele.rodaro@polimi.it.}}}

\date{\today \vspace*{-3mm}}

\usepackage[backend=biber, style=numeric, sorting=nyt, 
    doi=false,isbn=false,url=false,eprint=false, maxnames=50]{biblatex}
\begin{document}

\maketitle

\begin{abstract}
We prove a graph-theoretical characterisation of finitely generated subgroups of Thompson's group $V$: a finitely generated group embeds in $V$ if and only if it admits a faithful \emph{context-free action}, or equivalently if it belongs to the class \textbf{CF-TR} of transition groups of context-free graphs recently introduced by Matucci and the three last authors. \medbreak

\noindent Using this characterisation, we prove results in different directions:
\medbreak

\begin{itemize}[leftmargin=5mm]
    \setlength\itemsep{1.5mm}
    \item All known examples of groups with co-context-free Word Problem \emph{do} embed in $V$, providing evidence towards Lehnert's conjecture.
    \item The following alternative holds: each finitely generated subgroup of $V$ is either virtually abelian, or contains a free non-abelian semigroup. It follows that groups of intermediate growth \emph{do not} embed in Thompson's $V$.
\end{itemize}
\medbreak

\noindent We further study the relation between transition groups defined by graphs that are limits or covers of each others, and prove properties of transition groups of context-free graphs of polynomial growth. Finally, we prove that the Basilica and Hanoï Towers groups 
\emph{do not} embed in $V$. This uses the geometry of Schreier graphs of the natural actions of these groups and of Thompson's $V$. \vspace*{2mm}
\end{abstract}

\newpage

\renewcommand*{\thethm}{\Alph{thm}}
\import{Sections/}{sec0_intro.tex}
\renewcommand*{\thethm}{\arabic{thm}}
\counterwithin{thm}{section}

\section{Background} \label{sec:intro}
\import{Sections/}{sec1_background.tex}

\section{Thompson's \texorpdfstring{$V$}{V} and CF-TR groups}
\import{Sections/}{sec2_equiv.tex}

\section{Cloning systems and FSS groups}
\import{Sections/}{sec3_Lehnert.tex}

\section{Taking covers} \label{sec:covers}
\import{Sections/}{sec4_covers.tex}

\section{Going to the limits} \label{sec:limits}
\import{Sections/}{sec5_limits.tex}

\section{Graphs of polynomial growth} \label{sec:alternative}
\import{Sections/}{sec6_poly.tex}

\section{An obstruction using Schreier graphs}
\import{Sections/}{sec7_branch.tex}

\emergencystretch=1em
\AtNextBibliography{\small}
\printbibliography

\end{document}

%% file: Packages/packages.tex
\usepackage{amsmath, amsfonts, amsthm, amssymb}
\usepackage{mathrsfs}
\usepackage{mathtools}
\usepackage{hyperref}
\usepackage[margin=3.2cm]{geometry}
\usepackage{graphicx}
\usepackage{easy-todo}
\usepackage{enumitem}
\usepackage[hypcap=false]{caption}
\usepackage{parskip}
\usepackage[noabbrev,nameinlink,capitalise]{cleveref}
\usepackage[dvipsnames]{xcolor}
\definecolor{color0}{RGB}{216,27,96} 
\definecolor{color1}{RGB}{255,193,7}
\definecolor{color2}{RGB}{30,136,229}    

\usepackage{tikz-cd}
\usepackage{tikz}
\usetikzlibrary{shapes, shapes.geometric, arrows, arrows.meta, decorations.markings, decorations.pathreplacing, automata, positioning}
\pgfdeclarelayer{background} 
\pgfdeclarelayer{foreground}
\pgfsetlayers{background,main,foreground}

\topsep=7pt
\theoremstyle{plain}
\newtheorem{thm}{Theorem} 
\newtheorem*{thm*}{Theorem}
\newtheorem{lemma}[thm]{Lemma}
\newtheorem*{lemma*}{Lemma}
\newtheorem{cor}[thm]{Corollary}
\newtheorem*{cor*}{Corollary}
\newtheorem{prop}[thm]{Proposition}
\newtheorem*{prop*}{Proposition}

\theoremstyle{definition}
\newtheorem{defi}[thm]{Definition}
\newtheorem*{defi*}{Definition}
\newtheorem{rem}[thm]{Remark}
\newtheorem{exa}[thm]{Example}
\newtheorem*{conj*}{Conjecture}
\newtheorem{ques}[thm]{Question}

%% file: Packages/macros.tex
\newcommand{\say}[1]{``#1"}

\DeclareMathOperator{\id}{id}

\newcommand{\Bsc}{\mathscr B} 

\newcommand{\EG}{\mathbf{EG}}
\newcommand{\EA}{\mathrm{EA}}
\newcommand{\Sp}{\mathrm{Sp}}

\newcommand{\Cay}{\mathcal Cay}
\newcommand{\Sch}{\mathcal Sch}

\DeclareMathOperator{\Stab}{Stab}
\DeclareMathOperator{\GStab}{GeSt}
\DeclareMathOperator{\RStab}{RiSt}
\DeclareMathOperator{\SStab}{SetSt}
\DeclareMathOperator{\End}{End}
\DeclareMathOperator{\Aut}{Aut}
\DeclareMathOperator{\QAut}{QAut}
\newcommand{\Sim}{\mathrm{Sim}}
\newcommand{\Sym}{\mathrm{Sym}}
\newcommand{\FSym}{\mathrm{FSym}}
\newcommand{\Isom}{\mathrm{Isom}}
\newcommand{\Homeo}{\mathrm{Homeo}}
\newcommand{\Sub}{\mathrm{Sub}}
\newcommand{\Gc}{\mathcal G} 
\newcommand{\Ic}{\mathcal I} 
\newcommand{\F}{\mathsf F}
\DeclareMathOperator{\Wr}{\wr\wr} 

\newcommand{\start}{\mathsf{start}}
\newcommand{\accept}{\mathsf{accept}}
\newcommand{\q}{\mathsf q}

\newcommand{\Ac}{\mathcal A}
\newcommand{\Bc}{\mathcal B}
\newcommand{\Dc}{\mathcal D}
\newcommand{\Ec}{\mathcal E}
\newcommand{\Fc}{\mathcal F}
\newcommand{\Lc}{\mathcal L}

\newcommand{\Vc}{\mathcal V}


\newcommand{\supp}{\mathrm{supp}}


\DeclareMathOperator{\ord}{ord}


\newcommand{\norm}[1]{\left\|#1\right\|}
\newcommand{\abs}[1]{\left|#1\right|}
\newcommand{\la}{\left\langle}
\newcommand{\ra}{\right\rangle}

\renewcommand{\ge}{\geqslant}
\renewcommand{\le}{\leqslant}

\newcommand{\longto}{\longrightarrow}
\newcommand{\onto}{\twoheadrightarrow}
\newcommand{\into}{\hookrightarrow}
\newcommand{\acts}{\curvearrowright}
\newcommand{\racts}{\curvearrowleft}

\newcommand{\N}{\mathbb N}
\newcommand{\Z}{\mathbb Z}

\newcommand{\R}{\mathbb R}
\newcommand{\Ck}{\mathfrak C}   
\newcommand{\D}{\mathbb D}


\definecolor{newmagenta}{RGB}{216,27,96} 
\definecolor{newyellow}{RGB}{255,193,7}
\definecolor{newblue}{RGB}{30,136,229}

%% file: Sections/sec0_intro.tex
The interaction between formal language theory and combinatorial group theory has been a central theme since the pioneering work of Higman and Boone. Higman's embedding theorem \cite{Higman1961} established that a group has recursively enumerable Word Problem if and only if it embeds into a finitely presented group, while Boone and Higman later formulated their conjecture asserting that a finitely generated group has recursive Word Problem if and only if it embeds into a finitely presented simple group \cite{BooneHigman1974} (see \cite{BH_survey} for recent developments).

At the other end of the spectrum, Anisimov \cite{anisimov} and Muller and Schupp \cite{Muller_Schupp} gave an algebraic characterisation of groups with Word Problem languages belonging to the first two classes of the Chomsky Hierarchy. Among different intermediate classes between context-free and context-sensitive, groups with co-context-free Word Problem, first studied by Holt, Rees, R\"over and Thomas \cite{HoltReesRoverThomas2005}, rose to prominence. An algebraic characterisation was conjectured by Lehnert \cite{lehnert} (with the current form due to Bleak, Matucci, and Neunh\"offer \cite{BleakMatucciNeunhoffer2016}):

\textbf{Conjecture.} A group is co-context-free if and only if it embeds in Thompson's group $V$.

The present work aims at better understanding this conjecture. Note that Lehnert and Schweitzer proved that Thompson's group $V$ has co-context-free Word Problem \cite{LehnertSchweitzer2007}, taking care of one implication. Also, note that this conjecture mirrors the Boone-Higman conjecture, owing to the fact that Thompson's $V$ is a finitely presented simple group.

\bigskip

One of our main results is the following:
\begin{thm} \label{thm:intro_main}
	Let $G$ be a finitely generated group. The following are equivalent:
	\begin{itemize}[leftmargin=6mm]
		\item $G$ embeds in Thompson's group $V$, and
		\item $G$ is \textbf{CF-TR}, i.e.\ $G$ admits a faithful context-free action.
	\end{itemize}
\end{thm}

The result is built on the new graph-theoretical perspective developed by Matucci and the last three authors \cite{CFTR}. While constituting a natural subclass of co-context-free groups, \textbf{CF-TR} displays a rich combinatorial structure that opens avenues for deeper analysis.

\bigskip

The first main application of Theorem \ref{thm:intro_main} is a new formulation of Lehnert's conjecture.

\textbf{Conjecture*.} A group is co-context-free if and only if it is \textbf{CF-TR}.

Note that $V$ does not appear anywhere in this formulation, and both sides of the conjecture involve some context-free-ness. This makes the conjecture more believable. As further evidence, we observe that all known proofs that groups have co-context-free Word Problem \say{translate} (with some efforts) into proofs that these groups are \textbf{CF-TR}, hence embed in Thompson's $V$. More precisely, we prove the following embedding results:
\begin{thm} \label{thm:intro_coCF}
	The following groups embed in Thompson's group $V$:
	\begin{itemize}[leftmargin=6mm]
		\item $V_{(H,\theta)}$ with $H$ finite and $\theta\in\End(H)$,  introduced in \cite{ThompsonV_clones,cloning}.
		\item \emph{FSS} groups $G(\Sym_X)$ with $\Sym_X$ a finite similarity structure, introduced in \cite{Hughes,Farley_FFS}.
	\end{itemize}
\end{thm}
The first part improves on partial results for $V_{(H,\theta)}$ when $\theta=\id_H$ \cite{ThompsonV_clones}, and $F_{(H,\theta)}$ with $H$ abelian and $\theta$ idempotent \cite{cloning_explained}. This also makes significant progress towards answering Question 2.16 of Zaremsky \cite{Zaremsky_Q}. The second part includes a family of Röver-Nekrashevych $V_d(H)$ with $H\le\Sym(d)$ acting self-similarly on $\Ck_d$ considered in \cite{aroca2022some}, and answer their Question 1.3 in a strong form. Combined, Theorems \ref{thm:intro_main} and \ref{thm:intro_coCF} prove that \emph{all} groups that were previously known to be co-context-free \emph{do} embed in Thompson's $V$.

\bigbreak

The second direction where Theorem \ref{thm:intro_main} can be useful is to find \emph{obstructions} to embedding in Thompson's $V$. A handful of obstructions are given in the literature, see the survey \cite{Burillo_Cleary_Rover_2017}. We note that some of them have been proven in parallel for \textbf{CF-TR} groups (sometimes with more transparent proofs):
\begin{itemize}[leftmargin=6mm]
	\item Röver proved that finitely generated torsion subgroups of $V$ are finite \cite[Theorem 3]{Rover}. An analogous result is proven for \textbf{CF-TR} groups \cite[Corollary 5.8]{CFTR}.
	\item The period growth function $p_G(n)=\max\bigl\{\ord(g) : \norm g\le n \text{ and }\ord(g)<\infty\bigr\}$ satisfies $p_G(n)\preceq \exp(n^2)$ for \textbf{CF-TR} groups \cite[Proposition 5.4]{CFTR} and for finitely generated subgroups of Thompson's $V$ \cite[Theorem 6.2]{Period}.
\end{itemize}
In Section \ref*{sec:alternative}, we use the characterisation with \textbf{CF-TR} groups to provide a \emph{new} obstruction, which presents itself as an alternative:
\begin{thm} \label{thm:alt_intro}
    Let $G$ be a finitely generated subgroup of $V$. Then either
    \begin{itemize}[leftmargin=6mm]
        \item $G$ is virtually abelian, or
        \item $G$ contains a non-abelian free semigroup.
    \end{itemize}
\end{thm}
As an immediate corollary, we get the following result:
\begin{cor}
    Groups of intermediate growth \emph{do not} embed in $V$. \vspace*{-5pt}
\end{cor}
This contrasts with a series of results embedding groups of intermediate growth in topological full groups of (small) subshifts \cite{MatteBon_intermediate, GLN_Lysenok_subshift, Grigorchuk_ThueMorse_subshift}. In order to prove Theorem \ref*{thm:alt_intro}, we develop a machinery to compare transition groups of a (context-free) graph with transition groups of covers and limits, which could have wider applications (Sections \ref*{sec:covers} and \ref*{sec:limits}).

\bigbreak

In the last part of the paper, we prove that various groups of dynamical origin \emph{do not} embed inside $V$, by comparing the Schreier graphs of their natural actions (on the boundary $\partial T$
) with Schreier graphs of $V$ acting on $\Ck$. %
\begin{thm}
    The following groups do not embed in Thompson's $V$:
    \begin{itemize}[leftmargin=6mm]
        \item The Basilica group, introduced in \cite{Basilica},
        \item The Hanoï Towers groups, introduced in \cite{Hanoi_Schreier},
    \end{itemize}
\end{thm}
The proofs rely on a new lemma (Lemma \ref{lem:free_in_cover}) leveraging coverings from quasi-tree Schreier graphs (such as $\Sch(\xi, G;\Ac)$ for $G\le V$ and $\xi\in\Ck$) onto non-quasi-trees, combined with a \say{rigidity} result of Le Boudec and Matte Bon for weakly branch groups \cite{Branch_confined}. 

\medskip

\textbf{Important remark.} We completed our proof of Theorem \ref*{thm:intro_main} in March-April 2026, then learned in May about the existence of an independent proof by Henry Jaspars \cite{Jaspars}. He completed his proof around January 2026.

\smallskip

\textbf{Acknowledgements.} The authors are grateful to Dominik Francoeur for the proof of Lemma \ref{prop:Dominik_Basilica}, and to Jim Belk, Francesco Matucci, Nicolás Matte Bon and Tatiana Nagnibeda for discussions and encouragements at different stages of the project. DD and ER are members of the Gruppo Nazionale per le Strutture Algebriche, Geometriche e le loro Applicazioni (GNSAGA) of the Istituto Nazionale di Alta Matematica (INdAM). CB and DP thank Dipartimento di Matematica of Politecnico di Milano and of Università di Milano-Bicocca for the kind hospitality.

\smallskip

\textbf{Funding.} CB is funded by the Swiss NSF grant P500PT-225420. DP was supported by the Swiss NSF grant 200020-200400 and the research project PID2022-138719NAI00 financed by the Spanish Ministry of Science and Innovation.

%% file: Sections/sec1_background.tex
\subsection{Thompson's group \texorpdfstring{$V$}{V}}

Thompson's group $V$ is the group of homeomorphisms of the Cantor set that can be described by prefix replacement rules. Formally, given two partitions
$$\Ck = [a_1]\sqcup \ldots \sqcup [a_m] = [b_1]\sqcup \ldots \sqcup [b_m]$$
of the Cantor set $\Ck=\{0,1\}^\infty$, where $[a]:=\{a\eta:\eta\in\{0,1\}^\infty\}$ is the cone below $a$, we define an homeomorphism
\[ g \colon a_j\hspace{1pt}\eta\mapsto b_j\hspace{1pt}\eta \qquad \bigl(\eta\in\{0,1\}^\infty,\ j=1,\ldots, m \bigr). \]
The set of elements of this form is a group, called Thompson's group $V$.

\begin{center}
\begin{tikzpicture}[xscale = 7, yscale=2.1]
	\begin{scope}[red, thick]
		\draw (0,0) -- (.0256,0);
		\draw (.0384,0) -- (.064,0);
		\draw (.096,0) -- (.1216,0);
		\draw (.1344,0) -- (.16,0);
	\end{scope}
	
	\begin{scope}[blue, thick]
		\draw (.24,0) -- (.2656,0);
		\draw (.2784,0) -- (.304,0);
		\draw (.336,0) -- (.3616,0);
		\draw (.3744,0) -- (.4,0);
	\end{scope}
	
	\begin{scope}[green, thick]
		\draw (.6,0) -- (.6256,0);
		\draw (.6384,0) -- (.664,0);
		\draw (.696,0) -- (.7216,0);
		\draw (.7344,0) -- (.76,0);
		
		\draw (.84,0) -- (.8656,0);
		\draw (.8784,0) -- (.904,0);
		\draw (.936,0) -- (.9616,0);
		\draw (.9744,0) -- (1,0);		
	\end{scope}
	
	
	\begin{scope}[blue, thick, shift={(0,-1)}]
		\draw (0,0) -- (.0256,0);
		\draw (.0384,0) -- (.064,0);
		\draw (.096,0) -- (.1216,0);
		\draw (.1344,0) -- (.16,0);
		
		\draw (.24,0) -- (.2656,0);
		\draw (.2784,0) -- (.304,0);
		\draw (.336,0) -- (.3616,0);
		\draw (.3744,0) -- (.4,0);
	\end{scope}
	
	\begin{scope}[red, thick, shift={(0,-1)}]
		\draw (.6,0) -- (.6256,0);
		\draw (.6384,0) -- (.664,0);
		\draw (.696,0) -- (.7216,0);
		\draw (.7344,0) -- (.76,0);
	\end{scope}
	
	\begin{scope}[green, thick, shift={(0,-1)}]
		\draw (.84,0) -- (.8656,0);
		\draw (.8784,0) -- (.904,0);
		\draw (.936,0) -- (.9616,0);
		\draw (.9744,0) -- (1,0);		
	\end{scope}
	
	\begin{scope}
		\clip (0,-.95) rectangle (1,-.05);
		\draw[very thin, red, fill=red!20] (0,0) to[out=-90, in=90] (.6,-1) -- (.76,-1) to[out=90, in=-90] (.16,0);
		\draw[very thin, blue, fill=blue!40, opacity=.5] (.24,0) to[out=-90, in=90] (0,-1) -- (.4,-1) to[out=90, in=-90] (.4,0);
		\draw[very thin, green, fill=green!20] (.6,0) to[out=-90, in=90] (.84,-1) -- (1,-1) to[out=90, in=-90] (1,0);
	\end{scope}
\end{tikzpicture}
\captionsetup{font=small}

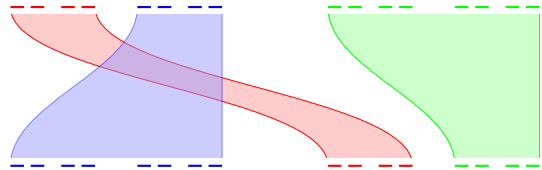
\captionof{figure}{An element of $V$, defined by $(a_j,b_j)=(00,10),(01,0),(1,11)$}
\end{center}

\subsection{Inverse graphs and their transition groups}

We recall the definition of (labelled) (di-)graphs, in the sense of Serre:
\begin{defi} An \emph{inverse $\Ac$-graph} is a tuple $\Gamma=(V, E, \iota, \tau, \lambda, \,\cdot^{-1})$ where
\begin{itemize}[leftmargin=6mm]
    \item $V$ is the vertex set and $E$ is the edge set,
    \item $\iota\colon E\to V$ (resp.\ $\tau$) maps each edge to its initial (resp.\ terminal) vertex,
    \item $\lambda\colon E\to \Ac$ is a labelling map,
    \item $\cdot^{-1}\colon E\to E$ is an involution satisfying
    \[ e^{-1}\ne e, \quad \iota(e^{-1})=\tau(e),\quad \tau(e^{-1})=\iota(e) \quad\text{and}\quad \lambda(e^{-1})=\lambda(e)^{-1}. \]
\end{itemize}
We will always assume that our graphs are
\begin{itemize}[leftmargin=6mm]
    \item \emph{deterministic}: for all $v\in V$ and $a\in\Ac$, there exists at most one $e\in E$ such that $\iota(e)=v$ and $\lambda(e)=a$, and 
    \item \emph{co-deterministic}: for all $v\in V$ and $a\in\Ac$, there exists at most one $e\in E$ such that $\tau(e)=v$ and $\lambda(e)=a$.
\end{itemize}
(Note that deterministic inverse graphs are also co-deterministic. However, we prefer to spell it out.) Except in \S\ref{sec:intro} and \S\ref{sec:limits_CF}, we will also assume
\begin{itemize}[leftmargin=6mm]
    \item \emph{completeness}: for all $v\in V$ and $a\in\Ac$, there exists at least one $e\in E$ such that $\iota(e)=v$ and $\lambda(e)=a$.
\end{itemize}
\end{defi}
Note also that, while in the literature an inverse graph is technically assumed to be connected, we do not make this implicit assumption. Instead, we will explicitly use the term ``connected" when we wish to emphasize this property.

\medskip

\begin{exa}
    Consider a group $G=\la \Ac\ra$ and a right action $X\racts G$. We define the \emph{Schreier graph} $\Sch(X,G;\Ac)$ of this action as the graph with vertex set $V=X$, and edges $x\overset{a}\longto xa$ for $a\in\Ac$ and $x\in X$. This graph is a complete inverse $\Ac$-graph. Often, we will want to restrict the action to a single orbit $\xi\cdot G$ for some $\xi\in X$, in which case we will simply write $\Sch(\xi,G;\Ac) \coloneqq \Sch(\xi\cdot G,G;\Ac)$. By construction, this graph is connected.
\end{exa}

\medskip

Schreier graphs of group actions are the prototypical example of complete inverse graphs. This is formalised by the following construction:

\begin{defi} Given $\Gamma$ a complete inverse $\Ac$-graph, its \emph{transition group} is
\[ \Gc(\Gamma) \coloneqq \la \Ac\ra \le \Sym(V\Gamma) \]
generated by the functions $a\colon x\mapsto xa$.
\end{defi}
These functions $a\colon X\to X$ are well-defined permutations of $X$ using that $\Gamma$ is complete, deterministic and co-deterministic. Moreover, $$\Sch(V\Gamma, \Gc(\Gamma);\Ac) \simeq\Gamma,$$ hence all complete inverse $\Ac$-graphs are Schreier graphs. Since the action is faithful, $\Gc(\Gamma)$ is the \say{smallest group} with this Schreier graph.

\newpage

Equivalently, transitions groups can be defined as quotients:
\begin{defi}
    Given a inverse $\Ac$-graph $\Gamma$ and $x\in V\Gamma$, we define
    \begin{align*}
        L(\Gamma,x) & \coloneqq \bigl\{ w\in\Ac^* \;\big|\; w\text{ labels a cycle } x\overset{w}\longto x \bigr\}, \\
        L(\Gamma) & \coloneqq \bigcap_{x\in V\Gamma} L(\Gamma,x).
    \end{align*}
\end{defi}
The language $L(\Gamma)$ is a subset of $\Ac^*$. When $\Gamma$ is complete, the map $\Ac^*\onto F_\Ac$ identifies it with the subgroup $\ker\bigl(F_\Ac\onto\Gc(\Gamma)\bigr)$, so that $\Gc(\Gamma)\simeq F_\Ac/L(\Gamma)$.

\subsection{Context-free graphs}

The article heavily depends on the notion of context-free languages. In some rare occasions (specifically in \S\ref{sec:cover_CF}), we will use the formalism of context-free grammars, see \cite{S2013} for background. More importantly, we will use the notion of pushdown automata, which we quickly recall. These are abstract machines which recognise context-free languages \cite{S2013}. 

\begin{defi}
A \emph{non-deterministic pushdown automaton (PDA)} is defined as a tuple
$(\mathsf{Q},\Ac , \mathcal{S}, \delta, \start, \#, \accept)$ where
\begin{itemize}[leftmargin=6mm]
    \item $\mathsf{Q}$ is a finite, non-empty set of \emph{states},
    \item $\Ac$ is a finite set of symbols constituting the \emph{alphabet},
    \item $\mathcal{S}$ is a finite set of symbols constituting the \emph{stack alphabet},
    \item $\delta \subset \mathsf{Q} \times (\Ac \cup \{\varepsilon\}) \times \mathcal{S}^* \times \mathcal{S}^* \times \mathsf{Q}$ is a finite set called \emph{transition function},
    \item $\start \in \mathsf{Q}$ is the \emph{start state},
    \item $\# \in \mathcal{S}$ is the \emph{start stack symbol}, and
    \item $\accept \in \mathsf{Q}$ is the \emph{accepting state}. 
\end{itemize}

An element $(\q_1, a, u, v, \q_2) \in \delta$ is understood as a transition from state $\q_1$ to state $\q_2$, where we read $a$, pop $u$ from the stack, and push $v$ onto the stack.
\end{defi}
\begin{rem}
Our definition of pushdown automaton is slightly more general than the
standard one, allowing transitions to inspect and replace a finite word on
top of the stack. This does not increase the expressive power, since it is
equivalent to the usual model of pushdown automata. We adopt this formalism
because it leads to substantially shorter descriptions of the automata
constructed later in the paper.
\end{rem}

\medskip

\begin{defi}
    A \emph{context-free graph} is a connected inverse $\Ac$-graph $\Gamma$ such that $L(\Gamma,x)$ is context-free for all (equivalently, for one) $x\in V\Gamma$.
\end{defi}
\begin{defi}
    Let $G=\la\Ac\ra$ be a finitely generated group.
    \begin{itemize}[leftmargin=6mm]
        \item An action $X\racts G$ is \emph{context-free} if it has finitely many orbits, and all orbital Schreier graphs $\Sch(\xi,G;\Ac)$ with $\xi\in X$ are context-free.
    \end{itemize}
    The group $G$ is \textbf{CF-TR} if the following equivalent conditions holds:
    \begin{itemize}[leftmargin=6mm]
        \item There exist complete context-free $\Ac$-graphs $\Gamma^{(1)},\ldots,\Gamma^{(k)}$ such that
        \[ G\simeq\Gc(\Gamma^{(1)}\sqcup \ldots\sqcup\Gamma^{(k)}). \vspace*{-2mm} \]
        \item $G$ admits a \emph{faithful} context-free action $X\racts G$.
    \end{itemize}
\end{defi}

This characterisation of \textbf{CF-TR} groups is sufficient to prove that many groups such as $V$ and $V_{(H,\theta)}$ are \textbf{CF-TR}. However, what makes this theory most fruitful is the equivalent geometric characterisation of context-free graphs, originating from seminal work of Muller and Schupp \cite{Muller_Schupp} (see also \cite{CF_pairs}, and \cite{Rodaro} for generalisation to the non-complete case).

Given an inverse graph $\Gamma$ and a fixed basepoint $x_0\in V\Gamma$, the \emph{end-cone} at $v\in V\Gamma$ is the connected component of $\Gamma \setminus D_{n-1}(x_0)$ containing $v$, where 
\[
D_n(x_0) = \bigl\{u \in V\Gamma \mid d(u,{x_0}) \le n \bigr\}
\]
and $n=d(v,{x_0})$. The end-cone is denoted $\Gamma(v, x_0)$. The \emph{set of frontier vertices} of $\Gamma(v, x_0)$ is 
$\Delta(v,x_0)=\Gamma(v, x_0) \cap D_n(x_0)$.

Two end-cones $\Gamma(v_1, x_0)$ and $\Gamma(v_2, x_0)$ are \textit{end-isomorphic} if there is a label- and direction-preserving graph isomorphism $\psi\colon \Gamma(v_1, x_0) \to \Gamma(v_2, x_0)$ mapping $\Delta(v_1, x_0)$ bijectively onto $\Delta(v_2, x_0)$. Equivalence classes of end-cones are \emph{end-cone types}.

\begin{thm}[Muller--Schupp \cite{Muller_Schupp}, Ceccherini-Silberstein--Woess \cite{CF_pairs}]
An inverse graph $\Gamma$ is context-free if and only if it has finitely many end-cone types $\{ \Gamma_0, \Gamma_1,\ldots, \Gamma_N\bigr\}$. This property does not depend on the basepoint $x_0$.
\end{thm}

Given $u,v\in V\Gamma$, we say that $\Gamma(u,x_0)$ is \emph{a second-level end-cone} of $\Gamma(v,x_0)$ if $\Gamma(u,x_0)\subsetneq\Gamma(v,x_0)$ and $d(u,x_0)=d(v,x_0)+1$. The number and types of second-level end-cones of an end-cone depends only on its type, therefore we can number the second-level end-cones of $\Gamma_i$ as $\Gamma_{1}^{(i)}$, $\Gamma_2^{(i)}$, ..., $\Gamma_{m_i}^{(i)}$. The \textit{end-cone type graph} is a finite directed graph $\mathcal{D}(\Gamma)$ where
\begin{itemize}[leftmargin=6mm]
    \item the set of vertices is the set of end-cone types, and
    \item the set of edges is $\bigl\{\Gamma_j^{(i)} \,\big|\, i\in\{0,\ldots, N\}, j\in\{1,\ldots,m_i\}\bigr\}$, with each edge $\Gamma_j^{(i)}$ going from $\Gamma_i$ to the type of $\Gamma_j^{(i)}$.
\end{itemize}
The regular language $\Lc$ recognised by this graph (viewed as an automaton, with starting vertex $\Gamma_0$) is the language of \emph{end-geodesics}, which are naturally associated with end-cones in $\Gamma$ (see \cite[Lemma 7.2]{CFTR}, modulo the first $\Gamma_0$ letter). Since $\Lc$ is prefix-closed, it can be turned into a tree using the edit distance
\[ d(w_1,w_2) = \min\bigl\{\abs{v_1}+\abs{v_2} : \exists u,\; w_i=uv_i \bigr\}\]
The map $f\colon \Gamma\to \Lc$ sending $v\mapsto \Gamma(v,x_0)$ satisfies
\[
\forall u,v\in V\Gamma,\quad d(u,v)-C\le d(f(u),f(v))\le d(u,v)
\]
(even $d(f(v),\Gamma_0)=d(v,x_0)$), and the fibres are the sets of frontier vertices $\Delta(v,x_0)$, hence have bounded size. It follows that the growth functions satisfies the following inequalities:
\[ \beta_{(\Lc,\varepsilon)}(r) \le \beta_{(\Gamma,x_0)}(r) \le \max_i\abs{\Delta_i}\cdot \beta_{(\Lc,\varepsilon)}(r).\]
Therefore, the growth type of $\Gamma$ can be recovered from the growth type of the regular language $\Lc$, which is well-understood, see for instance \cite{Tits_for_languages}.
\begin{lemma} \label{lem:growth_of_CF_graphs}
Let $\Gamma$ be a context-free graph, and $x_0\in V\Gamma$ a fixed basepoint.
\begin{itemize}[leftmargin=6mm]
    \item If $\mathcal D(\Gamma)$ contains two co-accessible cycles, then $\beta_{(\Gamma,x_0)}(r)\asymp 2^r$.
    \item Otherwise, if $\mathcal D(\Gamma)$ contains a sequence of $d$ cycles, each one accessible from the previous one, and $d$ is maximal, then $\beta_{(\Gamma,x_0)}(r)\asymp r^d$.
\end{itemize}
The same dichotomy holds for $\bar\beta_{\Gamma}(r)=\max_{x\in\Gamma}\# D_r(x)$.
\end{lemma}
\begin{center}
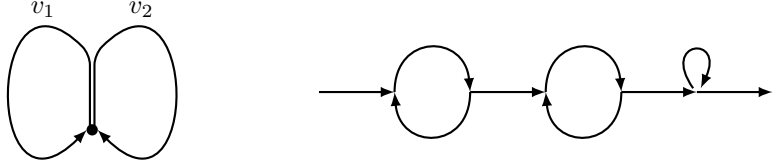

    \begin{tikzpicture}
        \clip (-1.5,-.8) rectangle (9.3,2);
        \node[circle, fill=black, inner sep=1.5pt] at (0,0) {};
        \draw[thick, rounded corners, -latex] (.03,0) --
            (.03,1) to[out=45, in=-45, looseness=5] (.07,0);
        \draw[thick, rounded corners, -latex] (-.03,0) --
            (-.03,1) to[out=135, in=-135, looseness=5] (-.07,0);
        \node at (-.65,1.6) {\small$v_1$};
        \node at (.65,1.6) {\small$v_2$};
        
        \begin{scope}[thick, -latex, shift={(3,.5)}]
        \draw (0,0) -- (1,0);
        \draw[out=90, in=95, looseness=2] (1,0) to (2,0);
        \draw[out=-90, in=-85, looseness=2] (2,0) to (1,0);
        \draw (2,0) -- (3,0);
        \draw[out=90, in=95, looseness=2] (3,0) to (4,0);
        \draw[out=-90, in=-85, looseness=2] (4,0) to (3,0);
        \draw (4,0) -- (5,0);
        \draw[out=125, in=65, looseness=20] (4.95,.05) to (5.05,.05);
        \draw (5,0) -- (6,0);
        \end{scope}
    \end{tikzpicture}
    \captionsetup{margin=2mm, font=small}
    \captionof{figure}{The two types of subgraphs of $\mathcal D(\Gamma)$ witnessing growth of $\Gamma$.}
\end{center}
\begin{proof}
    It only remains to justify $\bar\beta_\Gamma(r)\preceq r^d$ in the second case. Fix $r\ge 0$ and a vertex $x\in \Gamma$, and let $w=\Gamma(x,x_0)\in\Lc$ be the associated end-cone. Let $u$ be the shortest word in the ball $D^{\Lc}_r(w)$, of type $\Gamma_j$. Then
    \[
    f(D_r^\Gamma(x))
    \;\subseteq\; D_r^{\Lc}(w) 
    \;\subseteq \;D_{2r}^{\Lc}(u) \cap \{uv\in\Lc\} 
    \;=\; u\,\{v\text{ labels a path from }\Gamma_j: \abs{v}\le 2r\},
    \]
    hence $\#D^\Gamma_r(x)\le \max_i\abs{\Delta_i}\cdot \beta_{(\Lc_j,\varepsilon)}(2r)\preceq (2r)^d$, where $\Lc_j$ is one of finitely many languages $\{v\text{ labels a path from }\Gamma_j\}$. The constants can be made uniform in $r\ge 0$ and $x\in\Gamma$, since there are only finitely many end-cone types $\Gamma_j$.
\end{proof}
The first condition can be restated as \say{there exist $y,z_1,z_2\in V\Gamma$ such that
\[
\Gamma(y,x_0)\supset\Gamma(z_1,x_0),\Gamma(z_2,x_0), \quad\text{and}\quad
\Gamma(z_1,x_0)\cap \Gamma(z_2,x_0)=\emptyset,
\]
and the three end-cones have the same type}. Indeed, these two cycles give end-geodesics $u, uv_1, uv_2\in\Lc$ ending at the same vertex $\Gamma_i\in\mathcal D(\Gamma)$, with $v_1,v_2$ labelling distinct cycles so that neither is a prefix of the other. We recall that an end-geodesic is a prefix of another if and only if the associated end-cone contains the other. It follows that the end-cones associated to $u,uv_1,uv_2$ are of the same type $\Gamma_i$ and satisfies the incident conditions given.

On the other side, a typical example of context-free graph of polynomial growth is given on Figure \ref{fig:polygamma}. One can verify that the growth is quadratic, matching what is expected from the graph of end-cone types $\mathcal D(\Gamma)$.
\begin{center}
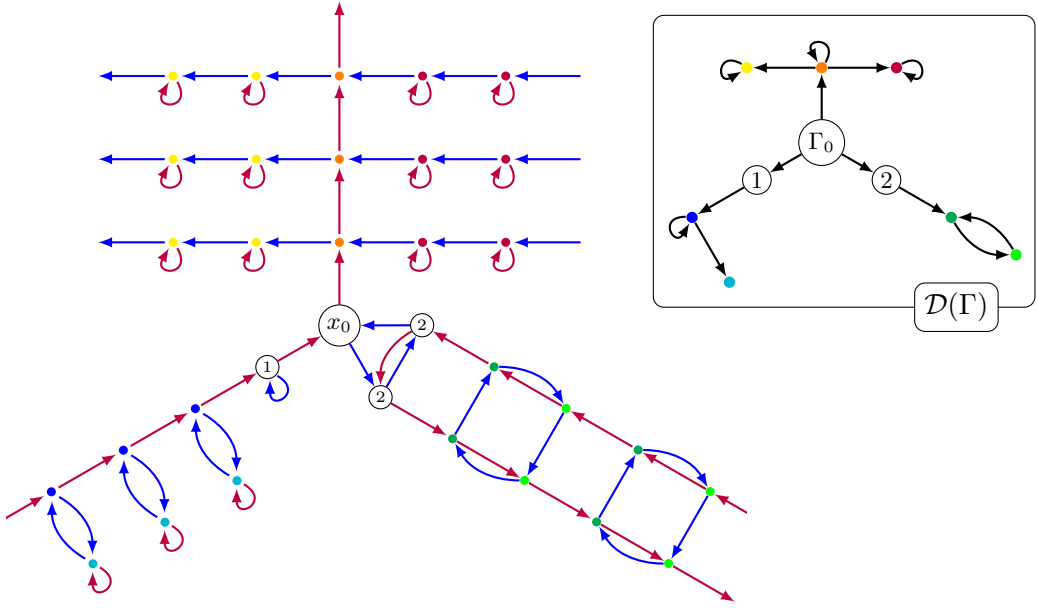

    \import{Pictures/}{tikz_polygamma.tex}
    \captionsetup{font=small}
    \captionof{figure}{A context-free graph $\Gamma$, and the associated end-cone type graph $\mathcal D(\Gamma)$.} 
    \label{fig:polygamma}
\end{center}

\subsection{Coarse geometry}

We recall the notions of $K$-fat and asymptotic minors, which were introduced in \cite{Asymptotic_minors_intro}, see also \cite{Asymptotic_minors_survey} for a more recent survey.
\begin{defi}
    Let $\Gamma$ and $\Sigma=(\Vc,\Ec)$ be two graphs, and $K\ge 1$. A \emph{$K$-fat $\Sigma$-minor} in $\Gamma$ is a tuple of connected subsets $S_p\subseteq\Gamma$ ($p\in\Vc\sqcup\Ec$) such that
    \[ d_{p,q} \coloneqq \min\bigl\{d(x,y) \;\big|\; x\in S_p,\ y\in S_q\bigr\} \]
    satisfies $d_{p,q}=0$ if $p,q$ are incident (i.e. $S_p\cap S_q\ne\emptyset$ if $p,q$ are an edge and one of its extremities), and $d_{p,q}\ge K$ otherwise. We say that $\Gamma$ contains an \emph{asymptotic $\Sigma$-minor} if it contains $K$-fat $\Sigma$-minors for all $K\ge 1$.
\end{defi}
\begin{rem}
    We may always assume that $S_e$ is a path for $e\in\Ec$.
\end{rem}
In particular, we will care about $\Sigma$-minors for the following two graphs:
\begin{center}
    \begin{tikzpicture}
        \node at (-.85,1.02) {$R_2=$};
        \node[circle, fill=black, inner sep=2pt] (0) at (0,0) {};
        \node[circle, fill=black, inner sep=2pt] (1) at (0,1) {};
        \node[circle, fill=black, inner sep=2pt] (2) at (0,2) {};
        \draw[thick, bend right] (0) to (1);
        \draw[thick, bend left] (0) to (1);
        \draw[thick, bend right] (1) to (2);
        \draw[thick, bend left] (1) to (2);

        \begin{scope}[shift={(5,0)}]
        \node at (-1.9,1.01) {$K_{2,3}=$};
        \node[circle, fill=black, inner sep=2pt] (0) at (0,0) {};
        \node[circle, fill=black, inner sep=2pt] (1) at (0,1) {};
        \node[circle, fill=black, inner sep=2pt] (2) at (0,2) {};
        \node[circle, fill=black, inner sep=2pt] (l) at (-1,1) {};
        \node[circle, fill=black, inner sep=2pt] (r) at (1,1) {};
        \draw[thick] (l) -- (0) -- (r) -- (1) -- (l) -- (2) -- (r);
        \end{scope}
    \end{tikzpicture}
\end{center}
\begin{prop}[{Fujiwara-Papasoglu $+\ \varepsilon$}] \label{prop:quasitree_iff_no_fat_rose}
    A connected graph $\Gamma$ is quasi-isometric to a tree if and only if it contains no $R_2$ asymptotic minors. \vspace*{-1mm}
\end{prop}
\begin{proof}
    The \say{only if} is trivial. Let $\Gamma$ be a graph without $R_2$ asymptotic minors. Since $R_2$ is a minor of $K_{2,3}$, we deduce that $\Gamma$ is also without $K_{2,3}$ asymptotic minors. Therefore \cite[Proposition 1.10]{Asymptotic_minors_cacti} implies that $\Gamma$ is quasi-isometric to a cactus $\tilde\Gamma$. If $\tilde\Gamma$ contains arbitrarily large (simple) cycles, it contains a $R_2$ asymptotic minor. (Take two cycles of length $\ge 5K$, which either intersect in a point or are linked by a path, and build a $K$-fat minor.) Otherwise, there is a uniform bound on the length of cycles hence contracting them gives a quasi-isometry from $\tilde\Gamma$ (hence from $\Gamma$) to a tree.
\end{proof}
\begin{defi}
Let $\Gamma$ be a connected graph. Fix $C\ge 0$ and $K,\lambda\ge 1$.
\begin{itemize}[leftmargin=6mm]
    \item $\Gamma$ is a \emph{$(\lambda,C)$-quasi-tree} if $\Gamma$ is $(\lambda,C)$-quasi-isometric to a tree.
    \item $\Gamma$ is a \emph{$K$-quasi-tree} if $\Gamma$ contains no $K$-fat $R_2$ minor.
\end{itemize}
\end{defi}

Combining Proposition \ref*{prop:quasitree_iff_no_fat_rose} and \cite[Proposition 1.8]{kerr2023tree} gives
\begin{thm} Let $\Gamma$ be a connected graph. The following are equivalent
\begin{itemize}[leftmargin=6mm]
    \item There exist $\lambda\ge 1$ and $C\ge 0$ such that $\Gamma$ is a $(\lambda,C)$-quasi-tree.
    \item There exist $C\ge 0$ and a surjective map $f\colon \Gamma\to T$ to a $\R$-tree such that
    \[ d_\Gamma(x,y)-C \le d_T\bigl(f(x),f(y)\bigr) \le d_\Gamma(x,y). \]
    \item There exists $K\ge 1$ such that $\Gamma$ is a $K$-quasi-tree.
\end{itemize}
\end{thm}

\subsection{Chabauty topology}

We recall the \emph{local topology} (or \emph{Chabauty topology}) on rooted $\Ac$-graphs. A \emph{rooted $\Ac$-graph} is a pair $(\Gamma,\xi)$, where $\Gamma$ is a connected $\Ac$-graph and $\xi\in V(\Gamma)$ is a distinguished root. For $R\ge 0$, we denote by $(D_R^\Gamma(\xi),\xi)$ the rooted labelled ball of radius $R$ around $\xi$.
\begin{defi}
    A sequence of rooted $\Ac$-graphs $(\Gamma_n,\xi_n)$ converges to a rooted $\Ac$-graph $(\Gamma,\xi)$ if, for every $R\ge 0$, there exists $N_R$ such that for all $n\ge N_R$ one has an isomorphism of rooted labelled graphs
\[
\bigl(D_R^{\Gamma_n}(\xi_n),\xi_n\bigr)\cong \bigl(D_R^\Gamma(\xi),\xi\bigr).
\]
\end{defi}
This space is metrisable: two rooted graphs are close (say at distance $2^{-R}$) if they agree on a large rooted ball (say of radius $R$) around the basepoint.

\begin{defi}
    Given two connected $\Ac$-graphs $\Lambda,\Gamma$, we say that $\Lambda$ is \emph{a limit of} $\Gamma$, and write $\Lambda\preceq \Gamma$, if $(\Lambda,\lambda_0) \in \overline{\bigl\{(\Gamma,\xi):\xi\in V\Gamma\bigr\}}$
    for any (equivalently, all) $\lambda_0\in V\Lambda$.
\end{defi}

\begin{lemma} \label{lem:A-q-trees_are_closed}
The following sets are closed for the Chabauty topology:
    \begin{enumerate}[leftmargin=8mm, label={\normalfont(\alph*)}]
        \item For a fixed constant $K\ge 0$, the set of $K$-quasi-trees.
        \item For a fixed function $f\colon\N\to\N$, the set of graphs $\Gamma$ satisfying
        \[ \bar\beta_\Gamma(r)\coloneqq \sup_{y\in\Gamma}\#\bigl\{x\in\Gamma\;\big|\;d(x,y)\le r\bigr\} \le f(r). \]
    \end{enumerate}
\end{lemma}
\begin{proof}
    These properties are conjunctions of properties depending only on finitely many points (clopen properties), hence are closed.
\end{proof}

\medskip

The knowledgeable reader might remember that the \emph{Chabauty topology} usually refers to the topology on the space of subgroups $\Sub(G) \subset \{0, 1\}^G$. The two topologies are strongly linked, explaining the common name: \emph{complete} $\Ac$-graphs define transitive actions $V\Gamma\racts F_\Ac$. The stabiliser map 
\[ (\Gamma,\xi)\mapsto \Stab_{V\Gamma\racts F_\Ac}(\xi) \]
then defines an homeomorphism between the space of rooted complete $\Ac$-graphs and $\Sub(F_\Ac)$ with the usual Chabauty topology.

In particular, we can consider notions from the study of $\Sub(G)$, and translate them into properties of the space of rooted $\Ac$-graph, or more precisely the (closed) subspace of Schreier graphs of $G$:
\begin{defi} Let $G=\la\Ac\ra$ be a finitely generated group.
\begin{itemize}[leftmargin=6mm] 
    \item A subgroup $H\le G$ is \emph{confined} if $\{1\}\notin\overline{\{H^g:g\in G\}}$. Equivalently, the associated Schreier graph $\Gamma=\Sch(H\backslash G,G;\Ac)$ is confined if $\Cay(G;\Ac)\not\preceq\Gamma$.
    \item A confined subgroup $H\le G$ is a \emph{URS} if $\overline{\{H^g:g\in G\}}$ is minimal, i.e., all orbits are dense. A confined graph $\Gamma$ is a URS if $\Lambda\preceq \Gamma$ implies $\Gamma\preceq \Lambda$.
\end{itemize}
\end{defi}
A classical result is that any non-empty closed invariant subset $C\subseteq\Sub(G)$ such that $\{1\}\notin C$ contains an URS. This will be useful in Section \ref*{sec:branch}.

\subsection{Topological full groups of edge shifts} \label{sec:defi_tfg}

Given a finite oriented graph $\Sigma=\bigl(\Vc,\Ec\bigr)$ (with two maps $\iota,\tau\colon\Ec\to\Vc$) and $v_0\in \Vc$, we define the \emph{initial one-sided edge subshift} of $(\Sigma,v_0)$ as
\[ \Omega(\Sigma,v_0) = \Bigl\{ (e_n)_{n\ge 1}\in \Ec^\infty \;\Big|\; \iota(e_1)=v_0 \text{ and } \tau(e_n)=\iota(e_{n+1}) \Bigr\}.\]
In the literature, one often encounter the \emph{one-sided edge subshift}
\[ \Omega(\Sigma) = \Bigl\{ (e_n)_{n\ge 1}\in \Ec^\infty \;\Big|\; \tau(e_n)=\iota(e_{n+1}) \Bigr\}.\]
The subshifts are \emph{irreducible} if $\Sigma$ is strongly connected, and not reduced to a cycle. For $a=a_1\ldots a_\ell\in\Ec^+$, we define the cone $[a]=\{(e_n)_n\in\Omega(\Sigma)\mid \forall i=1,\ldots,\ell,\;e_i=a_i\}$.

For each $a \in \Ec^+$, one can define $\sigma_a: [a] \rightarrow \Omega(\Sigma,\tau(a))$ such that $(a\eta)\sigma_a=\eta$.

The \emph{topological full group} $\F(\Sigma,v_0)$ is the group given by the homeomorphisms of $\Omega(\Sigma,v_0)$ of the form
\[\phi: a_j\eta \mapsto b_j\eta \qquad \text{with } \tau(a_j)=\tau(b_j) \ \text{for } j=1,\ldots,m\]
where $a_1, \ldots, a_m,b_1, \ldots, b_m\in\Ec^+$ are such that
\[
\Omega(\Sigma,v_0)
= [a_1] \sqcup [a_2] \sqcup \ldots \sqcup[a_m]
= [b_1] \sqcup [b_2] \sqcup \ldots \sqcup [b_m].
\]

The same definition holds for $\F(\Sigma)$ if we consider $\Omega(\Sigma)$ instead of $\Omega(\Sigma,v_0)$.

\begin{thm}[{\cite[Theorem 6.21]{MatuiFP}}]
    Topological full groups of irreducible (initial) edge shifts are of type $F_\infty$, and in particular are finitely generated.
\end{thm}

\subsection{Wreath products}

Given two groups $G,L$ and an action $X\racts G$, the \emph{permutational wreath product} is $L \Wr_X G = \bigl(\prod_X L \bigr)\rtimes G$ with the operation
\[ (\Phi_1, g_1)(\Phi_2,g_2) = \bigl(\Phi_1(*)\cdot \Phi_2(*\cdot g_1),\; g_1g_2\bigr),\]
where $\Phi_i\colon X\to L$, and $g_i\in G$. A natural subgroup of $L\Wr_X G$ is the \emph{restricted} wreath product $L\wr_XG=\bigl(\bigoplus_XL\bigr)\rtimes G$, where $\Phi_i$ is finitely supported.

Given an action $Y\racts L$, there is a natural action $Y\times X\racts L\Wr_X G$ given by
\[ (y,x) (\Phi,g) = \bigl(y\cdot \Phi(x),x\cdot g\bigr).\]
If both actions $X\racts G$ and $Y\racts L$ are faithful (resp.\ transitive), then the action $Y\times X\racts L\Wr_X G$ is faithful (resp.\ transitive). This action is \emph{imprimitive}: the partition $Y\times X= \bigsqcup_{x\in X}Y\times\{x\}$ is preserved. Reciprocally, we can recover a wreath product from an imprimitive action:
\begin{thm}[Krasner-Kaloujnine] \label{thm:Krasner_Kaloujnine}
    Let $\Omega\racts W$ be a faithful transitive action. Suppose that the action is imprimitive, i.e.\ there exists a partition $\Omega=\bigsqcup_{x\in X}Y_x$ which is preserved by the action. Then $W\into L \Wr_X G$, where
    \begin{itemize}[leftmargin=6mm]
        \item $L$ is the image of the set-wise stabiliser of $Y_{x_0}$ in $\Sym(Y_{x_0})$, and
        \item $G$ is the image of $W$ in $\Sym(X)$.
    \end{itemize}
    Moreover, if for each $g\in W$ satisfying $\forall x\in X,\; Y_x\cdot g=Y_x$, there exists a finite subset $S_g\subseteq X$ such that $\supp(g)\subseteq \bigsqcup_{x\in S_g}Y_x$, then $W\into L\wr _X G$.
\end{thm}

%% file: Pictures/tikz_polygamma.tex
\begin{tikzpicture}[scale=1.1]

    \begin{scope}
        \clip (-4,-3.35) rectangle (4.9,4);
		\node[circle, draw, inner sep=1.8pt] (v0) at (0,0) {\footnotesize $x_0$};
		\begin{scope}[rotate=90]
			\draw[thick, purple, -latex] (v0) -- (.92,0);
			\foreach \x in {1,...,3}{
				\node[circle, fill=orange, inner sep=1.2pt] at (\x,0) {};
				\draw[thick, purple, -latex] ({\x+.1},0) -- ({\x+.9},0);
				
				\draw[thick, blue, -latex] (\x,.1) -- (\x,.9);
				\node[circle, fill=yellow, inner sep=1.2pt] at (\x,1) {};
				\draw[thick, purple, looseness=10, out=-150, in=150, -latex] ({\x-.07},.95) to ({\x-.07},1.05);
				\draw[thick, blue, -latex] (\x,1.1) -- (\x,1.9);
				\node[circle, fill=yellow, inner sep=1.2pt] at (\x,2) {};
				\draw[thick, purple, looseness=10, out=-150, in=150, -latex] ({\x-.07},1.95) to ({\x-.07},2.05);
				\draw[thick, blue, -latex] (\x,2.1) -- (\x,2.9);
				
				\draw[thick, blue, -latex] (\x,-.9) -- (\x,-.1);
				\node[circle, fill=purple, inner sep=1.2pt] at (\x,-1) {};
				\draw[thick, purple, looseness=10, out=-150, in=150, -latex] ({\x-.07},-1.05) to ({\x-.07},-.95);
				\draw[thick, blue, -latex] (\x,-1.9) -- (\x,-1.1);
				\node[circle, fill=purple, inner sep=1.2pt] at (\x,-2) {};
				\draw[thick, purple, looseness=10, out=-150, in=150, -latex] ({\x-.07},-2.05) to ({\x-.07},-1.95);
				\draw[thick, blue, -latex] (\x,-2.9) -- (\x,-2.1);}
		\end{scope}
		
		\begin{scope}[rotate=210]
			\node[circle, draw, inner sep=1.2pt] (v1) at (1,0) {\tiny 1};
			\draw[thick, purple, latex-] (v0) -- (v1);
			\draw[thick, purple, latex-] (v1) -- (1.92,0);
			\draw[thick, blue, looseness=10, out=120, in=60, -latex] (.95,.15) to (1.05,.15);
			\foreach \x in {2,...,4}{
				\node[circle, fill=blue, inner sep=1.2pt] at (\x,0) {};
				\draw[thick, purple, latex-] ({\x+.1},0) -- ({\x+.9},0);
				
				\draw[thick, bend left, blue, -latex] ({\x-.05},.1) to ({\x-.05},.9);
				\node[circle, fill=Turquoise, inner sep=1.2pt] at (\x,1) {};
				\draw[thick, purple, looseness=10, out=120, in=60, -latex] ({\x-.05},1.08) to ({\x+.05},1.08);
				\draw[thick, bend left, blue, -latex] ({\x+.05},.9) to ({\x+.05},.1);}
		\end{scope}
		
		\begin{scope}[rotate=330]
			\node[circle, draw, inner sep=1.2pt] (v2a) at (.86,-0.5) {\tiny 2};
			\node[circle, draw, inner sep=1.2pt] (v2b) at (.86,0.5) {\tiny 2};
			
			\draw[thick, blue, -latex] (v0) -- (v2a);
			\draw[thick, blue, -latex] (v2a) -- (v2b);
			\draw[thick, blue, -latex] (v2b) -- (v0);
			
			\draw[thick, purple, -latex] ({1.86-.08},.5) -- (v2b);
			\draw[thick, purple, bend right, -latex] (v2b) to (v2a);
			\draw[thick, purple, -latex] (v2a) -- ({1.86-.08},-.5);
			 
			\foreach \x in {0,1}{
				\node[circle, fill=Green, inner sep=1.2pt] (P\x) at ({2*\x+1.86},0.5) {};
				\node[circle, fill=green, inner sep=1.2pt] (Q\x) at ({2*\x+2.86},0.5) {};
				\node[circle, fill=green, inner sep=1.2pt] (R\x) at ({2*\x+2.86},-0.5) {};
				\node[circle, fill=Green, inner sep=1.2pt] (S\x) at ({2*\x+1.86},-0.5) {};
				
				\draw[thick, purple, latex-] (P\x) to (Q\x);
				\draw[thick, blue, bend left, -latex] (P\x) to (Q\x);
				\draw[thick, blue, -latex] (Q\x) to (R\x);
				\draw[thick, purple, latex-] (R\x) to (S\x);
				\draw[thick, blue, bend left, -latex] (R\x) to (S\x);
				\draw[thick, blue, -latex] (S\x) to (P\x);
				
				\draw[thick, purple, latex-] (Q\x) to ({2*\x+3.86-.08},.5);
				\draw[thick, purple, latex-] ({2*\x+3.86-.08},-.5) to (R\x);}
		\end{scope}
    \end{scope}
    
    \begin{scope}[shift={(5.8,2.2)}, scale=.9]
        \draw[rounded corners] (-2.25,-2.2) rectangle (2.85,1.7);
        \node[rounded corners, draw, fill=white] at (1.8,-2.2) {$\mathcal D(\Gamma)$};
        \node[circle, draw, inner sep=1.8pt] (v0) at (0,0) {\footnotesize $\Gamma_0$};
		\begin{scope}[rotate=90]
			
			\node[circle, fill=orange, inner sep=1.5pt] (orange) at (1,0) {};
			\draw[thick, looseness=10, out=-30, in=30, -latex] (1.07,-.05) to (1.07,.05);
			
			\node[circle, fill=yellow, inner sep=1.5pt] (yellow) at (1,1) {};
			\draw[thick, looseness=10, out=60, in=120, -latex] (1.05,1.07) to (.95,1.07);
			
			\node[circle, fill=purple, inner sep=1.5pt] (purple) at (1,-1) {};
			\draw[thick, looseness=10, out=-60, in=-120, -latex] (1.05,-1.07) to (.95,-1.07);
				
			\draw[thick, -latex] (v0) -- (orange);
			\draw[thick, -latex] (orange) -- (yellow);
			\draw[thick, -latex] (orange) -- (purple);
		\end{scope}
		
		\begin{scope}[rotate=210]
			\node[circle, draw, inner sep=1.2pt] (v1) at (1,0) {\footnotesize 1};
			\node[circle, fill=blue, inner sep=1.5pt] (blue) at (2,0) {};
			\node[circle, fill=Turquoise, inner sep=1.5pt] (lblue) at (2,1) {};
			
			\draw[thick, -latex] (v0) -- (v1);
			\draw[thick, -latex] (v1) -- (blue);
			\draw[thick, -latex] (blue) -- (lblue);
			
			\draw[thick, looseness=10, out=-30, in=30, -latex] (2.07,-.05) to (2.07,.05);
		\end{scope}
		
		\begin{scope}[rotate=330]
			\node[circle, draw, inner sep=1.2pt] (v2) at (1,0) {\footnotesize 2};
			\node[circle, fill=Green, inner sep=1.5pt] (green) at (2,0) {};
			\node[circle, fill=green, inner sep=1.5pt] (lgreen) at (3,0) {};
			
			\draw[thick, -latex] (v0) -- (v2);
			\draw[thick, -latex] (v2) -- (green);
			\draw[thick, bend right, -latex] (green) to (lgreen);
			\draw[thick, bend right, -latex] (lgreen) to (green);
		\end{scope}
    \end{scope}
\end{tikzpicture}

%% file: Sections/sec2_equiv.tex
In this section we prove the equivalence between the class of finitely generated subgroups of $V$ and the class of CF-TR groups. 

\subsection{Subgroups of \texorpdfstring{$V$}{V} are CF-TR groups}

\begin{prop} \label{prop:Schreier_of_GleV_are_CF}
    Let $G=\la \Ac\ra$ be a f.g.\ subgroup of Thompson's $V$ and $\xi=uv^\infty\in\{0,1\}^\infty$. Then the graph $\Sch(\xi, G;\Ac)$ is context-free.
\end{prop}
This result should be compared to \cite[Theorem 4.1]{ET0L}, and the proof with \cite[Lemma 2.11]{bennett2016demonstrative}.
\begin{proof} We assume without loss of generality that $v$ is not a proper power. We suppose that $G$ is generated by $\Ac=\{g_1,\ldots,g_\ell\}$ with
\[ g_i\colon a_{ij}\hspace{1pt}\eta \mapsto b_{ij}\hspace{1pt}\eta \quad \bigl(\eta\in\{0,1\}^\infty,\; j=1,\ldots,m_i\bigr). \]
Here is a pushdown automaton recognizing $\{w\in\Ac^*:\xi w=\xi\}$, see Fig.~\ref{fig: pushdown}. It is defined on four states $\mathsf Q=\{\start, \q_1, \q_2,\accept\}$, and has the following transitions:
\begin{itemize}[leftmargin=8mm]
    \item $(\start,\varepsilon, \#, u\#,\q_1)$,
    \item $(\q_1, g_i, a_{ij}, b_{ij},\q_1)$ for $g_i\in \Ac$,
    \item If $p$ is a prefix of $a_{ij}$, and $a_{ij}$ is a prefix of $pv^\infty$, say $pv^n=a_{ij}r$ with $n\ge 0$ minimal, then we add a rule $(\q_1, g_i, p\#,b_{ij}r\#,\q_1)$,
    \item $(\q_1,\varepsilon, \varepsilon,\varepsilon,\q_2)$
    \item $(\q_2,\varepsilon, uv,u,\q_2)$
    \item $(\q_2, \varepsilon, u\#, \#, \accept)$
\end{itemize}
We charge $u$ on the stack and move to the main state $\q_1$, this is the only state where we can read $w\in\Ac^*$. After $t$ steps, we have read a prefix $w_t$ of $w$ and the stack reads $x_t\#$ where $\xi\cdot w_t=x_tv^\infty$. This can be proven by induction on $t$. At the end of $w$, the stack therefore reads $x\#$ where $\xi\cdot w=xv^\infty$.

We now pass to the \say{verification}. We must accept $w$ if and only if $xv^\infty=\xi$ that is $x=uv^n$ for some $n\ge 0$. This is checked by removing as many copies of $v$ after a prefix $u$ as possible, and then checking if the stack reads $u\#$.
\end{proof}
\begin{center}
	\begin{tikzpicture}[scale=.7, thick]	
		\node[state, initial left, initial distance=6mm, minimum size=15pt] (v0) at (0,0) {};
		\node[state, minimum size=20pt] (v1) at (5,0) {$\q_1$};
		\node[state, minimum size=20pt] (v2) at (10,0) {$\q_2$};
		\node[state, accepting, minimum size=15pt] (v3) at (15,0) {};
		
		\path (v0)  edge[-latex] node [above] {$(\varepsilon,\#,u\#)$} (v1)
		(v1) edge[-latex] node [above] {$(\varepsilon, \varepsilon, \varepsilon)$} (v2)
		(v1) edge [loop above, -latex] node [above] {$(g_i, a_{ij}, b_{ij})$}  (v1)
        (v1) edge [loop below, -latex] node [below] {$(g_i, p\#, b_{ij}r\#)$}  (v1) (v1)
		(v2) edge [loop above, -latex] node [above] {$(\varepsilon, uv, u)$} (v2)
		(v2) edge [-latex] node [above] {$(\varepsilon, u\#, \#)$} (v3);
	\end{tikzpicture}
    \captionsetup{font=small}
    \captionof{figure}{The pushdown automaton recognizing $L\bigl(\Sch(\xi, G;\Ac),\xi\bigr)$.}\label{fig: pushdown}
\end{center}

We need a last lemma:
\begin{lemma} \label{lem:finitely_many_isomorphism_type}
Fix $G=\la\Ac\ra\le V$ and $v\in\{0,1\}^*$. The graphs $\Sch(\xi, G;\Ac)$ with $\xi\in X=\bigl\{uv^\infty \mid u\in\{0,1\}^*\bigr\}$ fall into finitely many isomorphism classes.
\end{lemma}
\begin{proof}
    Fix $\tilde\Ac\supseteq\Ac$ such that $V=\langle\tilde\Ac\rangle$. The graph $\tilde\Gamma=\Sch(X, V;\tilde\Ac)$ is context-free, and the Schreier graphs $\Sch(\xi, G;\Ac)$ are exactly the connected components of the subgraph $\Gamma$ obtained from $\tilde\Gamma$ by just keeping the $\Ac$-edges.

    Fix $x_0\in\tilde\Gamma$ and consider a connected component $C$ of $\Gamma$. Let $z\in C$ be a vertex minimising $d(z,x_0)$. Observe that $C\subseteq\tilde\Gamma(z,x_0)$, and the isomorphism type of $C$ is fully determined by the end-cone type $\tilde\Gamma_i$ of $\tilde\Gamma(z,x_0)$ and the choice of $v\in \tilde\Delta_i$. Since $\tilde\Gamma$ is context-free, this leaves finitely many choices.
\end{proof}

\medskip

Finally, we can prove the first half of our characterisation: \vspace*{1mm}
\begin{thm}
Finitely generated subgroups of Thompson's $V$ are \textbf{CF-TR}. \vspace*{-1mm}  
\end{thm}
\begin{proof}
    Let $G=\la \Ac\ra\le V$. Since $X=\bigl\{uv^\infty \,\big|\, u\in\{0,1\}^*\bigr\}$ is dense in $\Ck$ and $G$ acts by homeomorphism, the support of any non-trivial element (a non-empty open set) intersects with $X$, i.e.\ the action on $X$ is faithful. Moreover, the graphs $\Sch(\xi,G;\Ac)$ are context-free for all $\xi\in X$ by Proposition \ref{prop:Schreier_of_GleV_are_CF}. The action on $X$ may have infinitely many orbits, however we can restrict to finitely many orbits and conserve faithfulness using Lemma \ref{lem:finitely_many_isomorphism_type}, hence $G$ admits a faithful context-free action; $G$ is \textbf{CF-TR}.
\end{proof}
\begin{rem}
    One could bypass Lemma \ref{lem:finitely_many_isomorphism_type} by proving that $V$ is \textbf{CF-TR} (since $V$ acts transitively on the dyadics $\D$), and using that finitely generated subgroups of \textbf{CF-TR} groups are \textbf{CF-TR} \cite[Proposition 4.17]{CFTR}.
\end{rem}
\subsection{All \texorpdfstring{$\Sch(\xi,G;\Ac)$}{Sch(xi,G;A)} are quasi-trees} \label{sec:quasi-trees}

Our Proposition \ref*{prop:Schreier_of_GleV_are_CF} is reminiscent of a recent result due to Hyde, Skipper and Zaremsky, stating that the graphs $\Sch(\xi, G;\Ac)$ are quasi-trees for \emph{all} $\xi\in\Ck$, see \cite[Theorem A]{hyde2026}. Indeed, our result implies their result whenever $\xi$ is eventually periodic, since context-free graphs are quasi-trees. We sketch a proof recovering a slight strengthening of their result:  

\begin{prop} \label{prop:Schreier_of_GleV_are_qtrees}
    Fix $G=\la\Ac\ra\le V$. There exists a constant $K\ge 0$ such that, for all aperiodic $\xi\in\Ck$, the graph $\Sch(\xi, G;\Ac)$ is a $K$-quasi-tree.
\end{prop}

\begin{proof}
Using Lemma \ref{lem:finitely_many_isomorphism_type}, we know that all Schreier graphs $\Sch(\xi, G;\Ac)$ with $\xi\in\D:=\{u0^\infty\mid u\in\{0,1\}^*\}$ are $K$-quasi-trees for a fixed constant $K$. Since this set is closed in the Chabauty topology, this opens the possibility to take limits, which is formalised using the following notion:
\begin{defi}
    Let $\mathfrak X$ be a topological space and $G\le\Homeo(\mathfrak X)$. For each $\xi\in\mathfrak X$, we define its \emph{germ stabiliser} as
    \[ \GStab_G(\xi) = \bigl\{g\in G \;\big|\; \exists U\text{ neighbourhood of }\xi\text{ s.t.\ } g|_U=\id_U \bigr\}. \]
\end{defi}
The following lemma is an easy exercise.
\begin{lemma} \label{lem:GStab}
    It $\xi_n\to \xi$, then
    \[ \GStab_G(\xi) \le \liminf_{n\to\infty} \Stab_G(\xi_n) \le \limsup_{n\to\infty} \Stab_G(\xi_n) \le \Stab_G(\xi).\]
\end{lemma}
Fix $\xi\in\Ck$ an \emph{aperiodic / irrational} point. Elements of $V$ are piecewise linear with \emph{rational} coefficients, therefore elements fixing $\xi$ must also fix a neighbourhood of $\xi$, i.e.\ $\GStab_G(\xi)=\Stab_G(\xi)$. 

Consider points $\xi_n\in\D$ such that $\xi_n\to \xi$. Combining Lemma \ref*{lem:GStab} and the previous equality gives $\Stab_G(\xi_n)\to\Stab_G(\xi)$, i.e.
    \[ \bigl( \Sch(\xi_n, G;\Ac),\xi_n\bigr) \longto \bigl(\Sch(\xi, G;\Ac),\xi\bigr). \]
    Since graphs on the left-hand side are $K$-quasi-trees, Lemma \ref{lem:A-q-trees_are_closed}(a) concludes that the limit $\Sch(\xi, G;\Ac)$ is also a $K$-quasi-tree.
\end{proof}

\subsection{CF-TR groups are subgroups of \texorpdfstring{$V$}{V}} \label{subsection:CFTR-embedding}

We prove that every \textbf{CF-TR} group embeds in $V$; adapting the embedding of $\Gc(\Gamma)$ into the rational group \cite[Section 7]{CFTR}. We re-use some of their notations, notably the parametrisation $x\leftrightarrow w_x$ of vertices of $\Gamma$ by \emph{well-formed words}. There are two key differences:
\begin{itemize}[leftmargin=6mm]
    \item A key observation in \cite{CFTR} was that $w_x$ and $w_{xa}$ parametrising vertices of $\Gamma$ only differ in a short \emph{suffix}. We consider the reverse $\overleftarrow{w_x}$ to get words that only differ on a short \emph{prefix}, compatible with the definition of $V$.
    \item In the previous embedding, the flexibility of rational homeomorphisms was used to restrict the support to infinite words of the form $w_x\eta$. Here, we will instead restrict the space we act on, obtaining an embedding into the topological full group of a proper subshift. The last step is a theorem of Matui embedding such groups inside $V$.
\end{itemize}
Let us proceed with the proof.
\begin{prop} \label{prop:embedding-tfg}
    Let $\Gc(\Gamma)$ be the transition group of a context-free graph. Then $\Gc(\Gamma)$ embeds in the topological full group $\F(\Sigma,v_0)$ of a one-sided irreducible edge shift $\Omega(\Sigma,v_0)$. \vspace*{-1mm}
\end{prop}
\begin{proof}
If $\Gamma$ is reduced to a single vertex, then $\Gc(\Gamma)$ and the statement are both trivial. Otherwise, $\Gamma$ has at least two end-cone types. We consider the subshift $\Omega(\Sigma,v_0)$ defined by the following graph $\Sigma=(\Vc,\Ec)$:
\begin{itemize}[leftmargin=6mm]
    \item The vertex set is $\Vc = \bigl\{ v_0\}\cup \bigl\{\Gamma_j \;\big|\;  \Gamma_j\ne \Delta_j\bigr\}$. 
    \item The edge set is $\Ec = \bigl\{ \Gamma_0, (\Gamma_0,x_0), \Gamma_i^{(j)}, (\Gamma_i^{(j)},v)\bigr\}$ where
    \begin{itemize}[leftmargin=5mm]
        \item $\Gamma_0$ goes from $\Gamma_0$ to $v_0$,
        \item $(\Gamma_0,x_0)$ goes from $v_0$ to $v_0$,
        \item $\Gamma_i^{(j)}$ goes from $\Gamma_k$ to $\Gamma_j$, when $\Gamma_i^{(j)}$ is an end-cone of type $\Gamma_k$, and
        \item $(\Gamma_i^{(j)},v)$ goes from $v_0$ to $\Gamma_j$.
    \end{itemize}
\end{itemize}
For instance, for the line graph in Figure \ref*{fig:line_graph} (with $\Gamma_1=\Gamma^{(0)}_1=\Gamma(x_1,x_0)$ and $\Gamma_2=\Gamma^{(0)}_2=\Gamma(x_2,x_0)$), we get the graph $\Sigma$ given in Figure \ref*{fig:Sigma}:
\begin{center}
\begin{minipage}{.31\linewidth}
    \centering
    \begin{tikzpicture}[scale=.75]
    \clip (-2.8,-3) rectangle (2.8,3.7);
    \begin{scope}[rotate=45]
        \foreach \x in {-3,3}{
			\node[circle, fill=black, inner sep=1.5pt] at (\x,0) {};}
        \foreach \x in {-4,...,3}{
			\draw[blue, very thick, -latex] ({\x+0.09},0) -- ({\x+.91},0);}

        {\footnotesize
        \node[circle, fill=black, inner sep=1.5pt] at (-2,0) {};
        \node at (-2,-.5) {$y_2$};
        \node[circle, fill=black, inner sep=1.5pt] at (-1,0) {};
        \node at (-1,-.5) {$x_2$};
        \node[circle, fill=Green, inner sep=2pt] at (0,0) {};
        \node at (0,-.5) {$x_0$};
        \node[circle, fill=black, inner sep=1.5pt] at (1,0) {};
        \node at (1,-.5) {$x_1$};
        \node[circle, fill=black, inner sep=1.5pt] at (2,0) {};}
        \node at (2,-.5) {$y_1$};
    \end{scope}
    \end{tikzpicture}\vspace*{-3mm}
    \captionsetup{font=small}
    \captionof{figure}{A graph $\Gamma$.}
    \label{fig:line_graph}
\end{minipage}
\begin{minipage}{.66\linewidth}
    \centering
    \begin{tikzpicture}[scale=.6, thick]	
		\node[state, inner sep=3pt, minimum size=12pt] (v0) at (0,5) {$v_0$};
		\node[state, inner sep=3pt, minimum size=15pt] (G2) at (-5,0) {$\Gamma_2$};
		\node[state, inner sep=3pt, minimum size=15pt] (G0) at (0,0) {$\Gamma_0$};
		\node[state, inner sep=3pt, minimum size=15pt] (G1) at (5,0) {$\Gamma_1$};

        {\footnotesize
		\path (v0) edge[->, out=120, in=60, looseness=4] node[above] {$(\Gamma_0,x_0)$} (v0)
        (G0) edge[->] node[right] {$\Gamma_0$} (v0)
        (v0) edge[->, bend left=40] node[right] {$(\Gamma^{(0)}_1,x_1)$} (G0)
        (v0) edge[->, bend right=40] node[left] {$(\Gamma^{(0)}_2,x_2)$} (G0)
        (v0) edge[->, bend left=40] node[above right] {$(\Gamma^{(1)}_1,y_1)$} (G1)
        (v0) edge[->, bend right=40] node[above left] {$(\Gamma^{(2)}_1,y_2)$} (G2)
        (G1) edge[->] node[below] {$\Gamma^{(0)}_1$} (G0)
        (G2) edge[->] node[below] {$\Gamma^{(0)}_2$} (G0)
        (G1) edge[->, out=30, in=-30, looseness=4] node[right] {$\Gamma^{(1)}_1$} (G1)
        (G2) edge[->, out=-150, in=150, looseness=4] node[left] {$\Gamma^{(2)}_1$} (G2)
		;}
	\end{tikzpicture}\vspace*{-2mm}
    \captionsetup{font=small}
    \captionof{figure}{The graph $\Sigma$ defining the edge shift for $\Gc(\Gamma)$.}
    \label{fig:Sigma}
\end{minipage}
\end{center}

\newpage

By construction, each vertex is contained in a circuit $v_0\to v_0$:
\begin{itemize}[leftmargin=6mm]
    \item The vertex $v_0$ admits a loop labelled by $(\Gamma_0,x_0)$.
    \item By hypothesis, each vertex $\Gamma_j$ admits a second-level end-cone $\Gamma_1^{(j)}$, giving edges $v_0\to\Gamma_j$ labelled by $(\Gamma_1^{(j)},v)$ for any $v\in \Delta_1^{(j)}$. Moreover, each representative $\Gamma(y,v_0)\simeq\Gamma_j$ is connected to $\Gamma_0$ by an end-geodesic  (as defined in \cite[p.24]{CFTR}), which labels a path from $\Gamma_j$ to $v_0$ in $\Sigma$.
\end{itemize}
More generally, each minimal circuit $v_0\to v_0$ describes the reverse of a unique well-formed word $w \in \mathcal{F}(\Gamma)$, hence a unique vertex $x\in V\Gamma$ \cite[Lemma 7.3]{CFTR}. We denote by $\overleftarrow{w_x}$ that unique reversed word.

Moreover, the vertex $v_0$ has at least two ingoing edges (from itself and from $\Gamma_0$), hence $\Sigma$ is not reduced to a cycle, the subshift $\Omega(\Sigma,v_0)$ is irreducible.

The key observation is that the words $w_x$ and $w_{xa}$ (for $a\in\Ac$ and $x\in V\Gamma$) only differ on a short suffix. This is already used in \cite[Lemma 7.6]{CFTR}. More precisely, if $w_x=\Gamma_0\Gamma_{j_1}^{(i_1)}\ldots \Gamma_{j_{\ell-2}}^{(i_{\ell-2})}\Gamma_{j_{\ell-1}}^{(i_{\ell-1})}(\Gamma_{j_\ell}^{(i_\ell)},v)$, we have three cases:
\begin{itemize}[leftmargin=6mm]
    \item If $va\in \Delta_{i_\ell}$ (i.e.\ $d(xa,v_0)=d(x,v_0)-1$), then
    \[ w_{xa} = \Gamma_0\Gamma_{j_1}^{(i_1)}\ldots \Gamma_{j_{\ell-2}}^{(i_{\ell-2})}(\Gamma_{j_{\ell-1}}^{(i_{\ell-1})},v').\]
    \item If $va\in \Delta_{j_\ell}^{(i_\ell)}$ (i.e.\ $d(xa,v_0)=d(x,v_0)$), then
    \[ w_{xa} = \Gamma_0\Gamma_{j_1}^{(i_1)}\ldots \Gamma_{j_{\ell-2}}^{(i_{\ell-2})}\Gamma_{j_{\ell-1}}^{(i_{\ell-1})}(\Gamma_{j_\ell}^{(i_\ell)},v').\]
    \item Otherwise (i.e.\ $d(xa,v_0)=d(x,v_0)+1$), then
    \[ w_{xa} = \Gamma_0\Gamma_{j_1}^{(i_1)}\ldots \Gamma_{j_{\ell-2}}^{(i_{\ell-2})}\Gamma_{j_{\ell-1}}^{(i_{\ell-1})}\Gamma_{j_\ell}^{(i_\ell)}(\Gamma_{j}^{(i)},v').\]
\end{itemize}
The precise $v'$, and in the last case $\Gamma_{j}^{(i)}$, only depends on $a$ and on the suffix $\Gamma_{j_{\ell-1}}^{(i_{\ell-1})} (\Gamma_{j_\ell}^{(i_\ell)},v)$ (see the proof of \cite[Lemma 7.6]{CFTR}). We deduce that
\[ \psi(a) \colon\; \overleftarrow{w_x} \eta \longmapsto \overleftarrow{w_{xa}}\eta \qquad \text{for all} \  \eta \in \Omega(\Sigma,v_0) \]
extends to an element of $\F(\Sigma,v_0)$. The function $\psi$ extends to a morphism $\psi\colon F_\Ac\to \F(\Sigma,v_0)$, which we compare with $\varphi\colon F_\Ac\to \Gc(\Gamma)$. By induction on the length of $u\in F_\Ac$, we have $(\overleftarrow{w_x}\eta)\cdot \psi(u)=\overleftarrow{w_{xu}}\eta$. Consider
\[ X=\bigl\{\overleftarrow{w_x}\eta \;\big|\; x\in V\Gamma,\; \eta\in \Omega(\Sigma,v_0)\bigr\}.\]
It follows that
\begin{align*}
\varphi(u)=1
& \iff \forall \overleftarrow{w_x}\eta \in X,\quad  (\overleftarrow{w_x}\eta)\cdot   \psi(u)= (\overleftarrow{w_{xu}}\eta) \overset!= \overleftarrow{w_x}\eta, \\
& \iff \forall \xi\in\Omega(\Sigma,v_0),\quad \xi \cdot \psi(u)=\xi \\
& \iff \psi(u)=1,
\end{align*}
where the second equivalence follows from $X$ being a dense subset of $\Omega(\Sigma,v_0)$, and $\psi(u)$ being continuous. We conclude $\Gc(\Gamma)\simeq \psi(F_\Ac)\le \F(\Sigma,v_0)$.
\end{proof}

\begin{thm}\label{thm:CFTR-in-V}
Let $G=\Gc(\Gamma^{(1)} \sqcup \ldots \sqcup \Gamma^{(k)})$ be a finitely generated \textbf{CF-TR} group. Then $G$ is a finitely generated subgroup of $V$. \vspace*{-1mm}
\end{thm}
\begin{proof}
Using Proposition \ref{prop:embedding-tfg} and \cite[Proposition 5.14]{Matui} (see also \cite[Corollary 11.15]{MatteBonTFG} or \cite[Section 8.3]{Tarrot}), we can embed each $\Gc(\Gamma^{(i)})$ into the topological full group of an (initial) irreducible edge shift, and then into $V$:
\[ \Gc(\Gamma^{(i)}) \,\into\, \F(\Sigma,v_0) = \RStab_{\F(\Omega(\Sigma))}(\Omega(\Sigma,v_0))  \le \F(\Sigma) \,\into\, V. \]
This concludes since  \cite[Lemma 4.6]{CFTR} gives
\[ \Gc(\Gamma^{(1)}\sqcup \ldots \sqcup \Gamma^{(k)}) \le \Gc(\Gamma^{(1)}) \times \cdots \times \Gc(\Gamma^{(k)}), \]
and the class of subgroups of $V$ is closed under direct products.
\end{proof}

\medskip

\begin{rem}
    The embedding $\iota_1\colon\Gc(\Gamma)\into \F(\Sigma,v_0)$ described in the proof of Proposition \ref{prop:embedding-tfg} satisfies 
    $\Sch\bigl(\xi,\iota_1(\Gc(\Gamma));\Ac\bigr)\simeq\Gamma$ for all $\xi$ in the dense open subset $X\subset \Omega(\Sigma,v_0)$. (Equivalently, for any path $\xi\in\Omega(\Sigma,v_0)$ returning to $v_0$.) Looking into the proof of \cite[Corollary 11.15]{MatteBonTFG}, the same holds for $\iota_2\colon \Gc(\Gamma)\into V$ and $\xi$ in a dense open subset of $\supp(\iota_2(\Gc(\Gamma))\subseteq \Ck$.

    Moreover, if $\xi\in\supp(\iota(\Gc(\Gamma)))$ is aperiodic, then $\Sch\bigl(\xi,\iota(\Gc(\Gamma));\Ac\bigr)\preceq\Gamma$, using the same argument with germ stabilisers as in Section \ref*{sec:quasi-trees}
\end{rem}

%% file: Sections/sec3_Lehnert.tex
In this section, we prove that two previously known classes of \textbf{coCF} groups are \textbf{CF-TR}, and therefore embed inside Thompson's $V$.

\subsection{Cloning systems \texorpdfstring{$V_{(H,\theta)}$}{V(H,theta)}}

Consider a group $H$ and $\theta\in\End(H)$. 
The group $V_{(H,\theta)}$ is an instance of cloning system as introduced in \cite{cloning, Clone_guide}. Elements of $V_{(H,\theta)}$ can be described by finitely many triplets $\{(a_i,b_i,h_i):i=1,\ldots,m\}$ with
\[ \Ck = [a_1]\sqcup \ldots \sqcup [a_m] = [b_1] \sqcup \ldots \sqcup [b_m], \]
and $h_i\in H$. Similarly to elements in $V$, this description is ambiguous: replacing a triplet $(a_i,b_i,h_i)$ by two triplets $(a_i0,b_i0,h_i),(a_i1,b_i1,(h_i)\theta)$ gives the same element. This can be understood pictorially; elements are tree pairs, with a permutation and elements of $H$ on each of the middle edges:
\begin{center}
    \import{Pictures/}{tikz_VGtheta.tex}
    \captionsetup{font=small}
    \captionof{figure}{Two elements $g_1,g_2$.} 
    \label{fig:VGtheta-elements}
\end{center}
To compose two elements $g_1=\{(a_i,b_i,h_i)\}$ and $g_2=\{(c_i,d_i,k_i)\}$, ensure that $b_i=c_i$ by taking the refinements, and then let $g_1g_2=\{(a_i,d_i,h_ik_i)\}$. That being said, composition is also best understood pictorially:
\begin{center}
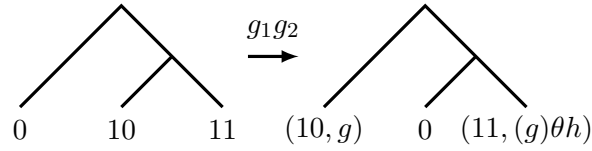

 \begin{tikzpicture}[scale = 1.34, very thick]
    \node at (2.5,.77) {$g_1g_2$};
    \draw[-latex] (2.25,.5) -- (2.75,.5);
    
    \draw (0,0) -- (1,1) -- (2,0);
    \draw (1.5,.5) -- (1,0);
    \node at (0,-.22) {$0$};
    \node at (1,-.22) {$10$};
    \node at (2,-.22) {$11$};
    
    \draw (3,0) -- (4,1) -- (5,0);
    \draw (4.5,.5) -- (4,0);
    \node at (3,-.22) {$(10,g)$};
    \node at (4,-.22) {$0$};
    \node at (5,-.22) {$(11,(g)\theta h)$};
    
 \end{tikzpicture} 
    \captionof{figure}{The product $g_1g_2$.}
\end{center}
There is a map onto $V$, forgetting the labels. More strongly, combining Observations 2.28, 2.29 and Lemma 3.2 of \cite{cloning} gives a split exact sequence
\begin{center}
    \begin{tikzcd}
        1 \arrow[r] & \bigoplus_{i=1}^\infty H \arrow[r, hook] & V_{(H,\theta)} \arrow[r, two heads, swap, "\pi"] & V \arrow[l, dashed, bend right] \arrow[r] & 1.
    \end{tikzcd}
\end{center}
This can be improved into an embedding in an unrestricted wreath product. Let $\D=\{v0^\infty:v\in\{0,1\}^*\}$ be the set of dyadic points in the Cantor set.
\begin{prop}
The group $V_{(H,\theta)}$ embeds in the unrestricted permutational wreath product $H\Wr_\D V$ via the monomorphism
\[ \iota\colon g=\{(a_i,b_i,h_i)\} \longmapsto (\Phi_g, \pi(g)), \]
where $\Phi_g(a_iv\hspace{1pt}0^\infty)=(h_i)\theta^v$ and $\theta^{s_1s_2\ldots s_\ell} = \theta^{s_1} \theta^{s_2} \ldots \theta^{s_\ell}$ for $s_1s_2\ldots s_\ell\in\{0,1\}^*$.
\end{prop}
\begin{proof}
    (1) We justify that the map $\iota$ is well-defined, i.e.\ that $\iota(g)$ remains unchanged when we replace $(a_i,b_i,h_i)$ by $(a_i0,b_i0,h_i), (a_i1,b_i1,(h_i)\theta)$:
    \begin{itemize}[leftmargin=8mm]
        \item Consider $\xi\in \D$ and suppose that $\xi=a_iv0^\infty$ with $v\in\{0,1\}^*$. Write $v=cv'$ with $c\in\{0,1\}$. We separate two cases:
        \begin{itemize}[leftmargin=5mm]
            \item if $c=0$, then $(h_i)\theta^v=(h_i)\theta^{v'}$.
            \item if $c=1$, then $(h_i)\theta^v = ((h_i)\theta)\theta^{v'}$.
        \end{itemize}
        That is, both definitions of $\Phi_g(\xi)$ coincide.
        \item $\pi(g)$ doesn't change, by definition of $V$.
    \end{itemize}

    \medskip
    
    (2) We prove that $\iota$ is a morphism. Consider two elements $h=\{(a_i,b_i,h_i)\}$ and $k=\{(b_i,c_i,k_i)\}$. Since $\pi$ is a morphism, we have to check that
    \[ \forall \xi\in\D,\quad \Phi_h(\xi)\cdot \Phi_k\bigl(\xi\cdot \pi(h)\bigr) = \Phi_{hk}(\xi). \]
    We suppose that $\xi=a_iv0^\infty$, so that $\xi\cdot\pi(h)=b_iv0^\infty$. Then \begin{itemize}[leftmargin=8mm]
        \item on the LHS, $\Phi_h(\xi)=(h_i)\theta^v$ and $\Phi_k(\xi\cdot \pi(h))=(k_i)\theta^v$, and
        \item on the RHS, $hk=\{(a_i,c_i,h_ik_i)\}$ so that $\Phi_{hk}(\xi)=(h_ik_i)\theta^v$.
    \end{itemize}
    The equality follows since $\theta$ (hence $\theta^v$) is an endomorphism.
    \medskip
    
    (3) We prove that $\iota$ is injective. Let $g=\{(a_i,b_i,h_i)\}$ and suppose that $\iota(g)$ is the neutral element. In particular, $\pi(g)=1_V$ gives $a_i=b_i$ for all $i$, and $\Phi_g(a_i0^\infty)=1_H$ gives $h_i=1_H$. We conclude that $g=1_{V_{(H,\theta)}}$.
\end{proof}
\begin{thm} \label{thm:V_Htheta_is_CFTR}
    If $H$ is finite, then the restriction to $V_{(H,\theta)}$ of the standard wreath product action $(H\times \D) \racts (H\Wr_\D V)$, given by
    \[ (h,\xi)\cdot (\Phi, g) = (h\cdot \Phi(\xi),\, \xi\cdot g), \]
    is context-free. In particular, the group $V_{(H,\theta)}$ embeds in Thompson's $V$.
\end{thm}
\begin{proof}
First, the action is transitive, in particular it has finitely many orbits: fix $h\in H$ and $\xi=v0^\infty\in\D$. We find two partitions
\[ \Ck = [a_1]\sqcup \ldots \sqcup [a_m] = [b_1]\sqcup \ldots \sqcup [b_m] \]
with $a_1=0$ and $b_1=v0$, and consider the element $g=\{(a_i,b_i,h)\}\in V_{(H,\theta)}$. We have $\Phi_g(0^\infty)=h$ and therefore
\[ (1_H,0^\infty) \cdot g = (1_H\cdot \Phi_g(0^\infty),\; 0^\infty\cdot \pi(g)) = (h,\;v00 ^\infty) = (h,\xi). \]

\medskip

We fix a set $\Ac=\{g_1,\ldots,g_k\}$ generating $V_{(H,\theta)}$, with $g_i=\{(a_{ij},b_{ij},h_{ij}):j=1,\ldots,m_i\}$. We construct a pushdown automaton recognising
\[ L=\bigl\{w\in\Ac^*:(1_H,0^\infty)\cdot w=(1_H,0^\infty)\bigr\}. \]
\begin{itemize}[leftmargin=8mm]
    \item The state space is $\mathsf Q=H\sqcup\{\q_2,\accept\}$.
    \item The start state is $\start=1_H$.
    \item The stack alphabet is $\{0_\kappa,1_\kappa,\#:\kappa\in\End(H)\}$. A word $s_{1,\kappa_1}\ldots s_{\ell,\kappa_\ell}(\#)$ with $s_{k,\kappa_k}\in\{0_\kappa,1_\kappa:\kappa\in\End(H)\}$ is \emph{correctly decorated} if 
    \[ \forall k=1,\ldots,\ell-1, \quad \kappa_k = \theta^{s_k}\circ \kappa_{k+1}, \]
    and $\kappa_\ell=\theta^{s_\ell}$ if the word ends in $\#$.
    \item The transitions are
    \begin{itemize}[leftmargin=3mm]
        \item $(h,g_i,\mathbf a c_\kappa, \mathbf b c_\kappa, h\cdot \kappa(h_i))$ if $\mathbf ac_\kappa$ and $\mathbf bc_\kappa$ are correctly decorated versions of $a_{ij}c$ and $b_{ij}c$ for some $c\in\{0,1\}$.
        \item If $a_{ij}=p0^n$ for $n\ge 0$, then we add a transition $(h, g_i, \mathbf p\#,\mathbf b\#,h\cdot h_i)$, where $\mathbf p\#$ and $\mathbf b\#$ are the correctly decorated versions of $p\#$ and $b_{ij}\#$. 
        \item $(1_H,\varepsilon,\varepsilon,\varepsilon,\q_2)$,
        \item $(\q_2,\varepsilon,0_{\id},\varepsilon,\q_2)$,
        \item $(\q_2,\varepsilon,\#,\#,\accept)$.
    \end{itemize}
\end{itemize}
After $t$ first steps, we have read a prefix $w_t$ of $w$, the automaton is at a state $h_t$ and the stack reading $\mathbf x_t\#$, such that $\mathbf x_t\#$ is correctly labelled and
\[ (1_H, 0^\infty) w_t = (h_t, x_t0^\infty). \] This can be proven by induction on $t$. At the end of $w$, the state is $h$ with stack and the stack reads $\mathbf x\#$ where $(1_H,0^\infty)\cdot w=(h,x0 ^\infty)$.

For the \say{verification}, we must accept $w$ if and only if $h=1_H$ and $x0^\infty=0^\infty$ that is $\mathbf x=0_{\id}^n$ for some $n\ge 0$. This is checked moving to $\q_2$ if $h=1_H$, then removing as many copies of $0_{\id}$ as possible and checking if the stack reads $\#$.

\medskip

Since the action is faithful (restriction of the wreath product action), we conclude that $V_{(H,\theta)}$ is \textbf{CF-TR} hence embeds in $V$ by Theorem \ref{thm:CFTR-in-V}.
\end{proof}

\subsection{FSS groups}

In 2009 Hughes \cite{Hughes} introduced \emph{groups defined by finite similarity structures} (or FSS groups) and, lately, Farley \cite{Farley_FFS} proved that a subclass of FSS groups is co-context-free. A finite similarity structure $\Sim_X(-,-)$ on a compact ultrametric space $X$ is a groupoid where objects are open balls in $X$, morphisms are surjective similarities between balls, $\Sim_X(B_1,B_2)$ is finite for any pair of balls $B_1,B_2$ and restriction closed, i.e. $h|_{B_3} \in \Sim_X(B_3,B_3h)$ for any $h \in \Sim_X(B_1,B_2)$ and any ball $B_3 \subseteq B_1$. Furthermore, Farley assumed that there are finitely many similarity classes, where two balls $B_1,B_2$ belong to the same class whenever $\Sim_X(B_1,B_2)$ is non-empty. The FSS group $G(\Sim_X)$ is the set of homeomorphisms of $X$ that are locally determined by $\Sim_X$, more formally $h \in G(\Sim_X)$ if for every $x \in X$ there exists a ball $B'$ such that $x\in B'$, the image $B'h$ is a ball in $X$, and $h|_{B'} \in \Sim(B',B'h)$. 

\medskip

Our goal is to embed these groups $G(\Sym_X)$ inside Thompson's $V$. Two main issues arise to apply Theorem \ref*{thm:intro_main}: first the group $G(\Sym_X)$ might not be finitely generated, and its natural action on $X$ might not have any dense orbit, making it hard to find a finite union of orbits where the action is faithful. Our strategy separates into two steps:
\begin{itemize}[leftmargin=6mm]
    \item First, we prove that these FSS groups embed inside \emph{Röver-Nekrashevych groups on irreducible edge shifts} \cite{Deaconu, ESS_groups} where the groups in the self-similar tuple are finite. 
    \item Second, we prove that these Röver-Nekrashevych groups are \textbf{CF-TR}.
\end{itemize}
Röver-Nekrashevych groups on an irreducible edge shifts generalise the classical ones acting on the full $n$-ary shift, and generated by the Higman-Thompson group $V_n$ and a self-similar group. In this case, we consider the topological full group of the shift $\Omega(\Sigma)$ together with a self-similar tuple $(H_v)_{v \in \Vc}$, meaning that 
\begin{enumerate}[leftmargin=6mm]
    \item $H_v$ are subgroups of $\Isom(\Omega(\Sigma,v))$ (equivalently, subgroups of the automorphisms of the underlying self-similar tree) for each $v \in \Vc$, and
    \item for every $a \in \Ec^*$ describing a path $u\to v$ in $\Sigma$, and for every $h\in H_u$, one has $h|_{a} \in H_{v}$. 
\end{enumerate}
See \cite{ESS_groups} for more precise definitions. The Röver-Nekrashevych group on $\Omega(\Sigma)$ with respect to $(H_v)_{v \in \Vc}$ is the group of the homeomorphisms of $\Omega(\Sigma)$ of the form
\[
\phi: a_j\hspace*{.5pt}\eta \mapsto b_j\hspace*{.5pt}(\eta)h_j \qquad \text{with } \tau(a_j)=\tau(b_j) \ \text{for } j=1,\ldots,m
\]
where the notations are as in the definition of topological full groups (see Section \ref{subsection:CFTR-embedding}) and $h_j\in H_{\tau(a_j)}$. Equivalently, this the subgroup of $\Homeo(\Omega(\Sigma))$ generated by $\F(\Sigma)$ and
\[ H_v@ a_v = \Bigl\{ g \;\Big|\; \supp(g)\subseteq [a_v] \text{ and } g|_{[a_v]} = \sigma_{a_v}^{-1}H_v\sigma_{a_v} \Bigr\} \quad\text{for }v\in\Vc,\] 
where $a_v$ are fixed non-empty words with $\tau(a_v)=v$. In particular, if $\Omega(\Sigma)$ is irreducible and $(H_v)_{v\in\Vc}$ are finite, then $\F(\Sigma)$ and the Röver-Nekrashevych group are finitely generated.

\begin{prop}
Let $G(\Sim_X)$ be a FSS group where $\Sim_X$ has finitely many similarity classes. Then $G(\Sim_X)$ embeds into a Röver-Nekrashevych groups on an irreducible edge shift with a finite group self-similar tuple.
\end{prop}
\begin{proof}
We reuse some notations and definitions of \cite[Section 5.2.4 and 5.4]{Farley_FFS}. that provides a natural alphabet for an edge shift: consider a set of representatives for the similarity classes, thanks to the ultrametric properties each representative $\widetilde{B}_i$ has finitely many maximal proper sub-balls $\widetilde{B}_{i1}, \ldots, \widetilde{B}_{i\ell_i}$ with $i=1, \ldots, k$. Now we define the graph $\Sigma=(\Vc,\Ec)$ where
\begin{enumerate}[leftmargin=6mm]
    \item $\Vc$ is the set of similarity classes, together with an extra vertex $\widetilde{B}_\emptyset$, and
    \item $\Ec$ is the set of sub-balls, so that $\widetilde{B}_i \xrightarrow{\widetilde{B}_{it}} \widetilde{B}_j$ happens whenever $\widetilde{B}_{it}$ is in the similarity class of $\widetilde{B}_j$. We also introduce an edge $\widetilde{B}_i \xrightarrow{\widetilde{B}_{i}} \widetilde{B}_\emptyset$ for each vertex $\widetilde{B}_i$.
\end{enumerate} 
Each ball correspond to a unique path $addr(B)$ in $\Sigma$ starting at $\tilde X$ and not passing by $\tilde B_\emptyset$. The reason to add the vertex $\widetilde B_\emptyset$ and the edges $\widetilde B_i$ is to get an irreducible shift. For each $B$ in the similarity class $\widetilde{B}_i$, we also fix $f_B \in \Sim_X(\widetilde{B}_i,B)$.

Exploiting compactness, if $g \in G(\Sim_X)$, one gets two finite partitions $\{B_j\}_{J}$ and $\{B_jg\}_{J}$ of balls of $X$ such that $g|_{B_j} \in \Sim(B_j,B_jh)$ (see \cite[Definition 5.16]{Farley_FFS}). The permutation induced by such partitions can  be described as an element of $\F(\Sigma)$ \cite[Proposition 5.37]{Farley_FFS}. We define each group $H_{\widetilde{B}_i}$ of the self-similar tuple as $\Sim_X(\widetilde{B}_i, \widetilde{B}_i)$ (this group does not \say{split} balls, hence acts on the tree), and $H_{\widetilde{B}_\emptyset}$ as the trivial group. Thanks to the axioms of similarity structure, these groups satisfy conditions 1 and 2 above. Finally
\[ g|_{B_j} \in f_{B_j}^{-1} H_{\widetilde{B}_i}f_{B_jg} \simeq \sigma_{addr(B_j)}^{-1}H_{\widetilde{B}_i}\sigma_{addr(B_jg)} \]
where $B_j$ and $B_jg$ are in the similarity class of $\widetilde{B}_i$. This provides the desired embedding. 
\end{proof}

We can proceed with the second step.
\begin{prop}
    Röver-Nekrashevych groups on irreducible edge shifts with finite group tuples are \textbf{CF-TR}, and therefore embed in Thompson's $V$. \vspace*{-1mm}
\end{prop}
\begin{proof}
We consider a finite generating set of the form
	\[ \Ac = \Ac_0 \sqcup \bigsqcup_{v\in\Vc} H_v @ a_v \]
	where $\Ac_0=\{g_1,\ldots,g_k\}$ generates the topological full group $\F(\Sigma)$, with $g_i=\{(a_{ij},b_{ij})\mid j=1,\ldots,m_i\}$, and $a_v \in \Ec^*$ are fixed non-empty paths in $\Sigma$ with $\tau(a_v)=v$. Consider a path $u=u_1u_2\ldots u_\ell\in\Ec^*$ from $v_0$ to $v_0$, which is not a proper power. We prove that $\Sch(\xi,G;\Ac)$ is context-free, for $\xi= u^\infty \in\Omega(\Sigma)$. Let $K$ be the stabiliser of $\xi$ in $H_{v_0}$.
	\begin{itemize}[leftmargin=6mm]
		\item The state space is $\mathsf Q= \{\start,\q_1,\q_1',\q_2,\accept\}$.
		\item The start vertex is $\start$.
		\item The stack alphabet is $\Ec\sqcup \bigsqcup_{v\in\Vc} H_v\sqcup\{\#\}$.
	\end{itemize}
	Of special importance will be words of the form $x=e_1s_1\ldots e_ns_n$, where $e_1\ldots e_n$ labels a path in $\Sigma$, and $s_i\in H_{\tau(e_i)}\cup\{\varepsilon\}$, which will be called \emph{legal words}. We describe how to \emph{reduce} legal words. We initialise with $x_1=x$, $p_1=e_1$ and $\bar s_1=s_1$, and then define recursively
	\[ x_k=p_k\bar s_ke_{k+1}s_{k+1}\ldots e_ns_n \longto x_{k+1}= p_{k+1}\bar s_{k+1}\ldots e_ns_n\]
	with $p_{k+1}=p_k\cdot (e_{k+1})\bar s_k$ and $\bar s_{k+1}=s_{k+1}\cdot \bar s_k|_{e_{k+1}}$. At the end of this process, we get a legal word $x_{\mathrm{red}}\coloneqq x_n=p_n\bar s_n$ with $p_n\in\Ec^*$ and $\bar s_n\in H_{\tau(p_n)}$.
	\begin{itemize}[leftmargin=6mm]
		\item The transition function is given by
		\begin{itemize}[leftmargin=3mm]
			\item $(\start,\,\varepsilon,\,\varepsilon,\, u ,\,\start)$
			\item $(\start,\, \varepsilon,\, \varepsilon,\, \varepsilon,\, \q_1)$
			\item $(\q_1,\, g_i,\, a_{ij},\, b_{ij},\, \q_1')$ if $g_i=\{(a_{ij},b_{ij}) \mid j=1,\ldots,m_i\} \in \Ac_0$
			\item $(\q_1,\, h@a_v,\, a_vs,\, a_v(sh),\, \q_1')$ if $h@a_v\in H_v@a_v$
			\item $(\q_1',\, \varepsilon,\, x,\, x_{\mathrm{red}},\, \q_1)$ where $x=e_1s_1\ldots e_ns_n$ is a legal word of length 
			\[ n= \max\{\abs{a_{ij}},\abs{a_v},\abs u\}+1. \]
			\item $(\q_1,\, \varepsilon,\, \varepsilon,\, \varepsilon,\, \q_2)$
			\item $(\q_2,\, \varepsilon,\, u,\, \varepsilon,\, \q_2')$
			\item $(\q_2',\, \varepsilon,\, x,\, x_{\mathrm{red}},\, \q_2)$ where $x=e_1s_1\ldots e_ns_n$ is a legal word of length $n=\abs u$.
			\item $(\q_2',\, \varepsilon,\, s\#,\, \#,\, \accept)$ for $s\in K$.
		\end{itemize}
	\end{itemize}
    We start by pushing a power $u^\ell$ on the stack (non-deterministically). After moving to $\q_1$, we read the word $w\in\Ac^*$. After $2t$ first steps (if $\ell$ was large enough), we have read a prefix $w_t$ of $w$ and the stack reads $x_t\#$ with $(x_t)_\mathrm{red}=p_ts_t$ such that $(u^\ell\cdot \eta)w_t=p_t\cdot (\eta)s_t$. At the end of $w$, the stack reads $x\#$, we need to check if
    $(u^\infty)w = p\cdot (u^\infty)s$
    where $x_{\mathrm{red}}=ps$.
    
    For the \say{verification}, we must accept $w$ if and only if $p\cdot (u^\infty)s=u^\infty$, i.e.\ if $p=u^k$ and $s\in K$. This is done by reducing $x$ gradually, removing copies of $u$, and checking if the stack reads $s\#$ with $s\in K$.

    Recall that $\F(\Sigma)$, and therefore the Röver-Nekrashevych group, acts minimally by homeomorphisms on $\Omega(\Sigma)$ (irreducible edge shift), hence they act faithfully on each orbit. It follows that our Röver-Nekrashevych group is \textbf{CF-TR}. \vspace*{1mm}
\end{proof}

Combining the two results, with get the following theorem.
\begin{thm}
Farley's FSS groups embed in Thompson's $V$. \vspace*{-1mm}
\end{thm}
We believe that a generalised approach for both $V_{(H,\theta)}$ groups and FSS groups is possible, though one would need to introduce cloning systems for edge shifts, which goes beyond the scope of this paper.

%% file: Pictures/tikz_VGtheta.tex
\begin{tikzpicture}[scale = 1.34, very thick]
    \clip (-.5,-1) rectangle (5.5,2);
   \node at (2.5,.77) {$g_1$};
   \draw[-latex] (2.25,.5) -- (2.75,.5);
    
    \draw (0,0) -- (1,1) -- (2,0);
    \node at (0,-.22) {$0$};
    \node at (2,-.22) {$1$};
    
    \draw (3,0) -- (4,1) -- (5,0);
    \node at (3,-.22) {$(1,g)$};
    \node at (5,-.22) {$0$};
    
\end{tikzpicture}
\begin{tikzpicture}[scale = 1.34, very thick]
    \clip (-.5,-1) rectangle (5.5,2);
    \node at (2.5,.77) {$g_2$};
    \draw[-latex] (2.25,.5) -- (2.75,.5);
    
    \draw (0,0) -- (1,1) -- (2,0);
    \draw (0.5,.5) -- (1,0);
    \node at (0,-.22) {$00$};
    \node at (1,-.22) {$01$};
    \node at (2,-.22) {$1$};
    
    \draw (3,0) -- (4,1) -- (5,0);
    \draw (4.5,.5) -- (4,0);
    \node at (3,-.22) {$00$};
    \node at (4,-.22) {$1$};
    \node at (5,-.22) {$(01,h)$};
    
 \end{tikzpicture} 

%% file: Sections/sec4_covers.tex
In this section, we study the relation between $\Gc(\Gamma)$ and $\Gc(\Lambda)$ when we have a covering $\pi\colon \Gamma\onto \Lambda$. In this case, \cite[Proposition 2.5]{CFTR} implies $\Gc(\Gamma)\onto \Gc(\Lambda)$. We give a more precise description using permutational wreath products.

\subsection{General case}

Suppose that $\Gamma$ is connected, so that the action $\Gamma\racts \Gc(\Gamma)$ is transitive. Note that the action is imprimitive, with imprimitivity blocks $B=\pi^{-1}(p)$. By the Krasner-Kaloujnine theorem (Theorem \ref{thm:Krasner_Kaloujnine}), we have
\[ \Gc(\Gamma)\into \Sym(B) \Wr_\Lambda \Gc(\Lambda). \]
We give two conditions under which this statement can be improved. Some computations prefacing these results are found in \cite[Proposition 10.6]{CFTR}.
\begin{prop} \label{prop:cover_Linvariant}
    Suppose that $L\acts \Gamma$ freely by automorphisms, and that the graph it covers is specifically $\Lambda=L\backslash \Gamma$. Then $\Gc(\Gamma)\into L\Wr_\Lambda \Gc(\Lambda)$. \vspace*{-2mm}
\end{prop}
\begin{proof}
Fix $x_0\in\Gamma$ and $B_0=\pi^{-1}(\pi(x_0))$ the associated imprimitivity block. We need to understand the action of the set-wise stabiliser
\[ \SStab_{\Gc(\Gamma)}(B_0) \coloneqq \bigl\{g\in \Gc(\Gamma)\mid B_0\cdot g=B_0 \bigr\} \]
on $B_0$. We identify $B_0$ with $L$ via $lx_0\leftrightarrow l$. We make two observations:
\begin{itemize}[leftmargin=6mm]
    \item Fix $x\in\Gamma$, $g\in\Gc(\Gamma)$ and $l\in L$. Let $w\in\Ac^*$ be a representative of $g$. We have a path $x\overset w\longto xg$. Since $L$ acts by automorphisms, we get another path $lx\overset w\longto l(xg)$. However, the end-point of the path starting at $lx$ and labelled by $w$ is by definition $(lx)g$. This proves that $l(xg)=(lx)g$.
    \item For every $g\in\SStab_{\Gc(\Gamma)}(B_0)$, there exists $l_g\in L$ such that $x_0g=l_gx_0$. Reciprocally, for every $l\in L$, there exists $g\in \Gc(\Gamma)$ such that $lx_0=x_0g$ (since $\Gamma$ is connected), and this $g$ belongs to $\SStab_{\Gc(\Gamma)}(B_0)$.
\end{itemize}
Finally, we look at the action of $g\in\SStab_{\Gc(\Gamma)}$ on $B_0$. For every $lx_0\in B_0$, we have $(lx_0)g=l(x_0g)=ll_gx_0$, i.e. $\SStab_{\Gc(\Gamma)}(B_0)$ acts by right multiplication, hence its image in $\Sym(B_0)$ is isomorphic to $L$.
\end{proof}

Another interesting case is when the sheets of the covering are glued along a \say{thin spine} in following sense: there exists a finite set $F\subseteq E(\Lambda)$ such that $\Lambda-F$ is connected and the restriction of $\pi|_C\colon C\to \Lambda-F$ is an isomorphism for each connected component $C$ of  $\Gamma-\pi^{-1}(F)$ (see Figure \ref{fig:spine}).
\begin{prop} \label{prop:cover_restricted}
    If $\Gamma\onto\Lambda$ has a thin spine, then $\Gc(\Gamma)\into \Sym(B)\wr_\Lambda\Gc(\Lambda)$.\vspace*{-1mm}
\end{prop}
\begin{proof}
Consider $g\in\Gc(\Gamma)$ such that $g$ fix the fibres of $\Gamma\onto\Lambda$ set-wise (i.e.\ $g$ maps to $1$ in $\Gc(\Lambda)$). We fix a word $w\in\Ac^*$ representing $g$.

For each $x\in\Gamma$, either the path $x\overset w\longto xg$ is included in $C_x$ the connected component of $\Gamma-\pi^{-1}(F)$ containing $x$, in which case
\[ xg \in C_x \cap \pi^{-1}(\pi(x)) = \{x\},\]
i.e.\ $xg=x$, or uses an edge from $\pi^{-1}(F)$. Therefore the support of $g$ satisfies
\[ \supp(g) \subseteq \bigcup_{i=0}^\ell \pi^{-1}(\partial F)\cdot w_i^{-1} = \bigcup_{p\in S_g} \pi^{-1}(p) \]
where $\partial F$ is the set of vertices incident to $F$, the words $w_0,w_1,\ldots,w_\ell\in\Ac^*$ are all the prefixes of $w$, and $S_g=\bigcup_{i=0}^\ell \partial F\cdot w_i^{-1}$. The statement now follows from Theorem \ref{thm:Krasner_Kaloujnine}.
\end{proof}
\begin{rem} \
    \begin{itemize}[leftmargin=6mm]
        \item If $\Gamma$ is context-free, then $L\acts\Gamma$ freely implies that $L$ is virtually free.
        \item If both conditions hold at the same time, then $\Gc(\Gamma)\into L\wr_\Lambda\Gc(\Lambda)$.
    \end{itemize}
\end{rem}
\begin{center}
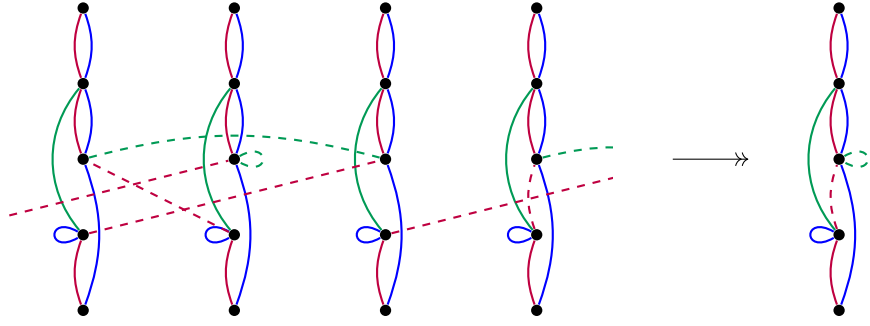

    \import{Pictures/}{tikz_spine.tex}
    \captionsetup{font=small}
    \captionof{figure}{A covering $\pi\colon \Gamma\onto\Lambda$ with a thin spine. $F$ and $\pi^{-1}(F)$ in dashes.}
    \label{fig:spine}
\end{center}

\subsection{Preserving context-freeness} \label{sec:cover_CF}

Suppose that $\pi\colon\Gamma\onto\Lambda$ has a thin spine, as defined in the previous section. We construct a \say{spine} graph $\Sp(\Gamma,\Lambda)$ that fully encodes the covering:
\begin{itemize}[leftmargin=6mm]
    \item The vertex set is $\pi^{-1}(\partial F)$, recall that $\partial F$ is the set of vertices incident to $F$.
    \item Fix a total order on $\partial F$. We define an extended alphabet
    \[ \tilde\Ac=\Ac\sqcup\bigl\{k_{p,q}, k_{p,q}^{-1}\mid p,q\in\partial F,\; p\le q\bigr\}. \]
    \item The edge set consists of two types of edges:
    \begin{itemize}[leftmargin=4mm]
        \item Edges from $\pi^{-1}(F)$ with their original labels in $\Ac$, and
        \item If $x,y$ belong to the same component of $\Gamma-\pi^{-1}(F)$ and $\pi(x)\le \pi(y)$, we add an edge $x\to y$ (resp.\ $y\to x$) labeled by $k_{\pi(x),\pi(y)}$ (resp.\ $k_{\pi(x),\pi(y)}^{-1}$).
    \end{itemize}
\end{itemize}

\medskip

We characterise when such covers are context-free:
\begin{prop} \label{prop:cover_is_CF}
    Let $\Gamma,\Lambda$ be inverse $\Ac$-graphs as above. Suppose that $\Gamma,\Lambda$ are connected and complete. The following assertions are equivalent:
    \begin{enumerate}[leftmargin=8mm, label={\normalfont(\alph*)}]
        \item $\Gamma$ is context-free, and
        \item $\Sp(\Gamma,\Lambda)$ and $\Lambda$ are context-free.
    \end{enumerate}
\end{prop}
\begin{proof}
    (a) $\Rightarrow$ (b): Fix $x_0\in \Sp(\Gamma,\Lambda)$. The paths in $\Sp(\Gamma,\Lambda)$ starting at $x_0$ and ending in the same fiber $\pi^{-1}(\pi(x_0))$ form a regular language \[ R(\Sp(\Gamma,\Lambda),x_0)\subseteq\tilde\Ac^* \] recognised by the following automaton:
    \begin{itemize}[leftmargin=6mm]
        \item The vertex set is $\partial F$,
        \item The edge set contains edges from $F$ with their original labels, together with a complete graph (with loops) on $\partial F$ labelled by $k_{p,q}^\pm$,
        \item The unique start and accept vertex is $\pi(x_0)$.
    \end{itemize}
    For each $p, q\in\partial F$ with $p\le q$, we fix a word $w_{p,q}\in\Ac^*$ labelling a path $p\to q$ in $\Lambda-F$. We define a homomorphism $\psi\colon\tilde\Ac^*\to\Ac^*$ sending each $a\in\Ac$ to itself, and $k_{p,q}^{\pm}$ to $w_{p,q}^{\pm}$.
    
    To prove that $\Sp(\Gamma,\Lambda)$ is context-free, we observe that
    \[ L\bigl(\Sp(\Gamma,\Lambda),x_0\bigr)=\psi^{-1}\Bigl(L(\Gamma,x_0)\cap \psi\bigl(R(\Sp(\Gamma,\Lambda),x_0)\bigl) \Bigr),\]
    which is context-free since $L(\Gamma,x_0)$ is context-free, and context-free language are closed under intersection with regular languages and inverse image by homomorphisms. (Note that $\Sp(\Gamma,\Lambda)$ is connected since $\Gamma$ is connected.)
    
    To prove that $\Lambda$ is context-free, fix $p_0\in\partial F$ and $x_0\in\pi^{-1}(p_0)$. Observe that
    \begin{align*}
        L(\Lambda,p_0)
        & = \bigl\{ w\in\Ac^* \mid p_0w=p_0 \text{ in }\Gc(\Lambda)\bigr\} \\
        & = \bigl\{ w\in\Ac^* \mid x_0w\in\pi^{-1}(p_0) \text{ in }\Gc(\Gamma)\bigr\}.
    \end{align*} 
    By definition, this language is context-free if and only if
    \[ \bigl\{g\in\Gc(\Gamma)\mid x_0g\in\pi^{-1}(p_0)\bigr\} = R\cdot \Stab_{\Gc(\Gamma)}(x_0) \le\Gc(\Gamma) \]
    is \emph{recognisably context-free}, where $R$ is the image of $R(\Sp(\Gamma,\Lambda),x_0)$ in $\Gc(\Gamma)$ (after applying $\psi$ and evaluating). This follows from \cite[Lemma 4.1]{Herbst1991}, since $R$ is rational and $\Stab_{\Gc(\Gamma)}(x_0)\le\Gc(\Gamma)$ is recognisably context-free.
    
    \medskip
    
   (b) $\Rightarrow$ (a): Fix $x_0\in \pi^{-1}(\partial F)$. The languages $L(\Sp(\Gamma,\Lambda),x_0)\subseteq\tilde\Ac^*$ and
   \[ K_{p,q} = \bigl\{w\in\Ac^* : p\overset w\longto q\text{ in }\Lambda- F \bigr\} \]
   are context-free, they are produced by grammars. We consider the union of these grammars, where we now consider the letter $k_{p,q}^\pm$ as non-terminals, and add production rules replacing $k_{p,q}$ (resp.\ $k_{p,q}^{-1}$) by the starting variable of the grammar producing $K_{p,q}$ (resp.\ $K_{q,p}$). This new grammar produces $L(\Gamma,x_0)\subseteq\Ac^*$ since every closed path in $\Gamma$ can be decomposed as a closed path in $\Sp(\Gamma,\Lambda)$, with each edge $k_{p,q}^\pm$ replaced by an excursion in a connected component of $\Gamma-\pi^{-1}(F)$. We conclude that $\Gamma$ is context-free.
\end{proof}

\medskip

\begin{rem}
    Note that we can drop the assumption \say{complete}, in which case $\Gc(\Gamma)$ is replaced by the inverse monoid $\Ic(\Gamma)$. The proof works verbatim.
\end{rem}
\begin{rem}
    For general coverings, we can always find a set $F\subseteq E(\Lambda)$ satisfying that $\Lambda-F$ is connected and the restriction of $\pi$ to each connected component of $\Gamma-\pi^{-1}(F)$ is an isomorphism. For instance, take $F$ the set of edges outside a spanning tree of $\Lambda$. However this set will not usually be finite (see eg.\ Figure \ref{fig:lamplighter}). This leads to the following question: \vspace*{1mm}
\begin{ques}
    What is the correct analogue of Proposition \ref*{prop:cover_is_CF} in the case that $F$ is infinite? The case of finite sheeted coverings is of special interest. \vspace*{-1mm}
\end{ques}
The result should be sufficiently general to treat the case of the Schreier graph $\Sch(H\times\D, V_{(H,\theta)})$ covering the graph $\Sch(\D,V)$.
\end{rem}


%% file: Pictures/tikz_spine.tex
\begin{tikzpicture}[every node/.style={circle, fill, inner sep=1.5pt}]
    
    \begin{scope}[shift={(0,0)}]
        \node[circle, fill, inner sep=1.5pt] (a0) at (0,1) {};
        \node[circle, fill, inner sep=1.5pt] (b0) at (0,2) {};
        \node[circle, fill, inner sep=1.5pt] (c0) at (0,3) {};
        \node[circle, fill, inner sep=1.5pt] (d0) at (0,4) {};
        \node[circle, fill, inner sep=1.5pt] (e0) at (0,5) {};

        \draw[purple, thick, bend left=20] (a0) to (b0);
        \draw[purple, thick, bend left=20] (c0) to (d0);
        \draw[purple, thick, bend left=20] (d0) to (e0);

        \draw[blue, thick, bend right=20] (a0) to (c0);
        \draw[blue, thick, bend right=20] (c0) to (d0);
        \draw[blue, thick, bend right=20] (d0) to (e0);
        \draw[blue, thick, loop, looseness=15, out=150, in=-150] (b0) to (b0);

        \draw[ForestGreen, thick, bend left=40] (b0) to (d0);
    \end{scope}
    
    \begin{scope}[shift={(2,0)}]
        \node[circle, fill, inner sep=1.5pt] (a1) at (0,1) {};
        \node[circle, fill, inner sep=1.5pt] (b1) at (0,2) {};
        \node[circle, fill, inner sep=1.5pt] (c1) at (0,3) {};
        \node[circle, fill, inner sep=1.5pt] (d1) at (0,4) {};
        \node[circle, fill, inner sep=1.5pt] (e1) at (0,5) {};

        \draw[purple, thick, bend left=20] (a1) to (b1);
        \draw[purple, thick, bend left=20] (c1) to (d1);
        \draw[purple, thick, bend left=20] (d1) to (e1);

        \draw[blue, thick, bend right=20] (a1) to (c1);
        \draw[blue, thick, bend right=20] (c1) to (d1);
        \draw[blue, thick, bend right=20] (d1) to (e1);
        \draw[blue, thick, loop, looseness=15, out=150, in=-150] (b1) to (b1);

        \draw[ForestGreen, thick, bend left=40] (b1) to (d1);
    \end{scope}
    
    \begin{scope}[shift={(4,0)}]
        \node[circle, fill, inner sep=1.5pt] (a2) at (0,1) {};
        \node[circle, fill, inner sep=1.5pt] (b2) at (0,2) {};
        \node[circle, fill, inner sep=1.5pt] (c2) at (0,3) {};
        \node[circle, fill, inner sep=1.5pt] (d2) at (0,4) {};
        \node[circle, fill, inner sep=1.5pt] (e2) at (0,5) {};

        \draw[purple, thick, bend left=20] (a2) to (b2);
        \draw[purple, thick, bend left=20] (c2) to (d2);
        \draw[purple, thick, bend left=20] (d2) to (e2);

        \draw[blue, thick, bend right=20] (a2) to (c2);
        \draw[blue, thick, bend right=20] (c2) to (d2);
        \draw[blue, thick, bend right=20] (d2) to (e2);
        \draw[blue, thick, loop, looseness=15, out=150, in=-150] (b2) to (b2);

        \draw[ForestGreen, thick, bend left=40] (b2) to (d2);
    \end{scope}

    \begin{scope}[shift={(6,0)}]
        \node[circle, fill, inner sep=1.5pt] (a3) at (0,1) {};
        \node[circle, fill, inner sep=1.5pt] (b3) at (0,2) {};
        \node[circle, fill, inner sep=1.5pt] (c3) at (0,3) {};
        \node[circle, fill, inner sep=1.5pt] (d3) at (0,4) {};
        \node[circle, fill, inner sep=1.5pt] (e3) at (0,5) {};

        \draw[purple, thick, bend left=20] (a3) to (b3);
        \draw[purple, thick, bend left=20] (c3) to (d3);
        \draw[purple, thick, bend left=20] (d3) to (e3);

        \draw[blue, thick, bend right=20] (a3) to (c3);
        \draw[blue, thick, bend right=20] (c3) to (d3);
        \draw[blue, thick, bend right=20] (d3) to (e3);
        \draw[blue, thick, loop, looseness=15, out=150, in=-150] (b3) to (b3);

        \draw[ForestGreen, thick, bend left=40] (b3) to (d3);
    \end{scope}

    \begin{scope}
        \clip (-1,.8) rectangle (7,5.2);
        \draw[purple, thick, dashed] (b0) to (c2);
        \draw[purple, thick, dashed] (b1) to (c0);
        \draw[purple, thick, dashed] (b2) to (8,3);
        \draw[purple, thick, bend left=20, dashed] (b3) to (c3);
        \draw[purple, thick, dashed] (-2,2) to (c1);

        \draw[ForestGreen, thick, bend left=15, dashed] (c0) to (c2);
        \draw[ForestGreen, thick, loop, looseness=15, out=30, in=-30, dashed] (c1) to (c1);
        \draw[ForestGreen, thick, bend left=15, dashed] (c3) to (8,3);
    \end{scope}
    
    \begin{scope}[shift={(10,0)}]
        \node[circle, fill, inner sep=1.5pt] (a) at (0,1) {};
        \node[circle, fill, inner sep=1.5pt] (b) at (0,2) {};
        \node[circle, fill, inner sep=1.5pt] (c) at (0,3) {};
        \node[circle, fill, inner sep=1.5pt] (d) at (0,4) {};
        \node[circle, fill, inner sep=1.5pt] (e) at (0,5) {};

        \draw[purple, thick, bend left=20] (a) to (b);
        \draw[purple, thick, dashed, bend left=20] (b) to (c);
        \draw[purple, thick, bend left=20] (c) to (d);
        \draw[purple, thick, bend left=20] (d) to (e);

        \draw[blue, thick, bend right=20] (a) to (c);
        \draw[blue, thick, bend right=20] (c) to (d);
        \draw[blue, thick, bend right=20] (d) to (e);
        \draw[blue, thick, loop, looseness=15, out=150, in=-150] (b) to (b);

        \draw[ForestGreen, thick, bend left=40] (b) to (d);
        \draw[ForestGreen, thick, dashed, loop, looseness=15, out=30, in=-30] (c) to (c);
    \end{scope}

    \draw[->>] (7.8,3) -- (8.8,3);
\end{tikzpicture}

%% file: Sections/sec5_limits.tex
The goal of this section is to replace a context-free graph $\Gamma$ by some set of \say{nicer} graphs, specifically $\Z$-invariant graphs, while keeping control on the associated transition groups. A key idea to get these nicer graphs is to take Chabauty limits of our starting graph in \say{periodic directions}.

\subsection{General case}

In this section, we prove the following result generalising \cite[\S8]{CFTR}:
\begin{prop} \label{prop:cutting}
    Let $\Gamma$ be a connected $\Ac$-graph and $F\subset V\Gamma$ be a finite set. Let $\Gamma^{(1)}$, ..., $\Gamma^{(k)}$ be the connected components of $\Gamma-F$. Consider connected complete $\Ac$-graphs $\Lambda^{(i)}$ such that $\Gamma^{(i)}\into \Lambda^{(i)}\preceq \Gamma$. Then
    \[ 1\longto \Bsc_F\longto \Gc(\Gamma)\longto H_F\longto 1 \]
    where $H_F$ is the image of the diagonal map from $F_\Ac\to \bigoplus_{i=1}^k \Gc(\Lambda^{(i)})$, and $\mathscr B_F$ is locally finite with $\ord(h)\le g\bigl(\abs F (\abs h_\Ac+1)\bigr)$ for all $h\in \mathscr B_F$, where $g$ is the Landau's function.
\end{prop}

Before proving the result, we start with a lemma of independent interest.

\begin{lemma} \label{lem:weak_contain_implies_quotient}
    Let $\Gamma,\Lambda$ be connected $\Ac$-graphs such that $\Gamma\succeq \Lambda$, then \[ \Gc(\Gamma)\onto \Gc(\Lambda). \]
\end{lemma}
\begin{proof}
Fix $y_0\in \Lambda$. We have $(\Lambda,y_0)\in\overline{\{(\Gamma,x):x\in V(\Gamma)\}}$. Let $w\in L(\Gamma)$ and $y\in V(\Lambda)$. We define $R=d_\Lambda(y_0,y)+\abs{w}_\Ac$. By assumption, there exists $x\in V(\Gamma)$ and a (rooted, labelled) isomorphism
\[ \psi\colon \bigl(D_R^\Lambda(y_0),y_0\bigr) \overset\sim\longto \bigl(D_R^\Gamma(x),x\bigr).\]
Since $w\in L(\Gamma)$, it labels a cycle $\psi(y)\to \psi(y)$ in $\Gamma$. Moreover this cycle is fully included in $D_R^\Gamma(x)$ as $R$ is large enough. We deduce that $w$ labels a cycle $y\to y$ in $\Lambda$. Since the argument holds for any $y\in V(\Lambda)$, we conclude that $w\in L(\Lambda)$. It follows that
$$\ker(F_\Ac\onto\Gc(\Gamma))=L(\Gamma)\subseteq L(\Lambda)=\ker(F_\Ac\onto\Gc(\Lambda)),$$
i.e. $\Gc(\Gamma)\onto\Gc(\Lambda)$ via the natural map.
\end{proof}

\medskip

We can now proceed with the proof:
\begin{proof}[Proof of Proposition \ref{prop:cutting}]
    Our last Lemma \ref*{lem:weak_contain_implies_quotient} implies that the epimorphism $F_\Ac\onto \Gc(\Lambda^{(i)})$ descends to $\Gc(\Gamma)\onto\Gc(\Lambda^{(i)})$. Therefore, we have a well-defined morphism $\Gc(\Gamma)\to H_F$, it only remains to study its kernel $\Bsc_F$.

    Let $h\in \Bsc_F$ and consider $w\in\Ac^*$ a representative for $h$ with $\abs w=\abs h_\Ac=\ell$. For each $x\in \Gamma$, either the path $x\overset w\longto xh$ is fully included in one of the $\Gamma^{(i)}\into \Lambda^{(i)}$, in which case $x=xh$ since $w=1$ in $\Gc(\Lambda^{(i)})$, or it intersects $F$. It follows that the support of $h$ is finite, and more precisely
    \[ \supp(h) \subseteq \bigcup_{i=0}^\ell Fw_i^{-1},\]
    where $w_0,w_1,\ldots,w_\ell$ are the prefixes of $w$. If $h_1,h_2,\ldots,h_m\in\Bsc_F$, then
    \[ \la h_1,h_2,\ldots,h_m\ra \le \Sym\left(\bigcup_{i=1}^m \,\supp(h_i)\right) \]
    is finite. Finally, we have $\ord(h)\le g\bigl(\abs{\supp(h)}\bigr) \le g\bigl(\abs F(\abs h_\Ac+1)\bigr)$.
\end{proof}

\subsection{Preserving context-freeness} \label{sec:limits_CF}

When $\Gamma$ is context-free, we explain how to construct a set $F$ and graphs $\Lambda^{(i)}$ as in Proposition \ref{prop:cutting}, with some additional properties. We do not assume that $\Gamma$ is complete.

\medskip

We isolate special end-cone types which allow to define \emph{periodic limits}:
\begin{defi} \label{prop:periodic_limit:fourre_tout}
    An end-cone type $\Gamma_i$ \emph{contains itself} if there exists $y,z\in \Gamma$ such that $\Gamma(y,x_0)\supsetneq\Gamma(z,x_0)$ and both end-cones have type $\Gamma_i$.
\end{defi}

\begin{prop} \label{prop:periodic_limits}
Let $\Gamma_i$ be an end-cone that contains itself. There exists an $\Ac$-graph $\vec\Gamma_i$ which is $\Z$-invariant, context-free, and satisfies $\Gamma_i\into \vec\Gamma_i \preceq \Gamma_i$. Moreover, if $\Gamma_i$ is complete at every vertex $v\notin\Delta_i$, then $\vec\Gamma_i$ is complete.
\end{prop}

\begin{proof}
    We have $y,z\in V\Gamma$ such that $\Gamma_i= \Gamma(y,x_0) \supsetneq \Gamma(z,x_0)$, and there exists an end-cone isomorphism $\psi_i \colon \Gamma(y,x_0) \to \Gamma(z,x_0)$. Up to replacing $\tilde z = \psi_i^2(y)$ and $\tilde\psi_i = \psi_i^2$, we may assume that $\psi_i(y)=z$ and $d(z,x_0)-d(y,x_0)\ge 2$.
    
    We first define the graph $\vec\Gamma_i$. We consider $Q = \Gamma(y,x_0) - \Gamma(z,x_0)$, and define $V(\vec\Gamma_i)=Q\times\Z$, together with three types of edges: \begin{itemize}[leftmargin=6mm]
    \item $(p,n)\overset a\longto (q,n)$ if $p\overset a\longto q$ in $\Gamma_i$,
    \item $(p,n)\overset a\longto (q,n+1)$ if $p\overset{a}\longto\psi_i(q)$ in $\Gamma_i$,
    \item $(p,n)\overset a\longto (q,n-1)$ if $\psi_i(p)\overset{a}\longto q$ in $\Gamma_i$.
\end{itemize}
By construction, we have a $\Z$-action $m\cdot (q,i)\coloneqq (q,m+i)$. Moreover, the map $\Psi_i\colon Q\times\Z_{\ge 0}\to \Gamma_i \colon (q,n)\mapsto \psi_i^n(q)$ is a graph isomorphism, hence $\Gamma_i\into\vec\Gamma_i$. Let us prove that $\vec\Gamma_i\preceq \Gamma_i$: for each $R\ge 0$, we have rooted isomorphisms
\[ \Bigl( D_R^{\vec\Gamma_i} (y,0), (y,0) \Bigr)
\simeq \Bigl( D_R^{\vec\Gamma_i} (y,R), (y,R) \Bigr)
\simeq \Bigl( D_R^{\Gamma_i} \bigl(\psi_i^R(y)\bigr), \psi_i^R(y) \Bigr) \]
using the two previous observations. Completeness follows similarly since
\[
D_1^{\vec\Gamma_i}(q,n)\simeq D_1^{\vec\Gamma_i}(q,1)\simeq D_1^{\Gamma_i}(\psi_i(q)) \quad\text{and}\quad \psi_i(q)\notin\Delta_i.
\]
Finally, we prove that $\vec\Gamma_i$ is context-free. We look at $\Z\backslash\vec\Gamma_i$. Let $u\in\Ac^*$ be a word labelling a path $y\overset u\longto z$ included in $\Gamma(y,x_0)$. A word $w\in\Ac^*$ belongs to $L(\Z\backslash\vec\Gamma_i,y)$ if and only if there exist $m,n\in\Z$ such that $(y,m)\overset w\longto (y,n)$.
Using $\Z$-invariance, we may suppose that $m,n\ge 0$. Therefore we have
\begin{align*}
    w\in L(\Z\backslash\vec\Gamma_i,y)
    & \iff \exists m,n\ge 0,\quad u^m w(u^{-1})^n \in L(\Gamma_i,y) \\
    & \iff w\in \{u\}^*\backslash L(\Gamma_i,y) /\{u^{-1}\}^*,
\end{align*}
where the quotient of two languages $K,L\subseteq\Ac^*$ is defined as
\[ K\backslash L = \bigl\{w\in \Ac^* \;\big|\; \exists v\in K,\; vw\in L \bigr\}. \]
Using that the quotient of a context-free language by a regular language is context-free (folklore, see eg.\ \cite[Problem 2.20]{S2013}), we conclude that $L(\Z\backslash\vec\Gamma_i,y)$ is context-free, i.e.\ the quotient graph $\Z\backslash\Gamma_i$ is context-free.

We use Proposition \ref{prop:cover_is_CF} to conclude. Consider the set $F\subset E(\Z\backslash\vec\Gamma_i)$ which are projections of edges $(p,0)\overset a\longto (q,1)$ in $\vec\Gamma_i$. By construction $q\in\Delta_i$, hence $F$ is finite: $\vec\Gamma_i$ has a thin spine. We also define
\[ \nabla_i \coloneqq \bigl\{p\in Q \;\big|\; \exists q\in Q,\; (p,0)\overset a\longto (q,1) \bigr\}. \]
Since $d(z,x_0)-d(y,x_0)\ge 2$, we have $\Delta_i\cap\nabla_i=\emptyset$. Therefore the spine graph $\Sp(\vec\Gamma_i,\Z\backslash\vec\Gamma_i)$ has a very specific structure, described in Figure \ref*{fig:spine_graph_for_limit}.
\begin{center}
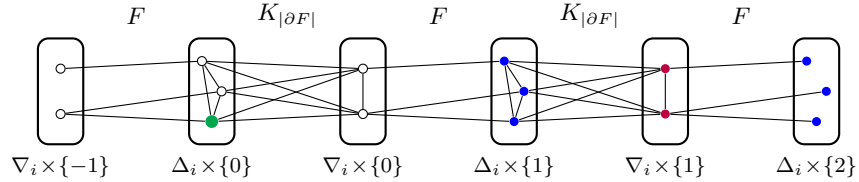

    \import{Pictures/}{tikz_spine_limit.tex}
    \captionsetup{font=small}
    \captionof{figure}{$\Sp(\vec\Gamma_i, \Z\backslash\vec\Gamma_i)$, formed of copies of complete graphs and $F$.}
    \label{fig:spine_graph_for_limit}
\end{center}
$\Sp(\vec\Gamma_i,\Z\backslash\vec\Gamma_i)$ has finitely many end-cone types,  which concludes.
\end{proof}

\begin{rem}
    We call $\vec\Gamma_i$ a \emph{periodic limit} of $\Gamma_i$. The graph $\vec\Gamma_i$ depends on the choice of $y,z$, however we will not record this dependence in the notation.
\end{rem}

Going back to implementing Proposition \ref{prop:cutting}, if we had $F\subseteq V\Gamma$ finite such that all connected components of $\Gamma-F$ were end-cones containing themselves, we could take $\Lambda^{(i)}=\vec\Gamma_i$. We explain how to construct such a set $F$:

\begin{prop} \label{prop:find_F_for_CF}
    Let $\Gamma$ be a context-free graph. Let
    \[ W = \bigl\{ v\in V\Gamma \;\big|\; \Gamma(v,x_0) \text{ contains itself } \bigr\}. \]
    Then $F=V\Gamma-\bigcup_{v\in W} \Gamma(v,x_0)$ is finite, and the connected components of $\Gamma-F$ are end-cones that contain themselves.
\end{prop}
\begin{proof}
    We prove that $F\subseteq D_\Gamma(x_0,N)$, where $N$ is the number of end-cone types. We argue by contraposition and take $y\in V\Gamma$ with $\ell\coloneqq d(y,x_0)\ge N$. Consider a geodesic path
    \[ x_0=v_0\sim  v_1 \sim \ldots \sim v_\ell =y. \]
    By the pigeonhole principle, there exist $0\le i<j\le \ell$ such that $\Gamma(v_i,x_0)$ and $\Gamma(v_j,x_0)$ have the same type. By construction $\Gamma(v_i,x_0)\supsetneq\Gamma(v_j,x_0)$, so that $\Gamma(v_i,x_0)$ contains itself, hence $y\in\Gamma(v_i,x_0)$ does not belong to $F$.
\end{proof}
\begin{rem}
We can read $F$ explicitly off the graph of end-cone types $\mathcal D(\Gamma)$.
\end{rem}

%% file: Pictures/tikz_spine_limit.tex
\begin{tikzpicture}
        {\scriptsize
        \node at (-2,0) {$\nabla_i\!\times\!\{-1\}$};
        \node at (0,0) {$\Delta_i\!\times\!\{0\}$};
        \node at (2,0) {$\nabla_i\!\times\!\{0\}$};
        \node at (4,0) {$\Delta_i\!\times\!\{1\}$};
        \node at (6,0) {$\nabla_i\!\times\!\{1\}$};
        \node at (8,0) {$\Delta_i\!\times\!\{2\}$};}

        {\footnotesize
        \node at (-1,2) {$F$};
        \node at (1,2) {$K_{\abs{\partial F}}$};
        \node at (3,2) {$F$};
        \node at (5,2) {$K_{\abs{\partial F}}$};
        \node at (7,2) {$F$};}

        \draw[thick, rounded corners] (-2.3,.3) rectangle (-1.7,1.7);
        \draw[thick, rounded corners] (-.3,.3) rectangle (.3,1.7);
        \draw[thick, rounded corners] (1.7,.3) rectangle (2.3,1.7);
        \draw[thick, rounded corners] (3.7,.3) rectangle (4.3,1.7);
        \draw[thick, rounded corners] (5.7,.3) rectangle (6.3,1.7);
        \draw[thick, rounded corners] (7.7,.3) rectangle (8.3,1.7);

        \node[draw, circle, inner sep=1.2pt] (cm1) at (-2,.7) {};
        \node[draw, circle, inner sep=1.2pt] (dm1) at (-2,1.3) {};
        
        \node[fill=Green, circle, inner sep=1.7pt] (y0) at (0,.6) {};
        \node[draw, circle, inner sep=1.2pt] (a0) at (0.13,1) {};
        \node[draw, circle, inner sep=1.2pt] (b0) at (-.13,1.4) {};

        \node[draw, circle, inner sep=1.2pt] (c0) at (2,.7) {};
        \node[draw, circle, inner sep=1.2pt] (d0) at (2,1.3) {};

        \node[fill=blue, circle, inner sep=1.2pt] (y1) at (4,.6) {};
        \node[fill=blue, circle, inner sep=1.2pt] (a1) at (4.13,1) {};
        \node[fill=blue, circle, inner sep=1.2pt] (b1) at (3.87,1.4) {};

        \node[fill=purple, circle, inner sep=1.2pt] (c1) at (6,.7) {};
        \node[fill=purple, circle, inner sep=1.2pt] (d1) at (6,1.3) {};
        
        \node[fill=blue, circle, inner sep=1.2pt] (y2) at (8,.6) {};
        \node[fill=blue, circle, inner sep=1.2pt] (a2) at (8.13,1) {};
        \node[fill=blue, circle, inner sep=1.2pt] (b2) at (7.87,1.4) {};

        \draw (cm1) -- (y0);
        \draw (cm1) -- (a0);
        \draw (dm1) -- (b0);

        \draw (y0) -- (a0);
        \draw (y0) -- (b0);
        \draw (y0) -- (c0);
        \draw (y0) -- (d0);
        \draw (a0) -- (b0);
        \draw (a0) -- (c0);
        \draw (a0) -- (d0);
        \draw (b0) -- (c0);
        \draw (b0) -- (d0);
        \draw (c0) -- (d0);

        \draw (c0) -- (y1);
        \draw (c0) -- (a1);
        \draw (d0) -- (b1);

        \draw (y1) -- (a1);
        \draw (y1) -- (b1);
        \draw (y1) -- (c1);
        \draw (y1) -- (d1);
        \draw (a1) -- (b1);
        \draw (a1) -- (c1);
        \draw (a1) -- (d1);
        \draw (b1) -- (c1);
        \draw (b1) -- (d1);
        \draw (c1) -- (d1);

        \draw (c1) -- (y2);
        \draw (c1) -- (a2);
        \draw (d1) -- (b2);
    \end{tikzpicture}

%% file: Sections/sec6_poly.tex
\subsection{Graphs of linear growth}

In the case of linear growth, we can provide refinement of Theorem \ref{thm:intro_main}, taking as \say{universal groups} the family of Houghton's groups $H_m$ instead of Thompson's group $V$.

\begin{thm} \label{thm:Houghton_universal} Let $G$ be a finitely generated group. The following are equivalent:
\begin{enumerate}[leftmargin=8mm, label={\normalfont(\alph*)}]
    \item $G=\Gc(\Gamma^{(1)}\sqcup\ldots\sqcup\Gamma^{(k)})$ for some context-free  graphs $\Gamma^{(i)}$ of linear growth.
    \item $G$ admits a finite-index subgroup which embeds in a Houghton group $H_m$.
\end{enumerate}
\end{thm}
\begin{proof}
(a) $\Rightarrow$ (b): Since $H_{m_1}\times\ldots\times H_{m_k}\le H_{m_1+\cdots+m_k}$, it suffices to prove the statement for $\Gc(\Gamma)$ with a $\Gamma$ context-free graph. Moreover, if $\Gamma$ is finite, then $\Gc(\Gamma)$ is a finite group and hence embeds in $H_1\simeq \FSym(\N)$. We may therefore assume that $\Gamma$ is infinite.

We reuse the language of well-formed words $\Fc(\Gamma)\subset\Ec^*$ (defined in \cite[\S7]{CFTR}, and used in Proposition \ref{prop:embedding-tfg}). Each word of $w_x\in\Fc(\Gamma)$ corresponds to a vertex $x\in V\Gamma$ with $\abs{w_x}=d(x,x_0)+1$. Since $\Gamma$ has linear growth, the regular language $\Fc(\Gamma)$ also has linear growth, hence it must be of a very specific form (see \cite{Tits_for_languages} or Lemma \ref{lem:growth_of_CF_graphs}).

Precisely, the automaton for $\Fc(\Gamma)$ has $c$ cycles which are not accessible from each other. We pick one state $\q_i$ on each cycle and define
\begin{itemize}[leftmargin=6mm]
    \item $P_i$ the set of words labelling \emph{simple} paths from the start vertex $v_0$ to $\q_i$
    \item $u_i$ the word labelling the cycle, read starting at $\q_i$,
    \item $Q_i$ the set of words labelling \emph{simple} paths from $\q_i$ to accept vertices, and
    \item $R$ the set of paths from $v_0$ to an accept vertex which do not pass through any $\q_i$.
\end{itemize}
Observe that the automaton recognising $\Fc(\Gamma)$ is based on $\Dc(\Gamma)$, hence each label is used exactly once. In particular, the words $u_i\in\Ec^*$ are not proper powers, and do not share any common letter. Moreover, the sets $P_i,Q_i,R\subset\Ec^*$ are finite. We can decompose
\[ \Fc(\Gamma) = \bigsqcup_{i=1}^c\bigl\{p\hspace{.5pt}u_i^n\hspace{.5pt}q \;\big|\; p\in P_i,\, n\ge 0,\, q\in Q_i \bigr\} \sqcup R.\]
We define a new graph structure on $\Fc(\Gamma)$, with edges $p\hspace{.5pt}u_i^n\hspace{.5pt}q \sim p\hspace{.5pt}u_i^{n+1}\hspace{.5pt}q$, so that $\Fc(\Gamma)$ is isomorphic to $m = \sum_{i=1}^c \abs{P_i} \abs{Q_i}$ copies of $\N$, together with $\abs R$ isolated points.

We now study how $g\in \Gc(\Gamma)$ acts on $\Fc(\Gamma)\simeq V\Gamma$. Recall from the proof of Proposition \ref{prop:embedding-tfg} that $w_x$ and $w_x\cdot g\coloneqq w_{xg}$ differ from a suffix of length at most $\abs g +1$, and the new suffix only depends on $g$ and the old suffix. In particular, for all $n\gg 1$, we have
\[ (p\,u_i^n\,q)\cdot g = p\,u_i^{n+\delta(g,q)}\,(q)\pi_{g,i}\]
with $\delta(g,q)\in\Z$ and $\pi_{g,i}\in\Sym(Q_i)$. Note that this is a quasi-automorphism of the graph defined on $\Fc(\Gamma)$, i.e.\ there exists a finite set $B_g\subset\Fc(\Gamma)^2$ such that
\[ \forall (v,w)\in \Fc(\Gamma)^2-B_g,\qquad \bigl(v\sim w\iff vg\sim wg\bigr). \]
This concludes since
\[ \Gc(\Gamma) \le \QAut\left(\bigsqcup_{i=1}^m\N\sqcup R\right) \simeq \QAut\left(\bigsqcup_{i=1}^m \N\right) \simeq H_m \rtimes \Sym(m).\]

\begin{center}
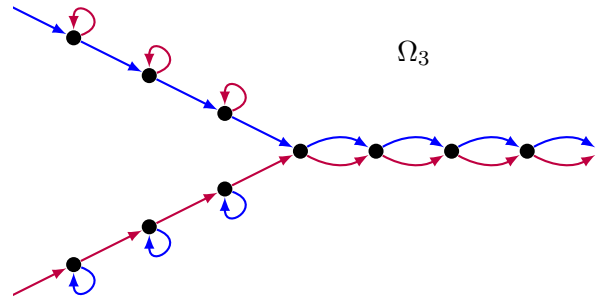

    \begin{tikzpicture}[scale=1]
        \clip (-2.8,-2) rectangle (5,2);
		\foreach \x in {1,...,4}{
			\node[circle, fill=black, inner sep=2pt] (P\x) at (\x,0) {};
            \node[circle, fill=black, inner sep=2pt] (Q\x) at (1-\x,{\x/2}) {};
            \node[circle, fill=black, inner sep=2pt] (R\x) at (1-\x,{-\x/2}) {};
            
			\draw[-latex, thick, blue, bend left] (P\x) to ({\x+.9},.05);
            \draw[-latex, thick, purple, bend right] (P\x) to ({\x+.9},-.05);
            
            \draw[-latex, thick, purple, looseness=10] (Q\x) to[out=20, in=90] (Q\x);
            \draw[-latex, thick, blue, looseness=10] (R\x) to[out=-20, in=-90] (R\x);
        }

		\draw[-latex, thick, blue] (Q1) to (P1);
        \draw[-latex, thick, blue] (Q2) to (Q1);
        \draw[-latex, thick, blue] (Q3) to (Q2);
        \draw[-latex, thick, blue] (Q4) to (Q3);
        
		\draw[-latex, thick, purple] (R1) to (P1);
        \draw[-latex, thick, purple] (R2) to (R1);
        \draw[-latex, thick, purple] (R3) to (R2);
        \draw[-latex, thick, purple] (R4) to (R3);

        \node at (2.5,1.3) {$\Omega_3$};
	\end{tikzpicture}
    \captionsetup{font=small}
    \captionof{figure}{A context-free graph $\Omega_3$ such that $H_3\simeq\Gc(\Omega_3)$.}
    \label{fig:Houghton_3}
\end{center}
(b) $\Rightarrow$ (a): This follows from the fact that $H_m$ is defined by a context-free graph of linear growth (see Figure \ref*{fig:Houghton_3}). Moreover, we can obtain context-free graphs defining first finite-index overgroups and then f.g.\ subgroups following the proofs of Proposition 4.16 and 4.20 in \cite{CFTR}, and these constructions do not increase the degree of growth.
\end{proof}

\begin{rem}
    It is likely one could prove Theorem \ref{thm:Houghton_universal} by adapting the arguments for Proposition \ref{prop:embedding-tfg}. Indeed, it is known that Houghton groups are topological full groups of edge shifts which are not irreducible \cite[Subsection 3.6.5]{Tarrot}, therefore an approach would be to use the same coding of Subsection \ref{subsection:CFTR-embedding}, but a different (non-irreducible) edge shift.
\end{rem}
\begin{rem}
    Passing to a finite-index subgroup in part (b) is important. For instance, the infinite dihedral group $D_\infty$ satisfies the first condition, but does not embed in any $H_m$. Indeed, we have a short exact sequence
    \[ 1 \longto \FSym(\Omega_m) \longto H_m \longto \Z^{m-1} \longto 1, \]
    Since \(\FSym(\Omega_m)\) is locally finite, every infinite finitely generated subgroup of \(H_m\) has infinite image in \(\Z^{m-1}\), and therefore admits a quotient isomorphic to \(\Z\). On the other hand, \(D_\infty^{\mathrm{ab}}\simeq C_2\times C_2\), so \(D_\infty\) has no quotient isomorphic to \(\Z\).
\end{rem}


\subsection{Elementary amenability}

We recall the definition of elementary amenable groups and their $\EA$-class.
\begin{itemize}[leftmargin=6mm]
    \item $\mathbf{EG}_0$ is the class of finite and abelian groups.
    \item $\mathbf{EG}_{\alpha+1}$ is the class of groups obtained as an extension of two groups in $\mathbf{EG}_\alpha$, or as a direct union of groups in $\mathbf{EG}_\alpha$.
    \item If $\alpha$ is a limit ordinal, then $\mathbf{EG}_\alpha=\bigcup_{\beta<\alpha}\mathbf{EG}_\beta$.
\end{itemize}
Finally, the class of \emph{elementary amenable} groups is $\bigcup_\alpha\mathbf{EG}_\alpha$. Given an elementary amenable group $G$, its EA-class is $\EA(G)=\min\{\alpha\mid G\in \mathbf{EG}_\alpha\}$.

\medskip

Recently, elementary amenable subgroups of $F$ with large EA-class were exhibited \cite{EA_in_Thompson_F}. In contrast, we show that transition groups of polynomially growing context-free graphs are elementary amenable, with finite EA-class.
\begin{thm} \label{thm:elementary_amenable}
    Let $\Gamma$ be a context-free graph of polynomial growth:
    \[\forall r\ge1,\quad \bar\beta_\Gamma(r) \coloneqq \sup_{y\in \Gamma}\#\bigl\{x\in\Gamma\;\big|\; d(x,y)\le r\bigr\}\le Cr^d. \vspace{-2mm} \]
    Then $\Gc(\Gamma)$ is elementary amenable and $\EA(\Gc(\Gamma))\le d+1$.
\end{thm}
\begin{proof}
    We define another sequence of classes $\mathbf C_\alpha$ (for $2\le \alpha<\omega$) which is better suited for our proof. We define recursively
    \begin{itemize}[leftmargin=6mm]
        \item $\mathbf C_2$ consists of (locally finite)-by-(virtually abelian) groups.
        \item $\mathbf C_{\alpha+1}$ is the class of $\mathbf C_2$-by-$\mathbf C_\alpha$ groups.
    \end{itemize}
    It is clear that $\mathbf C_\alpha\subset \EG_\alpha$. Moreover, an easy induction on $\alpha$ shows that
    \begin{itemize}[leftmargin=6mm]
        \item If $G_1,\ldots,G_k\in\mathbf C_\alpha$, then $G_1\times \ldots\times G_k\in\mathbf C_\alpha$.
        \item If $G\in\mathbf C_\alpha$ and $H\le G$, then $H\in\mathbf C_\alpha$.
    \end{itemize}

    We prove by induction on $d$ that $\Gc(\Gamma)\in\mathbf C_{d+1}$.
    
    \textbf{Base case:} If $d=1$, the group is (locally finite)-by-(virtually abelian) using Theorem \ref{thm:Houghton_universal}, hence belongs to $\mathbf C_2\subset\EG_2$.
    
    \textbf{Induction:} Using Propositions \ref{prop:cutting}, \ref{prop:periodic_limits}, \ref{prop:find_F_for_CF}, we get a short exact sequence
    \[ 1 \longto \Bsc_F \longto \Gc(\Gamma) \longto H_F \longto 1\]
    where $\Bsc_F$ is locally finite, and $H_F \le \bigoplus_{i=1}^k \Gc(\vec\Gamma_i)$ where each graph $\vec\Gamma_i$ is $\Z$-invariant, context-free, and has growth bounded by $Cr^d$ (Lemma \ref{lem:A-q-trees_are_closed}(b)).

    Using Propositions \ref{prop:cover_Linvariant}, \ref{prop:cover_restricted}, \ref{prop:cover_is_CF}, we have further short exact sequences
    \[ 1 \longto A  \longto \Gc(\vec\Gamma_i) \longto \Gc(\Z\backslash\vec\Gamma_i) \longto 1\]
    where $A\le\bigoplus_{\Z\backslash\vec\Gamma_i}\Z$ is abelian, and $\Z\backslash\vec\Gamma_i$ is context-free and has growth bounded by $r^{d-1}$. Combining the two observations gives a map
    \[ f\colon \Gc(\Gamma)\longto \bigoplus_{i=1}^k\Gc(\Z\backslash\vec\Gamma_i) \]
    where $\ker(f)$ is (locally finite)-by-abelian. By induction, $\Gc(\Z\backslash\vec\Gamma_i)$ hence $\bigoplus_{i=1}^k\Gc(\Z\backslash\vec\Gamma_i)$ and $\mathrm{im}(f)$ belongs to $\mathbf C_d$. We conclude that $\Gc(\Gamma)\in\mathbf C_{d+1}$.
\end{proof}

\medskip

\begin{rem}
    The reciprocal of Theorem \ref*{thm:elementary_amenable} does not hold: there exists an elementary amenable $G\le V$ with $\EA(G)=2$ which cannot be represented as $G=\Gc(\Gamma^{(1)}\sqcup\ldots\sqcup\Gamma^{(k)})$ where all the graphs $\Gamma^{(i)}$ have polynomial growth. An example is the BEN group $G=B(C_2)$ considered in \cite[\S 7]{Period}. This group is (locally finite)-by-$\Z$, hence satisfies $\EA(B(C_2))=2$. Moreover, given any \textbf{CF-TR} representation $B(C_2)\simeq\Gc(\Gamma^{(1)} \sqcup\ldots\sqcup \Gamma^{(k)})$, we get
    \[ \exp(n) \asymp p^{D_{\{2\}}}_{B(C_2)}(n) \preceq n\cdot \bar\beta_{\Gamma^{(1)} \sqcup\ldots\sqcup \Gamma^{(k)}}(n) \]
    using Theorem 7.1(a) and Proposition 7.5 of \cite{Period}.

    However, the following question remains open:
    \begin{ques}
        Consider a finitely generated \emph{solvable} subgroup $G$ of $V$. Can we represent $G=\Gc(\Gamma^{(1)}\sqcup\ldots\sqcup\Gamma^{(k)})$ where each of the graphs $\Gamma^{(i)}$ has polynomially growth?\vspace*{-2mm}
    \end{ques}
    A positive answer would be an important progress towards the classification of solvable subgroups of Thompson's $V$.
\end{rem}

\medskip

Theorem \ref{thm:elementary_amenable} can be used to partially address \cite[Question 6.3]{ET0L}.
\begin{cor} \label{cor:orbital_Schreier_are_not_CF}
    Let $G=\la\Ac\ra\le \Aut(T)$ be an automaton group of polynomial activity. Suppose that $G$ acts level-transitively on $T$, and that $G$ is not elementary amenable. Then no Schreier graph $\Sch(\xi,G;\Ac)$ with $\xi\in\partial T$ is context-free.
\end{cor}
\begin{proof}
    Each hypothesis translates into an information on the Schreier graphs:
    \begin{itemize}[leftmargin=6mm]
        \item Since $G$ is defined by an automaton of polynomial activity, all graphs $\Sch(\xi,G;\Ac)$ with $\xi\in\partial T$ have sub-exponential growth, see \cite[Theorem 1]{bondarenko2012growth}.
        \item The action on $\partial T$ is faithful, by homeomorphism and most importantly minimal (i.e.\ all orbits are dense). This last point is equivalent to level-transitivity. This implies that the action on each orbit $\xi\cdot G$ is faithful.
    \end{itemize}
    Therefore, if a single one of those Schreier graphs was context-free, then its growth must be polynomial from Lemma \ref*{lem:growth_of_CF_graphs} and we would conclude that $G$ is elementary amenable by Theorem \ref*{thm:elementary_amenable}.
\end{proof}
\begin{rem}
    Most bounded automaton groups in the literature are either virtually abelian or not elementary amenable \cite{Jus}. Two notable exceptions are given in \cite[\S9.1, \S 9.3]{Jus}, neither of which acts level-transitively. Let us focus on the second example, isomorphic to $C_2\wr \Z$: 
    \begin{center}
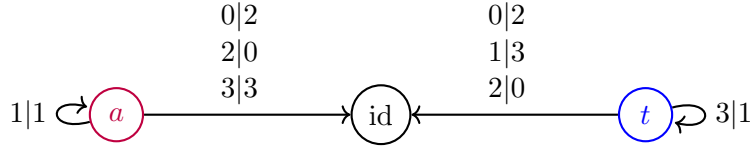

       \begin{tikzpicture}[scale=.7, thick]	
		\node[state, minimum size=20pt, purple] (a) at (0,0) {$a$};
		\node[state, minimum size=20pt] (e) at (5,0) {$\id$};
		\node[state, minimum size=20pt, blue] (t) at (10,0) {$t$};
		
		\path (a)  edge[->] node [above, text width=20pt] {$0|2$ $2|0$ $3|3$} (e)
		(t) edge[->] node [above, text width=20pt] {$0|2$ $1|3$ $2|0$} (e)
		(a) edge [loop left] node [left] {$1|1$}  (a)
       (t) edge [loop right] node [right] {$3|1$}  (t);
	\end{tikzpicture}
   \captionsetup{font=small}
   \captionof{figure}{An automaton generating $C_2\wr \Z$ with $C_2=\la a\ra$ and $\Z=\la t\ra$.}
    \end{center}
    The orbital Schreier graphs are described as follows:
    \begin{itemize}[leftmargin=6mm]
        \item For $\xi\in\{1,3\}^\infty$, the graph $\Sch(\xi,G;\Ac)$ is a copy of $\Z$ formed by $t$-edges, with $a$-self-loops at every vertex.
        \item Otherwise, say the first digit from $\{0,2\}$ in $\xi$ is in position $m$. Then the graph $\Sch(\xi,G;\Ac)$ is a cycle of length $2^m$ with $t$-edges, two antipodal vertices linked by a $a$-bigon, and $a$-self-loops at every other vertex.
    \end{itemize}
    In particular, all the orbital Schreier graphs are context-free.
\end{rem}

\medskip

\begin{ques}
    Does there exists an explicit family of groups which serves as universal containers for transition groups of finite union of context-free graphs of polynomial growth, in the same sense as Theorem \ref*{thm:Houghton_universal}?
\end{ques}
\begin{ques}
    Can we characterise when transition groups of finite union of context-free graphs of polynomial growth are finitely presented? Of type $F_n$? Compare with \cite{cox2025finiteness}.
\end{ques}

\subsection{Intermediate growth}

In this section, we prove an alternative which should be compared with results due to Wolf and Chou for solvable and elementary amenable groups:
\begin{thm} \label{thm:strong_Wolf_alternative}
    Let $G=\Gc(\Gamma^{(1)} \sqcup \ldots \sqcup \Gamma^{(k)})$ be a \textbf{CF-TR} group / a finitely generated subgroup of $V$. Then either
\begin{itemize}[leftmargin=6mm]
    \item $G$ is virtually abelian, or
    \item $G$ contains a free non-abelian semigroup.
\end{itemize}
    The first case happens only if the quasi-trees $\Gamma^{(i)}$ are finite, $1$ or $2$-ended.
\end{thm}
This recovers a known result about finitely generated torsion subgroups of $V$, but also extends it. For instance, an immediate corollary is the following: 
\begin{cor} \label{cor:intermediate}
    Groups of intermediate growth do not embed in $V$.
\end{cor}
This also cover other groups, such as the free group in the variety $x^py^p=y^px^p$ for odd $p\ge 665$ which is not virtually abelian (surjects on the free Burnside group $B(2,p)$) and does not contain any free semigroup (positive law).

\begin{proof}[Proof of Theorem \ref*{thm:strong_Wolf_alternative}]
    We first prove the alternative when $G=\Gc(\Gamma)$, starting from the alternative presented in Lemma \ref{lem:growth_of_CF_graphs}.
    \begin{itemize}[leftmargin=6mm]
        \item If $\Gamma$ has polynomial growth, then $G$ is elementary amenable by Theorem \ref*{thm:elementary_amenable}. By a result of Chou \cite[Theorem 3.2']{Chou}, this implies that $G$ is either virtually nilpotent or contains a free non-abelian semigroup. Moreover, the only f.g.\ virtually nilpotent groups embedding in $V$ are virtually abelian \cite[Corollary 1.10]{Burillo_Cleary_Rover_2017}, since all other f.g.\ virtually nilpotent groups contain a copy of $H_3(\Z)$ and therefore a distorted cyclic subgroup.
        \item If $\Gamma$ has exponential growth, then we construct a free semigroup in $G$, adapting the proof of \cite[Proposition 5.7]{CFTR}. Since $\Gamma$ has exponential growth, there exist end-cones $\Gamma(y,x_0)$, $\Gamma(z_1,x_0)$ and $\Gamma(z_2,x_0)$ such that
        \begin{enumerate}[leftmargin=8mm, label=(\arabic*)]
            \item $\Gamma(y,x_0)\supsetneq\Gamma(z_1,x_0)$ and $\Gamma(y,x_0)\supsetneq\Gamma(z_2,x_0)$,
            \item $\Gamma(z_1,x_0)\cap\Gamma(z_2,x_0)=\emptyset$, and
            \item there exist isomorphisms $\psi_i\colon \Gamma(y,x_0)\to \Gamma(z_i,x_0)$ with $\psi_i(y)=z_i$.
        \end{enumerate}
        Consider two words $u_i\in\Ac^*$ labelling paths $y\overset{u_i}\longto z_i$ included in $\Gamma(y,x_0)$. Consider a word $w=u_{i_1}u_{i_2}\ldots u_{i_m}$ ($m\ge 1$). It labels a path from $y$ to
        \[ \psi_{i_1}\psi_{i_2}\ldots \psi_{i_m}(y) \in \Gamma(z_{i_1},x_0)\]
        Condition (1) implies $y\notin\Gamma(z_{i_1},x_0)$ hence $w\ne 1$ in $\Gc(\Gamma)$. More generally, if two words represent the same element $u_{i_1}u_{i_2}\ldots u_{i_m}=u_{j_1}u_{j_2}\ldots u_{j_n}$, then $i_1=j_1$ using condition (2). By induction, we conclude that the semigroup generated by $\{u_1,u_2\}$ is free.
    \end{itemize}

    \medskip
    
    More generally, if $G$ is a finitely generated subgroup of Thompson's $V$, then $G \simeq \Gc\bigl(\Gamma^{(1)} \sqcup \ldots \sqcup \Gamma^{(k)}\bigr)$ for some context-free graphs $\Gamma^{(1)},\ldots,\Gamma^{(k)}$. Recall that $G$ is a subdirect product of the $\Gc(\Gamma^{(i)})$. We split the proof into two cases:
    \begin{itemize}[leftmargin=6mm]
        \item If all the groups $\Gc(\Gamma^{(i)})$ are virtually abelian, then $G$ itself is virtually abelian since it embeds in $\bigoplus_{i=1}^k \Gc(\Gamma^{(i)})$.
        \item Otherwise, there exists $i$ such that $\Gc(\Gamma^{(i)})$ contains a free semigroup, then $G$ itself contains a free semigroup since $G\onto \Gc(\Gamma^{(i)})$.
    \end{itemize}

    \medskip
    
    Finally we prove that, if $G=\Gc(\Gamma)$ is virtually abelian, then the number of ends of $\Gamma$ is $0$, $1$ or $2$. We recall a folklore lemma, related to the invariance of the number of ends $E(\Gamma)\in\{0,1,2,\ldots,\infty\}$ under quasi-isometries.
\begin{lemma}
    Let $\Gamma$ and $\Lambda$ be locally finite graphs. If $\iota\colon \Lambda\into \Gamma$ is a coarsely surjective graph embedding, then $E(\Lambda)\ge E(\Gamma)$.
\end{lemma}
\begin{proof}
    Suppose that $E(\Gamma)\ge m$. There exists a finite subset $F\subset\Gamma$ such that $\Gamma-F$ admits at least $m$ infinite connected components $C_1,\ldots,C_m$. Since $\iota$ is coarsely surjective, the preimages $\iota^{-1}(C_i)$ must be infinite, and they lie in distinct components of $\Lambda-\iota^{-1}(F)$. We conclude that $E(\Lambda)\ge m$.
\end{proof}

    After possibly changing the generating set, we may suppose that $\Ac$ contains a subset $\Bc$ generating a finite-index abelian subgroup $H$ of $G$. We observe that $\Lambda=\Sch(x_0,H;\Bc)$ is the Cayley graph of an abelian quotient of $H$, it must therefore be $0$-, $1$- or $2$-ended. Moreover, there exists a finite transversal $T$ such that $G=HT$, and therefore
    \[ V(\Gamma)=x_0\cdot G=x_0H\cdot T = V(\Lambda)\cdot T. \]
    This implies that $\Lambda\into \Gamma$ is coarsely surjective, hence $E(\Gamma)\le E(\Lambda)\le 2$.
\end{proof}
\begin{center}
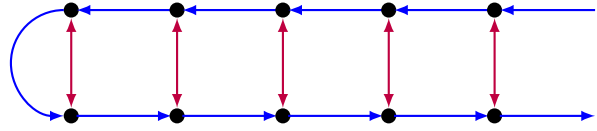

        \begin{tikzpicture}[scale=1.4]
		\foreach \x in {0,...,4}{
			\node[circle, fill=black, inner sep=2pt] (P\x) at (\x,0) {};
			\draw[-latex, thick, blue] ({\x+.05},0) -- ({\x+.95},0);
		      \node[circle, fill=black, inner sep=2pt] (Q\x) at (\x,1) {};
			\draw[latex-, thick, blue] ({\x+.05},1) -- ({\x+.95},1);}
            
		\foreach \x in {0,1,2,3,4}{
            \draw[latex-latex, thick, purple] (P\x) to (Q\x);}

        \draw[-latex, thick, blue, bend right=90] (-.07,1) arc (90:270:.5);
		
	\end{tikzpicture}
	\captionsetup{font=small, margin=18mm}
	\captionof{figure}{A $1$-ended context-free graph $\Gamma$ for $G=\Z\rtimes C_2$, together with a $2$-ended subgraph $\Lambda$ for $H=\Z$.}
    \end{center}
    
\begin{rem}
    There exist two-ended context-free graphs $\Gamma$ such that $\Gc(\Gamma)$ contains a free semigroup. For instance, the lamplighter group $C_2\wr\Z$ and the lampshuffler group $\FSym(\Z)\rtimes\Z$ are defined by the following graphs:
    \begin{center}
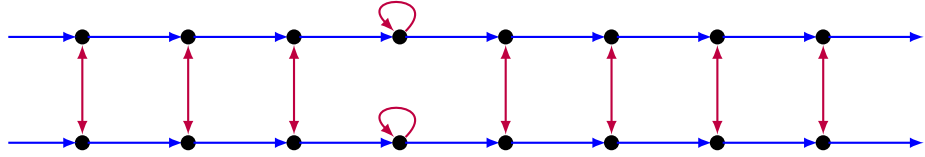

        \begin{tikzpicture}[scale=1.4]
		\foreach \x in {-3,...,4}{
			\node[circle, fill=black, inner sep=2pt] (P\x) at (\x,0) {};
			\draw[-latex, thick, blue] ({\x+.05},0) -- ({\x+.95},0);
		      \node[circle, fill=black, inner sep=2pt] (Q\x) at (\x,1) {};
			\draw[-latex, thick, blue] ({\x+.05},1) -- ({\x+.95},1);}
            
		\foreach \x in {-3,-2,-1,1,2,3,4}{
            \draw[latex-latex, thick, purple] (P\x) to (Q\x);}

        \draw[-latex, thick, purple, loop, below] (P0) to (P0);
        \draw[-latex, thick, purple, loop, above] (Q0) to (Q0);
        
		\draw[-latex, thick, blue] (-3.7,0) -- (-3.05,0);
		\draw[-latex, thick, blue] (-3.7,1) -- (-3.05,1);
	\end{tikzpicture}
	\captionsetup{font=small}
	\captionof{figure}{A context-free graph for the lamplighter group $C_2\wr\Z$.}
    \label{fig:lamplighter}
    \end{center}
    \begin{center}
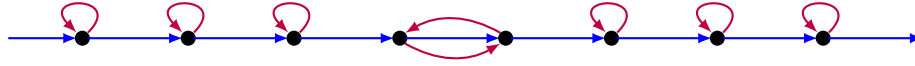

    \begin{tikzpicture}[scale=1.4]
		\foreach \x in {-3,...,4}{
			\node[circle, fill=black, inner sep=2pt] (P\x) at (\x,0) {};
			\draw[-latex, thick, blue] ({\x+.05},0) -- ({\x+.95},0);}
		
		\foreach \x in {-3,-2,-1,2,3,4}{
			\draw[-latex, thick, purple, loop, above] (P\x) to (P\x);}
		
		\draw[-latex, thick, bend right, purple] (.05,-.05) to (.95,-.05);
		\draw[-latex, thick, bend right, purple] (.95,.05) to (.05,.05);
		\draw[-latex, thick, blue] (-3.7,0) -- (-3.05,0);
	\end{tikzpicture}
	\captionsetup{font=small}
	\captionof{figure}{Context-free graph for the Houghton group $H_2=\FSym(\Z)\rtimes\Z$.}
    \end{center}
\end{rem}

%% file: Sections/sec7_branch.tex
In this section, we combine the observation that Schreier graphs of subgroups of $V$ on the Cantor set are quasi-trees with the following lemma to deduce that many groups of dynamical origin \emph{do not} embed in Thompson's $V$.

\begin{lemma} \label{lem:free_in_cover}
    Suppose that $\Cay(G;\Ac)\onto \Gamma\overset\pi\onto \Sigma$ where $\Gamma$ is a quasi-tree and $\Sigma$ is not a quasi-tree. Then $G$ contains a free non-abelian subgroup.
\end{lemma}

\begin{proof}
    Suppose that $\Gamma$ admits a non-expansive map $f\colon \Gamma\to T$ with
    \[ d_\Gamma(x,y)-C \le d_T\bigl(f(x),f(y)\bigr) \le d_\Gamma(x,y). \]
    Since $\Sigma$ is not a quasi-tree, it contains a $K$-fat $R_2$ minor for $K>C$. Consider the words $a=x_1y_1z_1y_2x_2$ and $b=x_3y_3z_3y_4x_4$ in $\Ac^*$ labelling the closed paths with common base point on Figure \ref*{fig:fat_minor} and $w\in F_2\setminus\{1\}$ (seen as a reduced word).
    \begin{center}
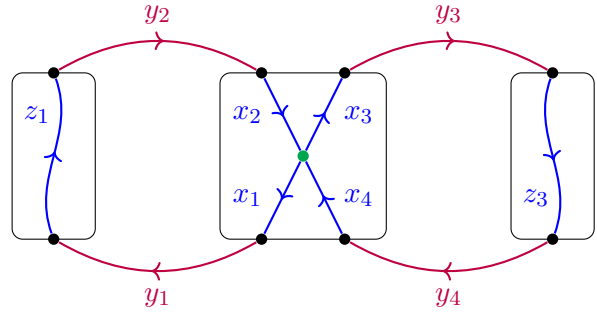

        \import{Pictures/}{tikz_fat_minor.tex}
        \captionsetup{font=small}
        \captionof{figure}{A $K$-fat $R_2$ minor, together with associated paths.}
        \label{fig:fat_minor}
    \end{center}
    Take the obvious path labelled by $w(a,b)$ in the minor, a lift in $\Gamma$, and consider its (continuous) image through $f$ in $T$. This latest path can be decomposed as a concatenation segments $s_1s_2\ldots s_\ell$ which alternates colors ($s_1=x_i^\pm$ is blue, $s_2=y_i^\pm$ is red, $s_3=z_j^\pm$ is blue, etc.). The distance between $s_i$ and $s_{i+2}$ in $\Sigma$ is at least $K$. Since $\pi\colon \Gamma\onto \Sigma$ is a graph covering, in particular is non-expanding, the same is true for their lift in $\Gamma$. Finally, by using the property of $d_T$, the distance between their images in $T$ is at least $K-C>0$. Let $q_i$ be the last intersection between $s_i$ and $s_{i+1}$ in $T$ (i.e., if you consider a parametrisation for each segment, this is $s_{i+1}(t)$ in $\mathrm{Im}(s_i)$ with $t$ maximal). 
    By induction, we have
    \[ d(q_i,q_0)\ge (i-1)(K-C) \]
    (see Figure \ref*{fig:quasi-tree}), so $w(a,b)\ne 1$ in $G$, i.e.\ $\la a,b\ra$ is a free subgroup in $G$.

\end{proof}
\begin{center}
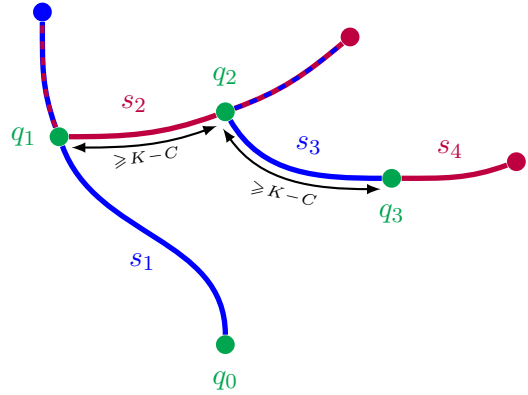

    \import{Pictures/}{tikz_quasi_tree.tex}
    \captionsetup{font=small}
    \captionof{figure}{Image of the lifted path in $T$}
    \label{fig:quasi-tree}
\end{center}

\subsection{Some weakly branch groups} \label{sec:branch}

Consider a group $G$ acting on a rooted tree $T$. As shown in Corollary \ref{cor:orbital_Schreier_are_not_CF}, the Schreier graphs $\Sch(\xi,G;\Ac)$ (with $\xi\in\partial T$) are not context-free for a wide family of groups. Since these are the only \say{natural} Schreier graphs at our disposition, we expect that those groups will not embed in $V$. We prove this for some weakly branch groups, using the previous lemma. For a definition and some background on (weakly) branch groups, see \cite{Bartholdi2005BranchG}.
    
\begin{thm} \label{thm:big_loops}
    Let $T$ be a rooted tree and $G=\langle \Ac\rangle\le \Aut(T)$. Suppose that
    \begin{enumerate}[leftmargin=8mm, label={\normalfont(\alph*)}]
        \item $G$ is weakly branch and $\RStab_G(\ell)'$ is finitely generated for all $\ell$,
        \item $G$ does not contain a non-abelian free subgroup, and
        \item for all $\xi\in\partial T$, the Schreier graph $\Sch(\xi, G;\Ac)$ \emph{is not} a quasi-tree.
    \end{enumerate}
    Then, for every action $X\racts G$ such that $\Gamma=\Sch(X, G;\Ac)$ \emph{is} a quasi-tree, there exists $\ell\ge 0$ such that $\ker(X\racts G)\ge \RStab(\ell)'$.
\end{thm}
Condition (a) is known for branch groups \cite[Corollary A.5]{Francoeur}. Condition (b) is known for many groups acting on rooted trees, including automata groups of polynomial activity \cite{no_free_in_poly_activity}, and contracting groups \cite{no_free_in_contracting}. The proof largely follows the structure of \cite[Theorem 6.3]{Branch_confined}:
\begin{proof}
    Note that graphs in $\overline{\{(\Gamma,x): x\in V(\Gamma)\}}$ are quasi-trees (Lemma \ref{lem:A-q-trees_are_closed}). Since $G$ is weakly branch, it contains a direct sum of infinite groups
    $$\RStab_G(1) = \bigoplus_{v\in \Lc_1} \RStab_G(v)$$
    hence $G$ cannot be virtually free and its Cayley graph is not a quasi-tree: $\Cay(G;\Ac)\not\preceq \Gamma$. Equivalently $\{1\}\notin \overline{\{H^g:g\in G\}}$ where $H=\Stab_G(x)$ for any $x\in V(\Gamma)$. We consider an URS $\tilde H\in \overline{\{H^g:g\in G\}}$. Using \cite[Corollary 5.20]{Branch_confined}, we are left with two cases:
    \begin{itemize}[leftmargin=6mm]
        \item Either there exists $\xi\in\partial T$ s.t.\ $\tilde H\le \Stab_G(\xi)$. Equivalently, there exists $\tilde \Gamma\preceq \Gamma$ s.t.\
        $$ \Cay(G;\Ac)\onto \tilde\Gamma\onto \Sigma,$$
        where $\tilde \Gamma=\Sch(\tilde H\backslash G, G;\Ac)$ and $\Sigma=\Sch(\xi, G;\Ac)$. By hypothesis, $\tilde \Gamma$ is a quasi-tree (Lemma \ref{lem:A-q-trees_are_closed}(a)) while $\Sigma$ is not a quasi-tree, therefore $G$ contains a free non-abelian subgroup by Lemma \ref*{lem:free_in_cover}, a contradiction.

        \item Or there exists $\ell\ge 0$ such that $\tilde H\ge \RStab_G(\ell)'$. Since $\RStab_G(\ell)'$ is finitely generated (say generated by $h_1,\ldots,h_m$), the set
        \[ \Bigl\{ K\in \Sub(G) \;\Big|\; K\ge \RStab_G(\ell)'\Bigr\} = \Bigl\{K\in\Sub(G) \;\Big|\; K\ni h_1,\ldots,h_m\Bigr\} \]
        is a clopen which intersects $\overline{\{H^g:g\in G\}}$ (it contains $\tilde H$), hence it intersects the dense subset $\{H^g:g\in G\}$. We deduce that $H^g \ge \RStab_G(\ell)'$ for some $g$ and
        \[ \ker(X\racts G) = \bigcap_g H^g  \ge \RStab_G(\ell)',\]
        since $\RStab_G(\ell)'$ is normal. \qedhere
    \end{itemize}
\end{proof}
\begin{rem}
    Condition (a) could be weakened to \say{$\RStab(\ell)'=\bigcup_{i\in I}N_i$ where $N_i$ are finitely generated normal subgroups}. It is unclear if this allows to treat more examples.
\end{rem}
\begin{cor}
    The Hanoï Towers groups $H^{(m)}$ with $m\geq 3$, and all the groups $G_n$ with $n\geq 3$ described in \cite{Skipper_thesis}, do not embed in $V$.
\end{cor}

\begin{proof}
We want to apply Theorem \ref{thm:big_loops}, starting with the groups $G_n$. Recall that the group $G_n\le\Aut(T_n)$ is generated by the automorphisms
\[
a_i=(1,\ldots,1,a_i,1,\ldots,1)\sigma_i
\]
for $i=1,\ldots,n$, where $a_i$ occupies the $i$-th section and 
\[
\sigma_i=(1,2,\ldots,i-1,i+1,\ldots,n-1,n)\in\Sym(n).
\]
By \cite{Skipper_thesis}, the groups $G_n$ are branch, so Condition~(a) holds. Moreover, they are generated by bounded automata, hence they do not contain non-abelian free subgroups, establishing Condition~(b). It remains to verify Condition~(c).

Recall that the orbit of $\xi\cdot G_n$ is the \emph{cofinality} class of $\xi$, i.e.\ the sequences coinciding with $\xi$ from some symbol onwards. We prove that $\Sch(\xi,G_n;\Ac)$ is not a quasi-tree, using Manning's bottleneck criterion \cite[Theorem 4.6]{manning2005geometry}:
\begin{lemma}
    A graph $\Sigma$ is a quasi-tree if and only if there exists $R\ge 0$ such that, for every $x,y,m\in \Sigma$ satisfying $d(x,m)= d(m,y)= \frac12 d(x,y)$, we have that all paths from $x$ to $y$ pass within a ball of radius $R$ around $m$.
\end{lemma}
We fix $R\ge 0$ and $k\ge R+1$. We observe that, for $a\in\Ac$, the point $\xi\in [n]^\omega$ differ by at most one symbol with its image $\xi\cdot a$. It follows that
\begin{itemize}[leftmargin=6mm]
    \item the distance between two points $\xi,\xi'$ in the Schreier graph is at least the Hamming distance between $\xi$ and $\xi'$, and
    \item for $u,v\in [n]^k$, all geodesics from $u\eta$ and $v\eta$ are included in $[n]^k\eta$.
\end{itemize}

Consider the points $p=3^k2\,\xi_k$, $m=1^k2\,\xi_k$, $m'=1^k3\,\xi_k$, $q=2^k3\,\xi_k$ where $\xi_k$ is the suffix of $\xi$ after the first $k+1$ symbols. Let $x,y$ be points on geodesics $[p,m]$ and $[m',q]$ such that $d(x,m)=d(m',y)+1=R+1$.
\begin{center}
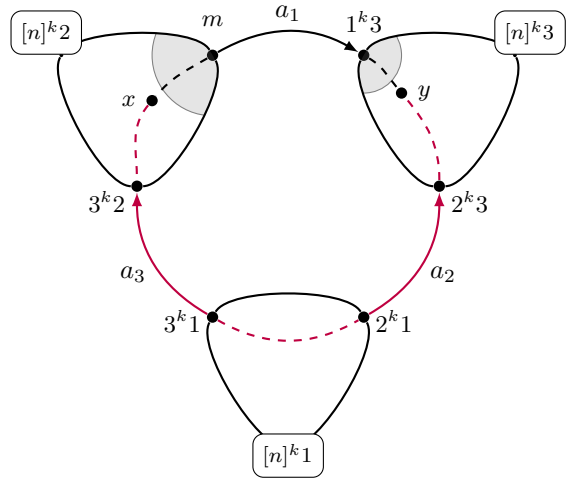

\import{Pictures/}{tikz_Hanoi.tex}
\captionsetup{font=small, margin=20mm}
\captionof{figure}{Part of $\Sch(1^\omega,G_n;\Ac)$ witnessing the failure of the bottleneck criterion. (We omit all $\xi_k$ suffixes.)}
\label{fig:Hanoi}
\end{center}
Observe that $d(m,[n]^k1\,\xi_k)\ge k+1$. Indeed, in order to reach a point with a $1$ in $(k+1)$-th position from $m=1^k2\,\xi_k$, one first needs to reach a point of the form $s^k2\,\xi_k$ with $s\ne 1$ which is at distance $k$ of $m$. It follows that the path $x\leftrightarrow y$ (in purple in Figure \ref*{fig:Hanoi}) avoids the ball $D_R(m)$ (in gray).

Since $R$ is arbitrary, we conclude that $\Sch(\xi,G_n;\Ac)$ is not a quasi-tree, hence Condition~(c) holds, and the result for $G_n$ (with $n\ge 3$) follows from Theorem~\ref{thm:big_loops}. Finally, for the classical Hanoï Towers groups on $m\ge 4$ pegs, we have $G_3=H^{(3)}\into H^{(m)}$, hence $H^{(m)}\not\into V$.
\end{proof}

\medskip

Finally, we apply Theorem \ref{thm:big_loops} to one notable \emph{weakly} branch group: \vspace*{1mm}
\begin{cor}
    The Basilica group $\Bc$ does not embed in $V$. \vspace*{-1mm}
\end{cor}
\begin{proof}
    Condition (a) follows from \cite[Theorem 1(b)]{Basilica} and Lemma \ref{prop:Dominik_Basilica} below. Condition (b) follows from $\Bc$ being contracting \cite[Theorem 1(j)]{Basilica}. Finally, Condition (c) follows from the study of orbital Schreier graphs in \cite{Basilica_Schreier}.
\end{proof}
\begin{center}
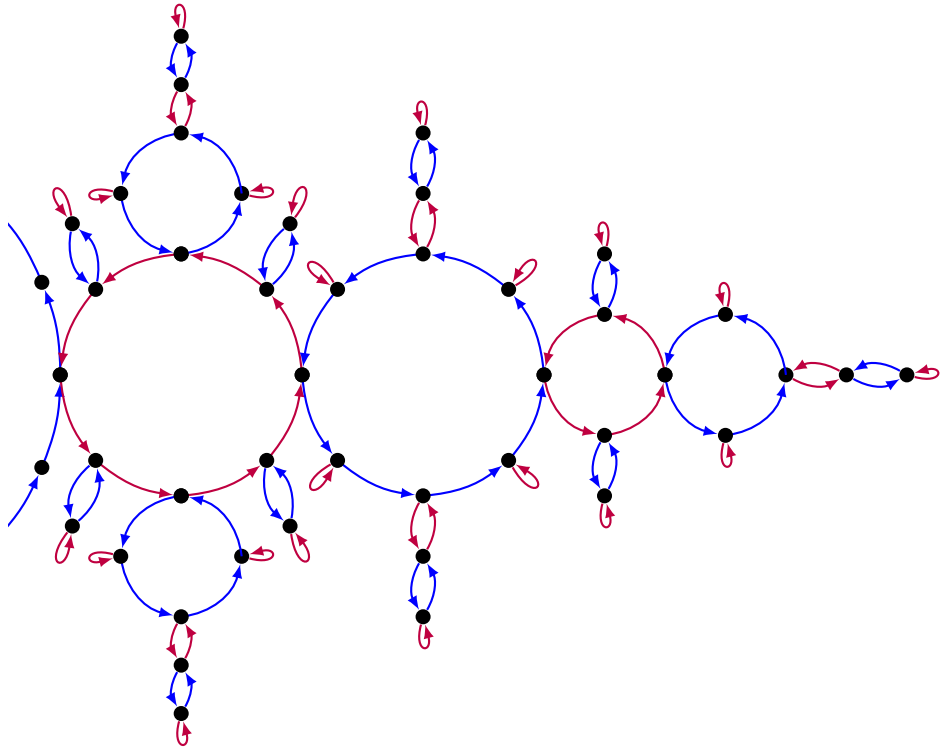

    \import{Pictures/}{tikz_basilica.tex}
    \captionsetup{margin=20mm, font=small}
    \captionof{figure}{A typical orbital Schreier graph of $\Bc\acts\partial T_2$, with arbitrarily long isometrically embedded cycles.} \vspace*{1mm}
\end{center}
\begin{lemma} \label{prop:Dominik_Basilica}
    Normal subgroups of the Basilica group are finitely generated.\vspace*{-1mm}
\end{lemma}
\begin{proof}
We check the hypotheses of \cite[Theorem A.4]{Francoeur}. The Basilica group $\Bc$ is regularly weakly branch over $\Bc'$. Proper quotients of $\Bc$ (eg.\ $\Bc/\RStab_\Bc(\ell)'$) are polycyclic and therefore Noetherian \cite[Proposition 6]{Basilica}. (More strongly $\Bc$ is just not-virtually nilpotent \cite[Proposition 4.13]{Francoeur_Basilica}.)  It follows that
\[ \RStab_\Bc(\ell)/(\Bc')^{2^\ell}\le \Bc/(\Bc')^{2^\ell} \]
is finitely generated. Since $\Bc'$ is finitely generated \cite[Proposition 4.7]{Francoeur_Basilica} (and therefore $(\Bc')^{2^\ell}$ too), we deduce that $\RStab_\Bc(\ell)$ is finitely generated.
\end{proof}


\medskip

    It is unclear how far these results extend among (weakly) branch groups. It is known that infinite finitely generated torsion groups fail to embed in Thompson's group $V$ (\cite[Theorem 3]{Rover}, see also \cite[Corollary 5.8]{CFTR}). Together with the fact that groups of intermediate growth fail to embed in $V$ and \cref{thm:big_loops}, this covers many examples from the literature.

\begin{ques}
    Does any (weakly) branch group embed in Thompson's $V$?
\end{ques}

\bigbreak

\begin{rem}
    Very recent results of Matte Bon, Nekrashevych and Zheng prove that all Schreier graphs of certain branch groups $G$ have finitely many ends. For instance, \cite[Theorem B]{End_Schreier_branch} covers all contracting self-replicating groups, which includes the generalised Hanoï Towers groups $G_n$. Therefore, if $G=\Gc(\Gamma^{(1)}\sqcup \ldots\sqcup \Gamma^{(k)})$ for context-free graphs $\Gamma^{(i)}$, then all the graphs would need to have finitely many ends, hence be finite or have linear growth. This quickly leads to a contradiction: branch groups are not elementary amenable \cite[Corollary 8]{juschenko2018nonEA}, contradicting with Theorem \ref{thm:elementary_amenable}.
\end{rem}
    
In particular, we would like to highlight the following question:\vspace*{1mm}
\begin{ques}
    Let $G$ be a finitely generated branch group. Is it true that all Schreier graphs of $G$ have finitely many ends? Countably many ends? \vspace*{-1mm}
\end{ques}
A positive answer to either question would prove that Thompson's $V$ does not contain any finitely generated branch subgroup.

%% file: Pictures/tikz_fat_minor.tex
\begin{tikzpicture}[scale=1.1]
        \draw[rounded corners] (0,0) rectangle (1,2);
        \draw[rounded corners] (2.5,0) rectangle (4.5,2);
        \draw[rounded corners] (6,0) rectangle (7,2);

        \node[circle, fill=black, inner sep=1.5pt] (i1) at (3,0) {};
        \node[circle, fill=black, inner sep=1.5pt] (t1) at (.5,0) {};
        \node[circle, fill=black, inner sep=1.5pt] (i2) at (3,2) {};
        \node[circle, fill=black, inner sep=1.5pt] (t2) at (.5,2) {};    
        \node[circle, fill=black, inner sep=1.5pt] (i3) at (4,2) {};
        \node[circle, fill=black, inner sep=1.5pt] (t3) at (6.5,2) {};
        \node[circle, fill=black, inner sep=1.5pt] (i4) at (4,0) {};
        \node[circle, fill=black, inner sep=1.5pt] (t4) at (6.5,0) {};
        \node[circle, fill=Green, inner sep=1.5pt] (c) at (3.5,1) {};

        \begin{scope}[thick,
            decoration={markings, mark=at position 0.53 with {\arrow{>}}}] 
            \draw[purple, bend left, postaction={decorate}] (i1) to (t1);
            \draw[blue, out=110, in=-70, postaction={decorate}] (t1) to (t2);
            \draw[purple, bend left, postaction={decorate}] (t2) to (i2);
            
            \draw[blue, postaction={decorate}] (c) to (i1);
            \draw[blue, postaction={decorate}] (i2) to (c);
            \draw[blue, postaction={decorate}] (c) to (i3);
            \draw[blue, postaction={decorate}] (i4) to (c);
            
            \draw[purple, bend left, postaction={decorate}] (i3) to (t3);
            \draw[blue, out=-110, in=70, postaction={decorate}] (t3) to (t4);
            \draw[purple, bend left, postaction={decorate}] (t4) to (i4);

            \node[purple] at (1.75,-.7) {$y_1$};
            \node[purple] at (1.75,2.7) {$y_2$};
            \node[purple] at (5.25,2.7) {$y_3$};
            \node[purple] at (5.25,-.7) {$y_4$};

            \node[blue] at (2.83,.5) {$x_1$};
            \node[blue] at (2.83,1.5) {$x_2$};
            \node[blue] at (4.17,1.5) {$x_3$};
            \node[blue] at (4.17,.5) {$x_4$};
            
            \node[blue] at (0.3,1.5) {$z_1$};
            \node[blue] at (6.3,.5) {$z_3$};
            
        \end{scope}
\end{tikzpicture}

%% file: Pictures/tikz_quasi_tree.tex
\begin{tikzpicture}[scale=1.1]
    \begin{scope}[every node/.style={circle, inner sep=2.5pt}]
        \node[fill=Green, label=below:$\color{Green}q_0$] (q0) at (0,0) {};
        \node[fill=Green, label=left:$\color{Green}q_1$] (q1) at (-2,2.5) {};
        \node[fill=blue] (s1) at (-2.2,4) {};
        \node[fill=Green, label=above:$\color{Green}q_2$] (q2) at (0,2.8) {};
        \node[fill=purple] (s2) at (1.5,3.7) {};
        \node[fill=Green, label=below:$\color{Green}q_3$] (q3) at (2,2) {};
        \node[fill=purple] (s3) at (3.5,2.2) {};
    \end{scope}

        \draw[line width=2pt, blue, out=90, in=-70] (q0) to (q1);
        \draw[line width=2pt, blue, out=110, in=-90] (q1) to (s1);
        \draw[ultra thick, purple, dashed, out=110, in=-90] (q1) to (s1);
        \draw[line width=2pt, purple, out=0, in=-160] (q1) to (q2);
        \draw[line width=2pt, purple, out=20, in=-140] (q2) to (s2);
        \draw[ultra thick, blue, dashed, out=20, in=-140] (q2) to (s2);
        \draw[line width=2pt, blue, out=-60, in=-180] (q2) to (q3);
        \draw[line width=2pt, purple, out=0, in=-160] (q3) to (s3);

        \draw[thick, latex-latex, out=0, in=-160] (-1.85,2.37) to (-.1,2.63);
        \node[rotate=8] at (-.95,2.25) {\tiny $\ge\! K\!-\!C$};

        \draw[thick, latex-latex, out=-55, in=-180] (-.03,2.6) to (1.87,1.87);
        \node[rotate=-13] at (.7,1.8) {\tiny $\ge\! K\!-\!C$};

        \node[blue] at (-1,1) {$s_1$};
        \node[purple] at (-1.1,2.9) {$s_2$};
        \node[blue] at (1,2.4) {$s_3$};
        \node[purple] at (2.7,2.35) {$s_4$};
\end{tikzpicture}

%% file: Pictures/tikz_Hanoi.tex
\begin{tikzpicture}
    \node[circle, fill=black, inner sep=.5pt] (L) at (-3,0) {};
    \node[circle, fill=black, inner sep=1.5pt] (m) at (-1,0) {};
    \node at (-1,.45) {\footnotesize$m$}; 
    \node[circle, fill=black, inner sep=1.5pt] (mp) at (1,0) {};
    \node at (1,.45) {\footnotesize$1^k3$};
    \node[circle, fill=black, inner sep=.5pt] (R) at (3,0) {};
    
    \node[circle, fill=black, inner sep=1.5pt] (p) at (-2,{-sqrt(3)}) {};
    \node at (-2.4,{-sqrt(3)-.2}) {\footnotesize$3^k2$};
    \node[circle, fill=black, inner sep=1.5pt] (pp) at (-1,{-2*sqrt(3)}) {};
    \node at (-1.4,{-2*sqrt(3)-.05}) {\footnotesize$3^k1$};
    \node[circle, fill=black, inner sep=1.5pt] (q) at (2,{-sqrt(3)}) {};
    \node at (2.4,{-sqrt(3)-.2}) {\footnotesize$2^k3$};
    \node[circle, fill=black, inner sep=1.5pt] (qq) at (1,{-2*sqrt(3)}) {};
    \node at (1.4,{-2*sqrt(3)-.05}) {\footnotesize$2^k1$};
    \node[circle, fill=black, inner sep=.5pt] (D) at (0,{-3*sqrt(3)}) {};

    \begin{scope}
        \clip[looseness=.58] (L) to[out=60, in=115] (-1,0) to[out=-55, in=0] (p);
        \draw[gray, fill=gray, fill opacity=.2] (-1,0) circle (.8cm);
    \end{scope}
    \begin{scope}
        \clip[looseness=.58] (q) to[out=180, in=-125] (1,0) to[out=65, in=120] (R);
        \draw[gray, fill=gray, fill opacity=.2] (1,0) circle (.5cm);
    \end{scope}
    
    \draw[thick, looseness=.5] (L) to[out=60, in=120] (m) to[out=-60, in=0] (p) to[out=180, in=-120] (L);
    \draw[thick, looseness=.5] (mp) to[out=60, in=120] (R) to[out=-60, in=0] (q) to[out=180, in=-120] (mp);
    \draw[thick, looseness=.5] (pp) to[out=60, in=120] (qq) to[out=-60, in=0] (D) to[out=180, in=-120] (pp);

    \node[draw, rounded corners,fill=white] at (-3.2,.3) {\scriptsize$[n]^k2\hspace*{.5pt}$};
    \node[draw, rounded corners,fill=white] at (3.2,.3) {\scriptsize$[n]^k3\hspace*{.5pt}$};
    \node[draw, rounded corners,fill=white] at (0,{-3*sqrt(3)-.1}) {\scriptsize$[n]^k1\hspace*{.5pt}$};

    \draw[thick, -latex, out=30, in=150] (m) to (mp);
    \node at (0,.55) {\small$a_1$};
    \draw[thick, purple, latex-, out=-90, in=150] (p) to (pp);
    \node at (-2.05,-2.9) {\small$a_3$};
    \draw[thick, purple, latex-, out=-90, in=30] (q) to (qq);
    \node at (2.05,-2.9) {\small$a_2$};

    \node[circle, fill=black, inner sep=1.5pt, label=left:\footnotesize$x$] (x) at (-1.8,-.6) {};
    \node[circle, fill=black, inner sep=1.5pt, label=right:\footnotesize$y$] (y) at (1.5,-.5) {};
    
    \draw[thick, dashed, out=-150, in=50] (m) to (x);
    \draw[thick, dashed, purple, out=-130, in=90] (x) to (p);
    \draw[thick, dashed, purple, out=-30, in=-150] (pp) to (qq);
    \draw[thick, dashed, purple, out=90, in=-50] (q) to (y);
    \draw[thick, dashed, out=-30, in=130] (mp) to (y);
\end{tikzpicture}

%% file: Pictures/tikz_basilica.tex


	
	\begin{tikzpicture}[scale=.8]
		\clip (-2.85,-6.6) rectangle (12.6,6.6);
		\begin{scope}[shift={(-6,0)}]
			\foreach \x in {3,...,-2}{
				\node[circle, fill=black, inner sep=2pt] (A\x) at ({4*cos(\x*22.5)},{4*sin(\x*22.5)}) {};
				\draw[blue, thick, bend right=10, -latex] ({4*cos((\x-1)*22.5)},{4*sin((\x-1)*22.5)}) to (A\x);}
		\end{scope}
		
		\begin{scope}
			\foreach \x in {8,...,1}{
				\node[circle, fill=black, inner sep=2pt] (B\x) at ({2*cos(\x*45)},{2*sin(\x*45)}) {};
				\draw[purple, thick, bend right=16, -latex] ({2*cos((\x-1)*45)},{2*sin((\x-1)*45)}) to (B\x);}
		\end{scope}
		\node[circle, fill=black, inner sep=2pt] (B1p) at (1.8,2.5) {};
		\draw[blue, thick, bend right, -latex] (B1) to (B1p);
		\draw[blue, thick, bend right, -latex] (B1p) to (B1);
		\draw[purple, thick, looseness=25, out=52, in=82, -latex] (B1p) to (B1p);
		\node[circle, fill=black, inner sep=2pt] (B3p) at (-1.8,2.5) {};
		\draw[blue, thick, bend right, -latex] (B3) to (B3p);
		\draw[blue, thick, bend right, -latex] (B3p) to (B3);
		\draw[purple, thick, looseness=25, out=98, in=128, -latex] (B3p) to (B3p);
		\node[circle, fill=black, inner sep=2pt] (B5p) at (-1.8,-2.5) {};
		\draw[blue, thick, bend right, -latex] (B5) to (B5p);
		\draw[blue, thick, bend right, -latex] (B5p) to (B5);
		\draw[purple, thick, looseness=25, out=232, in=262, -latex] (B5p) to (B5p);
		\node[circle, fill=black, inner sep=2pt] (B7p) at (1.8,-2.5) {};
		\draw[blue, thick, bend right, -latex] (B7) to (B7p);
		\draw[blue, thick, bend right, -latex] (B7p) to (B7);
		\draw[purple, thick, looseness=25, out=278, in=308, -latex] (B7p) to (B7p);
		
		\begin{scope}[shift={(0,3)}]
			\foreach \x in {4,...,1}{
				\node[circle, fill=black, inner sep=2pt] (C\x) at ({cos(\x*90)},{sin(\x*90)}) {};
				\draw[blue, thick, bend right=35, -latex] ({cos((\x-1)*90)},{sin((\x-1)*90)}) to (C\x);}
		\end{scope}
		\draw[purple, thick, loop left, latex-] (C2) to (C2);
		\draw[purple, thick, loop right, latex-] (C4) to (C4);
		
		\node[circle, fill=black, inner sep=2pt] (Cp) at (0,4.8) {};
		\draw[purple, thick, bend right, -latex] (C1) to (Cp);
		\draw[purple, thick, bend right, -latex] (Cp) to (C1);
		\node[circle, fill=black, inner sep=2pt] (Cpp) at (0,5.6) {};
		\draw[blue, thick, bend right, -latex] (Cp) to (Cpp);
		\draw[blue, thick, bend right, -latex] (Cpp) to (Cp);
		\draw[purple, thick, loop above, latex-] (Cpp) to (Cpp);
		
		\begin{scope}[shift={(0,-3)}]
			\foreach \x in {4,...,1}{
				\node[circle, fill=black, inner sep=2pt] (D\x) at ({cos(\x*90)},{sin(\x*90)}) {};
				\draw[blue, thick, bend right=35, -latex] ({cos((\x-1)*90)},{sin((\x-1)*90)}) to (D\x);}
		\end{scope}
		\draw[purple, thick, loop left, latex-] (D2) to (D2);
		\draw[purple, thick, loop right, latex-] (D4) to (D4);
		
		\node[circle, fill=black, inner sep=2pt] (Dp) at (0,-4.8) {};
		\draw[purple, thick, bend right, -latex] (D3) to (Dp);
		\draw[purple, thick, bend right, -latex] (Dp) to (D3);
		\node[circle, fill=black, inner sep=2pt] (Dpp) at (0,-5.6) {};
		\draw[blue, thick, bend right, -latex] (Dp) to (Dpp);
		\draw[blue, thick, bend right, -latex] (Dpp) to (Dp);
		\draw[purple, thick, loop below, latex-] (Dpp) to (Dpp);
		
		\begin{scope}[shift={(4,0)}]
			\foreach \x in {8,...,1}{
				\node[circle, fill=black, inner sep=2pt] (E\x) at ({2*cos(\x*45)},{2*sin(\x*45)}) {};
				\draw[blue, thick, bend right=16, -latex] ({2*cos((\x-1)*45)},{2*sin((\x-1)*45)}) to (E\x);}
		\end{scope}
		\draw[purple, thick, looseness=25, out=30, in=60, -latex] (E1) to (E1);
		\draw[purple, thick, looseness=25, out=120, in=150, -latex] (E3) to (E3);
		\draw[purple, thick, looseness=25, out=210, in=240, -latex] (E5) to (E5);
		\draw[purple, thick, looseness=25, out=300, in=330, -latex] (E7) to (E7);
		
		\node[circle, fill=black, inner sep=2pt] (Ep) at (4,3) {};
		\draw[purple, thick, bend right, -latex] (E2) to (Ep);
		\draw[purple, thick, bend right, -latex] (Ep) to (E2);
		\node[circle, fill=black, inner sep=2pt] (Epp) at (4,4) {};
		\draw[blue, thick, bend right, -latex] (Ep) to (Epp);
		\draw[blue, thick, bend right, -latex] (Epp) to (Ep);
		\draw[purple, thick, loop above, latex-] (Epp) to (Epp);
		
		\node[circle, fill=black, inner sep=2pt] (Em) at (4,-3) {};
		\draw[purple, thick, bend right, -latex] (E6) to (Em);
		\draw[purple, thick, bend right, -latex] (Em) to (E6);
		\node[circle, fill=black, inner sep=2pt] (Emm) at (4,-4) {};
		\draw[blue, thick, bend right, -latex] (Em) to (Emm);
		\draw[blue, thick, bend right, -latex] (Emm) to (Em);
		\draw[purple, thick, loop below, latex-] (Emm) to (Emm);
		
		\begin{scope}[shift={(7,0)}]
			\foreach \x in {4,...,1}{
				\node[circle, fill=black, inner sep=2pt] (F\x) at ({cos(\x*90)},{sin(\x*90)}) {};
				\draw[purple, thick, bend right=35, -latex] ({cos((\x-1)*90)},{sin((\x-1)*90)}) to (F\x);}
		\end{scope}
		\node[circle, fill=black, inner sep=2pt] (Fp) at (7,2) {};
		\draw[blue, thick, bend right, -latex] (F1) to (Fp);
		\draw[blue, thick, bend right, -latex] (Fp) to (F1);
		\draw[purple, thick, loop above, latex-] (Fp) to (Fp);
		\node[circle, fill=black, inner sep=2pt] (Fm) at (7,-2) {};
		\draw[blue, thick, bend right, -latex] (F3) to (Fm);
		\draw[blue, thick, bend right, -latex] (Fm) to (F3);
		\draw[purple, thick, loop below, latex-] (Fm) to (Fm);
		
		\begin{scope}[shift={(9,0)}]
			\foreach \x in {4,...,1}{
				\node[circle, fill=black, inner sep=2pt] (G\x) at ({cos(\x*90)},{sin(\x*90)}) {};
				\draw[blue, thick, bend right=35, -latex] ({cos((\x-1)*90)},{sin((\x-1)*90)}) to (G\x);}
		\end{scope}
		\draw[purple, thick, loop above, latex-] (G1) to (G1);
		\draw[purple, thick, loop below, latex-] (G3) to (G3);
		
		\node[circle, fill=black, inner sep=2pt] (H) at (11,0) {};
		\draw[purple, thick, bend right, -latex] (G4) to (H);
		\draw[purple, thick, bend right, -latex] (H) to (G4);
		
		\node[circle, fill=black, inner sep=2pt] (I) at (12,0) {};
		\draw[blue, thick, bend right, -latex] (H) to (I);
		\draw[blue, thick, bend right, -latex] (I) to (H);
		\draw[purple, thick, loop right, latex-] (I) to (I);

	\end{tikzpicture}
	

%% file: bibliography.bib
@inbook{ThompsonV_clones,
    author={Rose Berns-Zieve and Dana Fry and Johnny Gillings and Hannah Hoganson and Heather Mathews},
    title={Groups with Context-free Co-word Problem and Embeddings into Thompson’s Group $V$},
    booktitle={Topological Methods in Group Theory},
    publisher={Cambridge University Press},
    year={2018},
    pages={19–37},
    place={Cambridge},
    series={London Mathematical Society Lecture Note Series}}

@article{CFTR,
author = {D'Angeli, Daniele and Matucci, Francesco and Perego, Davide and Rodaro, Emanuele},
title = {Context-free graphs and their transition groups},
journal = {Transactions of the London Mathematical Society},
volume = {13},
number = {1},
pages = {e70034},
doi = {10.1112/tlm3.70034},
year = {2026}
}

@article{ET0L,
author = {Bishop, Alex and D'Angeli, Daniele and Matucci, Francesco and Nagnibeda, Tatiana and Perego, Davide and Rodaro, Emanuele},
title = {On the ET0L subgroup membership problem in bounded automata groups},
journal = {Journal of the London Mathematical Society},
volume = {113},
number = {4},
year={2026},
pages = {e70538}}

@article{cloning,
  title={Thompson groups for systems of groups, and their finiteness properties},
  author={Witzel, Stefan and Zaremsky, Matthew C.B.},
  journal={Groups, Geometry, and Dynamics},
  volume={12},
  number={1},
  pages={289--358},
  year={2018}}

@article{Clone_guide,
    title = {A user's guide to cloning systems},
    author = {Zaremsky, {Matthew C.B.}},
    journal = {Topology Proceedings},
    year = {2018},
    volume = {52},
    pages = {13--33}}

@phdthesis{cloning_explained,
    author = {Bennett, Daniel},
    title = {On plausible counterexamples to Lehnert's conjecture},
    school = {University of St Andrews} ,
    year =  {2018}}

@article{bennett2016demonstrative,
  title={A dynamical definition of f.g. virtually free groups},
  author={Bennett, Daniel and Bleak, Collin},
  journal={International Journal of Algebra and Computation},
  volume={26},
  number={01},
  pages={105--121},
  year={2016},
  publisher={World Scientific}}

@article{Muller_Schupp,
    title = {The theory of ends, pushdown automata, and second-order logic},
    journal = {Theoretical Computer Science},
    volume = {37},
    pages = {51-75},
    year = {1985},
    author = {David E. Muller and Paul E. Schupp}}

@article{manning2005geometry,
  author  = {Manning, Jason Fox},
  title   = {Geometry of pseudocharacters},
  journal = {Geometry \& Topology},
  volume  = {9},
  pages   = {1147--1185},
  year    = {2005}
}

@article{CF_pairs,
    title = {Context-free pairs of groups I: Context-free pairs and graphs},
    journal = {European Journal of Combinatorics},
    volume = {33},
    number = {7},
    pages = {1449-1466},
    year = {2012},
    author = {Ceccherini-Silberstein, Tullio and Woess, Wolfgang}}

@phdthesis{lehnert,
    title={Gruppen von quasi-Automorphismen},
    author={Lehnert, J{\"o}rg},
    school={Goethe Universität},
    adress={Frankfurt am Main},
    year={2008}}

@inbook{Farley_FFS,
    title={Local Similarity Groups with Context-free Co-word Problem}, 
    author={Daniel Farley},
    booktitle={Topological Methods in Group Theory},
    publisher={Cambridge University Press},
    year={2018},
    pages={67–91},
    place={Cambridge},
    series={London Mathematical Society Lecture Note Series}, 
    collection={London Mathematical Society Lecture Note Series}}

@inbook{Burillo_Cleary_Rover_2017,
    series={London Mathematical Society Lecture Note Series},
    title={Obstructions for subgroups of Thompson’s group $V$},
    booktitle={Geometric and Cohomological Group Theory},
    publisher={Cambridge University Press},
    author={Burillo, José and Cleary, Sean and Röver, Claas E.},
    editor={Kropholler, Peter H. and Leary, Ian J. and Martínez-Pérez, Conchita and Nucinkis, Brita E. A.},
    year={2017},
    pages={1–4},
    collection={London Mathematical Society Lecture Note Series}}

@article{cox2025finiteness,
  title={Finiteness properties of Subgroups of Houghton Groups of full Hirsch length},
  author={Cox, Charles and Kropholler, Peter and Martino, Armando},
  journal={arXiv preprint arXiv:2508.07816},
  year={2025}}

@ARTICLE{Jus,
    author = "Kate Juschenko and Benjamin Steinberg and Phillip Wesolek",
     title = "On elementary amenable bounded automata groups",
   journal = "Indiana Univ. Math. J.",
  fjournal = "Indiana University Mathematics Journal",
    volume = 70,
      year = 2021,
     issue = 6,
     pages = "2479--2526"}

@article{bondarenko2012growth,
  title={Growth of Schreier graphs of automaton groups},
  author={Bondarenko, Ievgen V.},
  journal={Mathematische Annalen},
  volume={354},
  number={2},
  pages={765--785},
  year={2012},
  publisher={Springer}}

@article{Branch_confined,
  title={A commutator lemma for confined subgroups and applications to groups acting on rooted trees},
  author={Le Boudec, Adrien and Matte Bon, Nicol{\'a}s},
  journal={Transactions of the American Mathematical Society},
  volume={376},
  number={10},
  pages={7187--7233},
  year={2023}}

@article{aroca2022some,
  title={Some embeddings between symmetric R. Thompson groups},
  author={Aroca, Julio and Bleak, Collin},
  journal={Proceedings of the Edinburgh Mathematical Society},
  volume={65},
  number={1},
  pages={1--18},
  year={2022},
  publisher={Cambridge University Press}}

@article{Bartholdi2005BranchG,
  title={Branch groups},
  author={Laurent Bartholdi and Rostislav Grigorchuk and Zoran Sunik},
  journal={Mathematical Notes},
  year={2005},
  volume={67},
  pages={718-723}}

@article{Basilica,
    author = {Grigorchuk, Rostislav I. and \.{Z}uk, Andrzej},
    title = {On a torsion-free weakly branch group defined by a three state automaton},
    journal = {International Journal of Algebra and Computation},
    volume = {12},
    number = {01n02},
    pages = {223-246},
    year = {2002}}

@misc{MatteBonTFG,
      title={Rigidity properties of full groups of pseudogroups over the Cantor set}, 
      author={Matte Bon, Nicolás},
      year={2018},
      note={\url{https://arxiv.org/abs/1801.10133}}}

@article{kerr2023tree,
  title={Tree approximation in quasi-trees},
  author={Kerr, Alice},
  journal={Groups, Geometry, and Dynamics},
  volume={17},
  number={4},
  pages={1193--1233},
  year={2023}}

@article{no_free_in_poly_activity,
    AUTHOR = {Sidki, Said},
     TITLE = {Finite automata of polynomial growth do not generate a free
              group},
   JOURNAL = {Geom. Dedicata},
  FJOURNAL = {Geometriae Dedicata},
    VOLUME = {108},
      YEAR = {2004},
     PAGES = {193--204}}

@article{no_free_in_contracting,
    AUTHOR = {Nekrashevych, Volodymyr},
     TITLE = {Free subgroups in groups acting on rooted trees},
   JOURNAL = {Groups Geom. Dyn.},
  FJOURNAL = {Groups, Geometry, and Dynamics},
    VOLUME = {4},
      YEAR = {2010},
    NUMBER = {4},
     PAGES = {847--862}}

@article {Francoeur,
    author = {Dominik Francoeur},
    title = {Normal subgroups of finitely generated branch groups are finitely generated},
     note = {Appendix to \emph{Commensurated subgroups and micro-supported actions} by P.-E. Caprace and A. Le Boudec},
   JOURNAL = {J. Eur. Math. Soc. (JEMS)},
  FJOURNAL = {Journal of the European Mathematical Society (JEMS)},
    VOLUME = {25},
      YEAR = {2023},
    NUMBER = {6},
     PAGES = {2289--2292}}

@article{Francoeur_Basilica,
    title = {On maximal subgroups of infinite index in branch and weakly branch groups},
    journal = {Journal of Algebra},
    volume = {560},
    pages = {818-851},
    year = {2020},
    author = {Dominik Francoeur}}

@article{Rover,
title = {Constructing Finitely Presented Simple Groups That Contain Grigorchuk Groups},
journal = {Journal of Algebra},
volume = {220},
number = {1},
pages = {284-313},
year = {1999},
issn = {0021-8693},
doi = {https://doi.org/10.1006/jabr.1999.7898},
url = {https://www.sciencedirect.com/science/article/pii/S0021869399978985},
author = {Claas E. Röver}}

@article {Chou,
    AUTHOR = {Chou, Ching},
     TITLE = {Elementary amenable groups},
   JOURNAL = {Illinois J. Math.},
  FJOURNAL = {Illinois Journal of Mathematics},
    VOLUME = {24},
      YEAR = {1980},
    NUMBER = {3},
     PAGES = {396--407}}

@misc{hyde2026,
    title={Action graphs, semiconjugacy, and non-embedding in Thompson's group $V$}, 
    author={James Hyde and Rachel Skipper and Matthew C. B. Zaremsky},
    year={2026},
    note={\url{https://arxiv.org/abs/2605.20564}}}

@article{EA_in_Thompson_F,
    AUTHOR = {Bleak, Collin and Brin, Matthew G. and Moore, Justin Tatch},
     TITLE = {Complexity among the finitely generated subgroups of
              {T}hompson's group},
   JOURNAL = {J. Comb. Algebra},
  FJOURNAL = {Journal of Combinatorial Algebra},
    VOLUME = {5},
      YEAR = {2021},
    NUMBER = {1},
     PAGES = {1--58}}

@article{Herbst1991,
	author = {Herbst, Thomas},
	title = {On a subclass of context-free groups},
	journal = {RAIRO - Theoretical Informatics and Applications - Informatique Th\'eorique et Applications},
	pages = {255--272},
	volume = {25},
	number = {3},
	year = {1991}}

@article{Asymptotic_minors_intro,
    title={Asymptotic dimension of minor-closed families and Assouad--Nagata dimension of surfaces},
    author={Bonamy, Marthe and Bousquet, Nicolas and Esperet, Louis and Groenland, Carla and Liu, Chun-Hung and Pirot, Fran{\c{c}}ois and Scott, Alexander},
    journal={Journal of the European Mathematical Society},
    volume={26},
    number={10},
    pages={3739--3791},
    year={2023}}

@article{Asymptotic_minors_survey,
    title={Graph minors and metric spaces},
    author={Georgakopoulos, Agelos and Papasoglu, Panos},
    journal={Combinatorica},
    volume={45},
    number={3},
    pages={33},
    year={2025},
    publisher={Springer}}

@misc{Asymptotic_minors_cacti,
      title={A coarse-geometry characterization of cacti}, 
      author={Koji Fujiwara and Panos Papasoglu},
      year={2023},
      note={\url{https://arxiv.org/abs/2305.08512}}}

@phdthesis{Skipper_thesis,
  title={On a generalization of the Hanoi Towers Group},
  author={Skipper, Rachel},
  year={2018},
  school={State University of New York at Binghamton}}

@article{Hanoi_Schreier,
  title={Asymptotic aspects of Schreier graphs and Hanoi Towers groups},
  author={Grigorchuk, Rostislav and {\v{S}}unik, Zoran},
  journal={Comptes Rendus Mathematique},
  volume={342},
  number={8},
  pages={545--550},
  year={2006},
  publisher={Elsevier}}

@article{Basilica_Schreier,
 author = {D'Angeli, Daniele and Donno, Alfredo and Matter, Michel and Nagnibeda, Tatiana},
 title = {Schreier graphs of the {Basilica} group.},
 journal = {Journal of Modern Dynamics},
 volume = {4},
 number = {1},
 pages = {167--205},
 year = {2010}}

@misc{Deaconu,
    title={Higman-Thompson groups from self-similar groupoid actions}, 
    author={Valentin Deaconu},
    year = {2021},
    note = {\url{https://arxiv.org/abs/2108.01178}}}

@misc{ESS_groups,
    title={Eventually Self-Similar Groups acting on Fractals}, 
    author={Davide Perego and Matteo Tarocchi},
    year = {2024},
    note = {\url{https://arxiv.org/abs/2412.04138}}}

@article{Hughes,
  title = {Local similarities and the Haagerup property (with an appendix by Daniel S. Farley)},
  volume = {3},
  ISSN = {1661-7215},
  DOI = {10.4171/ggd/58},
  number = {2},
  journal = {Groups,  Geometry,  and Dynamics},
  author = {Hughes,  Bruce},
  year = {2009},
  month = {6},
  pages = {299–315}}

@misc{Zaremsky_Q,
    title = {Some open problems},
    author = {Matthew C.B. Zaremsky},
    note = {\url{https://zaremsky.github.io/open_problems.pdf}}}

@misc{Period,
    title={Period growth and co-context-free groups}, 
    author={Alex Bishop and Corentin Bodart and Letizia Issini and Davide Perego},
    year={2026},
    note = {\url{https://arxiv.org/abs/2601.13058}}}

@article{MatteBon_intermediate,
    title = {Topological full groups of minimal subshifts with subgroups of intermediate growth},
    journal = {Journal of Modern Dynamics},
    volume = {9},
    number = {1},
    pages = {67-80},
    year = {2015},
    author = {Nicolás Matte  Bon}}

@inbook{GLN_Lysenok_subshift,
    place={Cambridge},
    series={London Mathematical Society Lecture Note Series},
    title={Schreier Graphs of Grigorchuk’s Group and a Subshift Associated to a Nonprimitive Substitution},
    booktitle={Groups, Graphs and Random Walks},
    publisher={Cambridge University Press},
    author={Grigorchuk, Rostislav and Lenz, Daniel and Nagnibeda, Tatiana}, editor={Ceccherini-Silberstein, Tullio and Salvatori, Maura and Sava-Huss, Ecaterina},
    year={2017},
    pages={250–299},
    collection={London Mathematical Society Lecture Note Series}}

@misc{Grigorchuk_ThueMorse_subshift,
    title={Thue-Morse sequence and groups of intermediate growth}, 
    author={Rostislav Grigorchuk and Yaroslav Vorobets},
    year={2026},
    notes = {\url{https://arxiv.org/abs/2605.30605}}}

@article{Matui,
    title = {Étale groupoids arising from products of shifts of finite type},
    journal = {Advances in Mathematics},
    volume = {303},
    pages = {502-548},
    year = {2016},
    author = {Hiroki Matui}}

@phdthesis{Tarrot,
    author = {Matteo Tarocchi},
    title = {Rearrangement Groups of Fractals: Structure and Conjugacy},
    school = {University of Pavia},
    year = {2024},
    notes = {\url{https://arxiv.org/abs/2412.02339}}}

@book{S2013,
  address = {Boston, MA},
  author = {Sipser, Michael},
  edition = {Third},
  isbn = {113318779X},
  publisher = {Course Technology},
  refid = {814441519},
  title = {Introduction to the Theory of Computation},
  year = 2013
}

@article{Rodaro,
author = {Rodaro, Emanuele},
title = {Generalizations of the Muller–Schupp theorem and tree-like inverse graphs},
journal = {Journal of the London Mathematical Society},
volume = {109},
number = {5},
pages = {e12903},
doi = {10.1112/jlms.12903},
year = {2024}
}

@article{juschenko2018nonEA,
  title={Non-elementary amenable subgroups of automata groups},
  author={Juschenko, Kate},
  journal={Journal of Topology and Analysis},
  volume={10},
  number={01},
  pages={35--45},
  year={2018},
  publisher={World Scientific}}

@misc{End_Schreier_branch,
      title={Commensurating actions and self-similar groups}, 
      author={Nicolás Matte Bon and Volodymyr Nekrashevych and Tianyi Zheng},
      year={2026},
      note = {\url{https://arxiv.org/abs/2607.13776}}}

@article{Tits_for_languages,
  title={Growth functions of some classes of languages},
  author={Vladimir I. Trofimov},
  journal={Cybernetics},
  year={1981},
  volume={17},
  pages={727-731}}

@article{BleakMatucciNeunhoffer2016,
  author  = {Bleak, Collin and Matucci, Francesco and Neunh{\"o}ffer, Max},
  title   = {Embeddings into {T}hompson's group {$V$} and {$\mathrm{coCF}$} groups},
  journal = {J. Lond. Math. Soc. (2)},
  volume  = {94},
  number  = {2},
  year    = {2016},
  pages   = {583--597}
}

@article{BooneHigman1974,
  author  = {Boone, William W. and Higman, Graham},
  title   = {An algebraic characterization of groups with soluble word problem},
  journal = {J. Austral. Math. Soc.},
  volume  = {18},
  year    = {1974},
  pages   = {41--53}
}

@article{Higman1961,
  author  = {Higman, Graham},
  title   = {Subgroups of finitely presented groups},
  journal = {Proc. Roy. Soc. Ser. A},
  volume  = {262},
  year    = {1961},
  pages   = {455--475}
}

@article{HoltReesRoverThomas2005,
  author  = {Holt, Derek F. and Rees, Sarah and R{\"o}ver, Claas E. and Thomas, Richard M.},
  title   = {Groups with context-free co-word problem},
  journal = {J. London Math. Soc. (2)},
  volume  = {71},
  number  = {3},
  year    = {2005},
  pages   = {643--657}
}

@article{LehnertSchweitzer2007,
  author  = {Lehnert, J{\"o}rg and Schweitzer, Pascal},
  title   = {The co-word problem for the {H}igman--{T}hompson group is context-free},
  journal = {Bull. Lond. Math. Soc.},
  volume  = {39},
  number  = {2},
  year    = {2007},
  pages   = {235--241}}

@misc{Jaspars,
    author = {Henry Jaspars},
    title = {On context free subgroups and R. Thompson’s group $V$},
    note = {in preparation}}

@ARTICLE{anisimov,
  AUTHOR = {Anatoly V. Anisimov},
  year=1971,
  ISSN = {0023-1274},
  journal = {Kibernetika (Kiev)},
  NUMBER = {4},
  PAGES = {18--24},
  TITLE = {The group languages},
  VOLUME = {7},
}

@article{BH_survey,
  title = {Progress around the Boone–Higman conjecture},
  journal = {EMS Surveys in Mathematical Sciences},
  publisher = {European Mathematical Society - EMS - Publishing House GmbH},
  author = {Belk,  James and Bleak,  Collin and Matucci,  Francesco and Zaremsky,  Matthew C. B.},
  year = {2025},
  month = {6}}

@article{MatuiFP,
url = {https://doi.org/10.1515/crelle-2013-0041},
title = {Topological full groups of one-sided shifts of finite type},
author = {Hiroki Matui},
pages = {35--84},
volume = {2015},
number = {705},
journal = {Journal für die reine und angewandte Mathematik (Crelles Journal)},
doi = {doi:10.1515/crelle-2013-0041},
year = {2015}}
